\documentclass[11pt,oneside, a4paper]{amsart}

\usepackage[authoryear]{natbib}
\usepackage[foot]{amsaddr}
\usepackage[shortlabels]{enumitem}
\usepackage{amssymb}
\usepackage{amsmath}
\usepackage{amsthm}
\usepackage{fancyhdr}
\usepackage{latexsym}
\usepackage{todonotes}
\usepackage{booktabs}
\usepackage{tikz}
\usepackage{tkz-graph}

\usepgflibrary{arrows}
\usetikzlibrary{calc,fadings,decorations.pathreplacing,intersections,positioning}
\usetikzlibrary{decorations.markings}

\usepackage{mathabx}
\usepackage[hidelinks]{hyperref}
\usepackage[linesnumbered,ruled]{algorithm2e}

\usepackage{mathbbol}
\DeclareSymbolFontAlphabet{\mathbb}{AMSb}
\DeclareSymbolFontAlphabet{\mathbbl}{bbold}
\newcommand{\one}{\mathbbl{1}}

\setlist[itemize]{leftmargin=*}
\calclayout

\DeclareMathOperator{\var}{Var}
\DeclareMathOperator{\cov}{Cov}

\renewcommand{\phi}{\varphi}
\newcommand{\bw}{\mathbf{w}}
\newcommand{\bW}{\mathbf{W}}
\newcommand{\bz}{\mathbf{z}}
\newcommand{\bZ}{\mathbf{Z}}
\newcommand{\bV}{\mathbb{V}}

\newcommand{\cZ}{\mathcal{Z}}
\newcommand{\ex}{\mathbb{E}}
\newcommand{\avar}{\mathbb{V}_{\! \Delta}}

\newcommand{\ind}{\,\perp \! \! \! \perp\,}
\newcommand{\indP}{\,{\perp \! \! \! \perp}_P\,}
\newcommand{\given}{\,|\,}
\newcommand{\real}{\mathbb{R}}

\newcommand{\ol}{\overline}
\newcommand{\pa}[1]{\left(#1\right)}

\newtheorem{thm}{Theorem}[section]
\newtheorem{prop}[thm]{Proposition}
\newtheorem{lem}[thm]{Lemma}
\newtheorem{cor}[thm]{Corollary}
\newtheorem{asm}{Assumption}[section]

\theoremstyle{definition}
\newtheorem{dfn}[thm]{Definition}
\newtheorem{exm}[thm]{Example}

\theoremstyle{remark}
\newtheorem{rmk}[thm]{Remark}

\newcommand\inner[2]{\langle #1, #2 \rangle}
\title[Adjustment for complex covariates]{Efficient adjustment for complex covariates: \\
Gaining efficiency with DOPE
}
\date{\today}

\author[A. M. Christgau]{Alexander Mangulad Christgau$^\sharp$}
\email{amc@math.ku.dk}

\author[A. R. Lundborg]{Anton Rask Lundborg$^\sharp$}
\email{arl@math.ku.dk}

\author[N. R. Hansen]{Niels Richard Hansen$^\sharp$}

\thanks{}
\email{niels.r.hansen@math.ku.dk}
\address{$^\sharp$Department of Mathematical Sciences, University of Copenhagen \newline
Universitetsparken 5, Copenhagen, 2100, Denmark}

\begin{document}

\begin{abstract}
  Covariate adjustment is a ubiquitous method used to estimate the average
  treatment effect (ATE) from observational data. Assuming a known graphical
  structure of the data generating model, recent results give graphical criteria
  for optimal adjustment, which enables efficient estimation of the ATE. However,
  graphical approaches are challenging for high-dimensional and complex data, and
  it is not straightforward to specify a meaningful graphical model of
  non-Euclidean data such as texts. We propose a new framework that
  accommodates adjustment for any subset of information expressed by the
  covariates, and we show that the information that is minimally sufficient
  for prediction of the outcome given the treatment is also most efficient 
  for adjustment.

  Based on our theoretical results, we propose the Debiased Outcome-adapted
  Propensity Estimator (DOPE) for efficient estimation of the ATE, and we provide
  asymptotic results for DOPE under general conditions. Compared to
  the augmented inverse propensity weighted (AIPW) estimator, DOPE can
  retain its efficiency even when the covariates are highly predictive of
  treatment. We illustrate this with a single-index model, and with an
  implementation of DOPE based on neural networks, we demonstrate its
  performance on simulated and real data. Our results show that DOPE
  provides an efficient and robust methodology for ATE estimation in various
  observational settings.
\end{abstract}

\maketitle

\section{Introduction}
Estimating the population average treatment effect (ATE) of a treatment on an
outcome variable is a fundamental statistical task. A naive approach is to
contrast the mean outcome of a treated population with the mean outcome of an
untreated population. Using observational data this is, however, generally a
flawed approach due to confounding. If the underlying confounding mechanisms are
captured by a set of pre-treatment covariates $\bW$, it is possible to adjust
for confounding by conditioning on $\bW$ in a certain manner. 
Given that multiple subsets of $\bW$ may be valid for this adjustment, it is
natural to ask if there is an `optimal adjustment subset' that enables the most
efficient estimation of the ATE.

Assuming a causal linear graphical model, \citet{henckel2022graphical}
established the existence of -- and gave graphical criteria for -- an
\emph{optimal adjustment set} for the OLS estimator.
\citet{rotnitzky2020efficient} extended the results of
\citet{henckel2022graphical}, by proving that the optimality was valid within
general causal graphical models and for all regular and asymptotically linear
estimators. 
Critically, this line of research assumes knowledge of the underlying graphical structure. 

To accommodate the assumption of \textit{no unmeasured confounding},
observational data is often collected with as many covariates as possible, which
means that $\bW$ can be high-dimensional. In such cases, assumptions of a known
graph are unrealistic, and graphical estimation methods are statistically
unreliable \citep{uhler2013geometry,shah2020hardness,chickering2004large}.
Furthermore, for non-Euclidean data such as images or texts, it is not clear how
to impose any graphical structure pertaining to causal relations. Nevertheless,
we can in these cases still imagine that the information that $\bW$ represents
can be separated into distinct components that affect treatment and outcome
directly, as illustrated in Figure~\ref{fig:semistructuredW}.
\begin{figure}
    \centering
      \begin{tikzpicture}[node distance=1.5cm, thick, roundnode/.style={circle, draw, inner sep=1pt,minimum size=7mm},
        squarenode/.style={rectangle, draw, inner sep=1pt, minimum size=7mm}]
        \node[roundnode] (T) at (0,-0.3) {$T$};
        \node (t) at (0,-0.78) {\scriptsize Treatment};
        \node[roundnode] (Y) at (2,-0.3) {$Y$};
        \node (y) at (2,-0.78) {\scriptsize Outcome};
        \node (Z) at (1,1.85) {$\bW$};
        \node[roundnode,dashed] (Z1) at (-0.5,1.4) {$\mathbf{C}_1$};
        \node[roundnode,dashed] (Z2) at (1,1.1) {$\mathbf{C}_2$};
        \node[roundnode,dashed] (Z3) at (2.5,1.4) {$\mathbf{C}_3$};
    
        \draw (1,1.35) ellipse (2.5cm and 0.9cm);
        \path [-latex,draw,thick, blue] (T) edge [bend left = 0] node {} (Y);
        \path [-latex,draw,thick] (Z1) edge [bend left = 0] node {} (T);
        \path [-latex,draw,thick] (Z2) edge [bend left = 0] node {} (T);
        \path [-latex,draw,thick] (Z2) edge [bend left = 0] node {} (Y);
        \path [-latex,draw,thick] (Z3) edge [bend left = 0] node {} (Y);
      \end{tikzpicture}
    \caption{The covariate $\bW$ can have a complex data structure, even if the information it represents is structured and can be categorized into
    components that influence treatment and outcome separately.}
    \label{fig:semistructuredW}
\end{figure}

In this paper, we formalize this idea and formulate a novel and general
adjustment theory with a focus on efficiency of the
estimation of the average treatment effect. Based on this adjustment theory, we
propose a general estimation procedure and analyze its asymptotic behavior. 

\subsection{Setup}
Throughout we consider a discrete treatment variable $T\in \mathbb{T}$, a
square-integrable outcome variable $Y\in \real$, and pre-treatment covariates
$\bW \in \mathbb{W}$. For now we only require that $\mathbb{T}$ is a finite set
and that $\mathbb{W}$ is a measurable space. The joint distribution of
$(T,\bW,Y)$ is denoted by $P$ and it is assumed to belong to a collection of
probability measures~$\mathcal{P}$. Formally, we define the triplet $(T, \bW,
Y)$ as a measurable function from a background measurable space $(\Omega,
\mathbb{F})$ equipped with a family of probability measures $(\mathbb{P}_P)_{P
\in \mathcal{P}}$ such that for every $P \in \mathcal{P}$, the distribution of
$(T,\bW,Y)$ is $P$ when the background measure is $\mathbb{P}_P$. If needed, we
denote probability statements and expectations by $\mathbb{P}_P$ and $\ex_P$,
respectively, but we will usually omit the subscript $P$ for ease of notation.

Our model-free target parameters of interest are of the form
\begin{align}\label{eq:adjustedmean}
    \mu_t
    \coloneq \ex[\ex[Y\given T=t,\bW]]
    = \ex\left[
        \frac{\one(T=t)Y}{\mathbb{P}(T=t \given \bW)}
    \right],
    \qquad t\in \mathbb{T}.
\end{align}
In words, these are treatment specific means of the outcome when adjusting for
the covariate $\bW$. To ensure that this quantity is well-defined, we assume the
following condition, commonly known as \textit{positivity}.
\begin{asm}[Positivity] \label{asm:positivity}
    It holds that $0<\mathbb{P}(T=t \given \bW)<1$ almost surely for each
    $t\in\mathbb{T}$.
\end{asm}

Under additional assumptions common in causal inference literature -- which
informally entail that $\bW$ captures all confounding -- the target parameter
$\mu_t$ has the interpretation as the \emph{interventional mean}, which is
expressed by $\ex[Y\given \mathrm{do}(T=t)]$ in \textit{do-notation} or by
$\ex[Y^{t}]$ using \textit{potential outcome} notation
\citep{peters2017elements,van2011targeted}. Under such causal assumptions, the
average treatment effect is identified, and when $\mathbb{T} = \{0,1\}$ it is
typically expressed as the contrast $\mu_1-\mu_0$. 
The theory in this paper is agnostic with regards to whether or not
$\mu_t$ has this causal interpretation, although it is the
primary motivation for considering $\mu_t$ as a target
parameter.

Given $n$ i.i.d. observations of $(T,\bW,Y)$, one may estimate
$\mu_t$ by estimating either of the equivalent expressions
for $\mu_t$ in Equation \eqref{eq:adjustedmean}. Within
parametric models, the outcome regression function 
$$
    g(t,\bw) = \ex[Y\given T=t,\bW=\bw]
$$ 
can typically be estimated with a $\sqrt{n}$-rate. In this case, the sample mean
of the estimated regression function yields a $\sqrt{n}$-consistent estimator of
$\mu_t$ under Donsker class conditions or sample splitting.
However, many contemporary datasets indicate that parametric model-based
regression methods get outperformed by nonparametric methods such as boosting
and neural networks \citep{bojer2021kaggle}.

Nonparametric estimators of the regression function typically converge at rates
slower than $\sqrt{n}$, and likewise for estimators of the propensity score 
$$
    m(t\mid \bw) = \mathbb{P}(T=t\given \bW=\bw).
$$
Even if both nonparametric estimators have rates slower than $\sqrt{n}$, it is
in some cases possible to estimate $\mu_t$ at a $\sqrt{n}$-rate by modeling both
$m$ and $g$, and then combining their estimates in a way that achieves `rate
double robustness' \citep{smucler2019unifying}. That is, an estimation error of
the same order as the product of the errors for $m$ and $g$. Two prominent
estimators that have this property are the Augmented Inverse Probability
Weighted estimator (AIPW) and the Targeted Minimum Loss-based Estimator (TMLE)
\citep{robins1995semiparametric,chernozhukov2018,van2011targeted}.

In what follows, the premise is that even with a $\sqrt{n}$-rate estimator of
$\mu_t$, it might still be intractable to model $m$ -- and
possibly also $g$ -- as a function of $\bW$ directly. This can happen for
primarily two reasons: (i) the sample space $\mathbb{W}$ is high-dimensional or
has a complex structure, or (ii) the covariate is highly predictive of
treatment, leading to unstable predictions of the inverse propensity score
$\mathbb{P}(T=t\given \bW)^{-1}$.

In either of these cases, which are not exclusive, we can try to manage these
difficulties by instead working with a \emph{representation} $\bZ = \phi(\bW)$,
given by a measurable mapping $\phi$ from $\mathbb{W}$ into a more tractable
space such as Euclidean space. In the first case above,
such a representation might be a pre-trained word embedding, e.g., the
celebrated BERT and its offsprings \citep{devlin2018bert}. The second case has
been well-studied in the special case where $\mathbb{W}=\real^k$ and where
$\mathcal{P}$ contains the distributions that are consistent with respect to a
fixed DAG (or CPDAG). We develop a general theory that subsumes both cases, and
we discuss how to represent the original covariates to efficiently estimate the
adjusted mean $\mu_t$.

\subsection{Relations to existing literature}
Various studies have explored the adjustment for complex data structures by
utilizing a (deep) representation of the covariates, as demonstrated in works
such as \citet{shi2019adapting,veitch2020adapting}. In a different research
direction, the \textit{Collaborative TMLE} \citep{van2010collaborative} has
emerged as a robust method for estimating average treatment effects by
collaboratively learning the outcome regression and propensity score,
particularly in scenarios where covariates are highly predictive of treatment
\citep{ju2019scalable}. Our overall estimation approach shares similarities with
the mentioned strategies; for instance, our proof-of-concept estimator in the
experimental section employs neural networks with shared layers. However, unlike
the cited works, it incorporates the concept of efficiently tuning the
representation specifically for predicting outcomes, rather than treatment.
Related to this idea is another interesting line of research, which builds upon
the \textit{outcome adapted lasso} proposed by \citet{shortreed2017outcome}.
Such works include
\citet{ju2020robust,benkeser2020nonparametric,greenewald2021high,balde2023reader}.
These works all share the common theme of proposing estimation procedures that
select covariates based on $L_1$-penalized regression onto the outcome, and then
subsequently estimate the propensity score based on the selected covariates
adapted to the outcome. The theory of this paper generalizes the particular
estimators proposed in the previous works, and also allows for other feature
selection methods than $L_1$-penalization. Moreover, our generalization of
(parts of) the efficient adjustment theory from \citet{rotnitzky2020efficient}
allows us to theoretically quantify the efficiency gains from these estimation
methods. Finally, our asymptotic theory considers a novel regime, which,
according to the simulations, seems more adequate for describing the finite
sample behavior than the asymptotic results of \citet{benkeser2020nonparametric}
and \citet{ju2020robust}.

Our general adjustment results in Section~\ref{sec:informationbounds}
draw on the vast literature on classical adjustment and confounder
selection, for example
\citet{rosenbaum1983central,hahn1998role,henckel2022graphical,rotnitzky2020efficient,guo2022confounder,perkovic2018complete,peters2017elements,forre2023mathematical}.
In particular, two of our results are direct extensions of results from
\citet{rotnitzky2020efficient,henckel2022graphical}.

\subsection{Organization of the paper}
In Section~\ref{sec:adjustment} we discuss generalizations of classical
adjustment concepts to abstract conditioning on information. In
Section~\ref{sec:informationbounds} we discuss information bounds in the
framework of Section \ref{sec:adjustment}. In Section~\ref{sec:estimation} we
propose a novel method, DOPE, for efficient estimation of adjusted
means, and we discuss the asymptotic behavior of the resulting estimator. In
Section~\ref{sec:experiments} we implement DOPE and demonstrate its
performance on synthetic and real data. The paper is concluded by a discussion
in Section~\ref{sec:discussion}.

\section{Generalized adjustment concepts}\label{sec:adjustment} In this section
we discuss generalizations of classical adjustment concepts. These
generalizations are motivated by the premise from the introduction: it might be
intractable to model the propensity score directly as a function of~$\bW$, so
instead we consider adjusting for a representation $\bZ=\phi(\bW)$. This is, in
theory, confined to statements about conditioning on $\bZ$, and is therefore
equivalent to adjusting for any representation of the form $\widetilde{\bZ} =
\psi \circ \phi (\bW)$, where $\psi$ is a bijective and bimeasurable mapping.
The equivalence class of such representations is characterized by the
$\sigma$-algebra generated by~$\bZ$, denoted by $\sigma(\bZ)$, which informally
describes the information contained in $\bZ$. In view of this, we define
adjustment with respect to sub-$\sigma$-algebras contained in
$\sigma(\bW)$. 

\begin{rmk}\label{rmk:conditional} 
Conditional expectations and probabilities are, unless otherwise indicated,
defined conditionally on $\sigma$-algebras as in
\citet{kolmogoroff1933grundbegriffe}, see also \citet[Ch.
8]{kallenberg2021foundations}. Equalities between conditional expectations are
understood to hold almost surely. When conditioning on the random variable $T$
and a $\sigma$-algebra $\cZ$ we write `$\given T,\cZ$' as a shorthand for
`$\given \sigma(T)\vee \cZ$'. Finally, we define conditioning on both the event
$(T=t)$ and on a $\sigma$-algebra $\cZ\subseteq \sigma(\bW)$ by
$
  \ex[Y \given T=t, \cZ] 
  \coloneq \frac{\ex[Y \one(T=t) \given \cZ]
      }{\mathbb{P}(T=t \given \cZ)},
$
which is well-defined under Assumption \ref{asm:positivity}.
\end{rmk}

\begin{dfn}\label{def:adjustment} 
A sub-$\sigma$-algebra $\cZ\subseteq \sigma(\bW)$ is called a 
\emph{description} of $\bW$. For each $t\in \mathbb{T}$ and $P\in\mathcal{P}$, 
and with $\cZ$ a description of $\bW$, we define
\begin{align*}
  \pi_t(\cZ; P) &\coloneq \mathbb{P}_P(T=t\given \cZ), \\
  b_t(\cZ; P) &\coloneq \ex_P[Y\given T=t, \cZ], \\
  \mu_t(\cZ;P) &\coloneq \ex_P[b_t(\cZ;P)].
\end{align*}
If a description $\cZ$ of $\bW$ is given as $\cZ = \sigma(\bZ)$ for a
representation $\bZ=\phi(\bW)$, we may write
$\mu_t(\bZ;P)$ instead of
$\mu_t(\sigma(\bZ);P)$ etc. 
We say that  
\begin{align*}
  \textnormal{$\cZ$ is $P$-valid if:} 
      &\qquad
      \mu_t(\cZ; P) = \mu_t(\bW; P), &\text{ for all } 
      t \in \mathbb{T}, \\
  \textnormal{$\cZ$ is $P$-OMS if:} 
      &\qquad
      b_t(\cZ; P) = b_t(\bW; P), &\text{ for all } 
      t \in \mathbb{T}, \\
  \textnormal{$\cZ$ is $P$-ODS if:} 
      &\qquad
      Y \indP \bW \given T,\cZ.
\end{align*}
Here OMS means \emph{Outcome Mean Sufficient} and ODS means 
\emph{Outcome Distribution Sufficient}.
If $\cZ$ is $P$-valid for all $P\in \mathcal{P}$, we say that it is $\mathcal{P}$-valid. We define $\mathcal{P}$-OMS and $\mathcal{P}$-ODS analogously. 
\end{dfn}
A few remarks are in order.
\begin{itemize}
  \item We have implicitly extended the shorthand notation described in Remark
  \ref{rmk:conditional} to the quantities in Definition \ref{def:adjustment}. 

  \item The quantity $\mu_t(\cZ; P)$ is deterministic,
  whereas $\pi_t(\cZ;P)$ and $b_t(\cZ;P)$ are $\cZ$-measurable real valued
  random variables. Thus, if $\cZ$ is generated by a representation $\bZ$, then
  by the Doob-Dynkin lemma \citep[Lemma 1.14]{kallenberg2021foundations} these
  random variables can be expressed as functions of $\bZ$. This fact does not
  play a role until the discussion of estimation in
  Section~\ref{sec:estimation}.

  \item There are examples of descriptions $\cZ \subseteq \sigma(\bW)$ that are
  not given by representations. Such descriptions might be of little practical
  importance, but our results do not require $\cZ$ to be given by a
  representation. We view the $\sigma$-algebraic framework as a convenient
  abstraction that generalizes equivalence classes of representations. 

  \item We have the following hierarchy of the properties in
  Definition~\ref{def:adjustment}:
  \begin{align*}
  P\textnormal{-ODS} 
  \Longrightarrow P \textnormal{-OMS} 
  \Longrightarrow P \textnormal{-valid},
  \end{align*}
  and the relations also hold if $P$ is replaced by $\mathcal{P}$.

  \item The $P$-ODS condition relates to existing concepts in statistics. The
  condition can be viewed as a $\sigma$-algebraic analogue of \textit{prognostic
  scores} \citep{hansen2008prognostic}, and it holds that any prognostic score
  generates a $P$-ODS description. Moreover, if a description $\cZ =
  \sigma(\bZ)$ is $P$-ODS, then its generator $\bZ$ is \textit{$c$-equivalent}
  to $\bW$ in the sense of \citet{pearl2009causality}, see in particular the
  claim following his Equation (11.8). In Remark~\ref{rmk:minimalsufficiency} we
  discuss the relations between $\mathcal{P}$-ODS descriptions and classical
  statistical sufficiency.
\end{itemize}

The notion of $\mathcal{P}$-valid descriptions can be viewed as a generalization
of valid adjustment sets, where subsets of variables are replaced with
sub-$\sigma$-algebras.

\begin{exm}[Comparison with adjustment sets in causal DAGs]
  Suppose $\bW \in \real^k$ and let $\mathcal{D}$ be a DAG on the nodes
  $\mathbf{V}=(T,\bW,Y)$. Let $\mathcal{P}=\mathcal{M}(\mathcal{D})$ be the
  collection of continuous distributions (on $\real^{k+2}$) that are Markovian
  with respect to $\mathcal{D}$ and with $\ex|Y|<\infty$.

  Any subset $\bZ\subseteq \bW$ is a representation of $\bW$ given by a
  coordinate projection, and the corresponding $\sigma$-algebra $\sigma(\bZ)$ is
  a description of $\bW$. In this framework, a subset $\bZ\subseteq \bW$ is
  called a \textit{valid adjustment set} for $(T,Y)$ if for all
  $P\in\mathcal{P}$, $t\in \mathbb{T}$, and $y\in \real$
  \begin{align*}
    \ex_P\left[
        \frac{\one(T=t)\one(Y\leq y)}{\mathbb{P}_P(T=t\given \mathrm{pa}_{\mathcal{D}}(T))}
        \right]
    = \ex_P[\mathbb{P}_P(Y\leq y\given T=t, \bZ)],
  \end{align*}
  where $\mathrm{pa}_{\mathcal{D}}$ denotes the parents of $T$ in
  $\mathcal{D}$, see for example Definition 2 in
  \citet{rotnitzky2020efficient}.
  It turns out that $\bZ$ is a valid adjustment set if and only if 
  \begin{align*}
    \ex_P\left[
        \frac{\one(T=t)Y}{\mathbb{P}_P(T=t\given \mathrm{pa}_{\mathcal{D}}(T))}
        \right]
    = \ex_P[\ex_P[Y\given T=t, \bZ]] 
    = \mu_t(\cZ; P)
  \end{align*}
  for all $P \in\mathcal{P}$ and $t\in \mathbb{T}$. This follows\footnote{We
  thank Leonard Henckel for pointing this out.} from results of
  \citet{perkovic2018complete}, which we discuss for completeness in
  Proposition~\ref{prop:adjset} in the supplement. Thus, if we assume that $\bW$
  is a valid adjustment set, then any subset $\bZ\subseteq \bW$ is a valid
  adjustment set if and only if the corresponding description $\sigma(\bZ)$ is
  $\mathcal{P}$-valid. 
\end{exm}

In general, $\mathcal{P}$-valid descriptions are not necessarily generated from
valid adjustment sets. This can happen if structural assumptions are imposed on
the conditional mean, which is illustrated in the following example.

\begin{exm}\label{exm:magnitude} 
  Suppose that the outcome regression is known to
  be invariant under rotations of $\bW \in \real^{k}$ such that for any $P\in
  \mathcal{P}$,
  \begin{align*}
    \ex_P[Y \given T,\bW] = \ex_P[Y \given T,\|\bW\|\,] \eqcolon g_P(T,\|\bW\|).
  \end{align*}
  Without further graphical assumptions, we cannot deduce that any proper subset of $\bW$ is a valid adjustment set. In contrast, the magnitude $\|\bW\|$ generates a $\mathcal{P}$-OMS -- and hence also $\mathcal{P}$-valid -- description of $\bW$ by definition.
  
  Suppose that there is a distribution $\tilde{P}\in
  \mathcal{P}$ for which $g_{\tilde{P}}(t,\cdot)$ is
  bijective. Then $\sigma(\|\bW\|)$ is also the smallest $\mathcal{P}$-OMS
  description up to $\mathbb{P}_{\tilde{P}}$-negligible
  sets: if $\cZ$ is another $\mathcal{P}$-OMS description, we see that $b_t(\cZ;
  \tilde{P}) = b_t(\bW;
  \tilde{P}) =
  g_{\tilde{P}}(t,\|\bW\|)$ almost surely. Hence 
  $$
      \sigma(\|\bW\|) =\sigma(g_{\tilde{P}}(t,\|\bW\|))  
      \subseteq \ol{\sigma(b_t(\cZ; \tilde{P}))} 
      \subseteq \ol{\cZ},
  $$
  where overline denotes the $\mathbb{P}_{\tilde{P}}$-completion of a
  $\sigma$-algebra, i.e., $\ol{\cZ}$ is the smallest
  $\sigma$-algebra containing $\cZ$ and all its
  $\mathbb{P}_{\tilde{P}}$-negligible sets.
\end{exm}
    

Even in the above example, where the regression function is known to depend on a
one-dimensional function of $\bW$, it is generally not possible to estimate the
regression function at a $\sqrt{n}$-rate without restrictive assumptions. Thus,
modeling of the propensity score is required in order to obtain a
$\sqrt{n}$-rate estimator of the adjusted mean. If $\bW$ is highly predictive of
treatment, then naively applying a doubly robust estimator (AIPW, TMLE) can be
statistically unstable due to large inverse propensity weights. Alternatively,
since $\sigma(\|\bW\|)$ is a $\mathcal{P}$-valid description, we could also base
our estimator on the \textit{pruned propensity} $\mathbb{P}(T=t \given
\|\bW\|)$. This approach should intuitively provide more stable weights, as we
expect $\|\bW\|$ to be less predictive of treatment. We proceed to analyze the
difference between the asymptotic efficiencies of the two approaches and we show
that, under reasonable conditions, the latter approach is never worse
asymptotically.


\section{Estimation efficiency for adjusted means}
\label{sec:informationbounds}
We now discuss estimation efficiency based on the
concepts introduced in Section~\ref{sec:adjustment}. To this end, let $\bZ =
\phi(\bW)$ be a representation of $\bW$ and define
\begin{equation}\label{eq:IF}
  \psi_{t}(\bZ; P)
      = b_t(\bZ; P) + \frac{\one(T=t)}{\pi_t(\bZ; P)}(Y-b_t(\bZ; P))
          -\mu_t(\bZ;P),
\end{equation}
as the \emph{influence function}. For specific choices of 
$\mathcal{P}$, $\psi_t(\bZ; P)$ \emph{is} the influence function
of the estimand $\mu_t(\bZ; P)$ in $\mathcal{P}$ 
(sometimes called the influence curve or canonical gradient), see
\citep{hines2022demystifying} and Section~\ref{sec:efficiencybounds} in 
the supplementary material. 
Using the formalism of Section \ref{sec:adjustment}, we define $\psi_{t}(\cZ;
P)$ analogously to \eqref{eq:IF} for any description $\cZ$ of $\bW$, and we denote
the variance of the influence function by 
\begin{align*}
    \bV_t(\cZ;P) 
    &\coloneq \var_P[\psi_{t}(\cZ; P)] = \ex_P[\psi_t(\cZ; P)^2].
\end{align*}

For a specific representation $\bZ$, and 
 provided that the propensity score $\pi_t(\bZ; P)$ and outcome
regression $b_t(\bZ; P)$ can be estimated with sufficiently fast rates, 
it is possible to construct a regular asymptotically linear (RAL) estimator 
$\widehat{\mu}_t(\mathbf{Z}; P)$ based on $n$ i.i.d. observations such that 
\[
\sqrt{n}(\widehat{\mu}_t(\mathbf{Z}; P) - \mu_t(\bZ; P)) - \frac{1}{\sqrt{n}} \sum_{i=1}^n \psi_t(\bZ_i; P) \overset{P}{\to} 0.
\]
Indeed, the AIPW estimator \citep{chernozhukov2018} and TMLE \citep{van2011targeted} 
are examples. Following, e.g., \cite{newey1994}, $\psi_t(\bZ; P)$ 
is the influence function of such estimators, which justifies our terminology. 
Importantly, the variance $\bV_t(\bZ;P)$ is the asymptotic variance of 
those estimators.

Now if $\sigma(\bZ)$ is a $\mathcal{P}$-valid description of $\bW$, $\mu_t
\coloneq \mu_t(\bW;P) = \mu_t(\bZ; P)$ for all $P$, and it is natural to ask if
we should estimate $\mu_t$ by adjusting for $\bW$ or the representation $\bZ$?
We investigate this question in terms of the asymptotic variances $\bV_t(\bZ ;
P)$ and $\bV_t(\bW ; P)$. That is, in terms of the asymptotic relative
efficiency of RAL estimators with influence functions $\psi_t(\bZ; P)$ and
$\psi_t(\bW; P)$, respectively. We discuss how our results are related to
asymptotic efficiency bounds in Section~\ref{sec:efficiencybounds} in the
supplementary material.

We formulate our results in terms of the more general contrast parameter
\begin{equation}\label{eq:contrasttarget}
  \Delta = \Delta(\bW ; P) = \sum_{t\in \mathbb{T}} c_t \mu_t(\bW; P),
\end{equation}
where $\mathbf{c}\coloneq (c_t)_{t\in \mathbb{T}}$ are fixed real-valued
coefficients. The prototypical example, when $\mathbb{T} = \{0,1\}$, is $\Delta
=\mu_1 - \mu_0$, which is the average
treatment effect under causal assumptions, cf. the discussion following
Assumption~\ref{asm:positivity}. Note that the family of $\Delta$-parameters
includes the adjusted mean $\mu_t$ as a special case.

To estimate $\Delta$, we consider estimators of the form
\begin{equation}\label{eq:hatDelta}
  \widehat \Delta (\cZ;P) =
      \sum_{t\in \mathbb{T}} c_t \widehat{\mu}_t(\cZ;P),
\end{equation}
where $\widehat{\mu}_t(\cZ;P)$ denotes a consistent RAL estimator of
$\mu_t(\cZ;P)$ with influence function $\psi_t(\cZ; P)$. Correspondingly, the
asymptotic variance for such an estimator is
\begin{equation}
  \avar(\cZ;P) \coloneq \var_P\Big(\sum_{t\in \mathbb{T}} c_t\psi_{t}(\cZ; P)\Big).
\end{equation}

It turns out that two central results by \citet{rotnitzky2020efficient},
specifically their Lemmas 4 and 5, can be generalized from covariate subsets to
descriptions. One conceptual difference is that there is 
no natural generalization of
\textit{precision variables} and \textit{overadjustment (instrumental)
variables} -- if $\cZ_1$ and $\cZ_2$ are descriptions,
there is no canonical way\footnote{equivalent conditions to the existence of an
independent complement are given in Proposition 4 in \citet{emery2001vershik}.}
to subtract $\cZ_2$ from $\cZ_1$ in a way that maintains their join $\cZ_1 \vee
\cZ_2 \coloneq \sigma(\cZ_1,\cZ_2)$. Apart from this technical detail, the
proofs translate more or less directly. The following lemma is a direct
extension of \citet[Lemma 4]{rotnitzky2020efficient}.

\begin{lem}[Deletion of overadjustment]\label{lem:overadj} 
  Fix a distribution $P\in \mathcal{P}$ and let $\cZ_1\subseteq \cZ_2$ be $\sigma$-algebras such that $Y\indP \cZ_2 \given T,\cZ_1$. Then it always 
  holds that
  $$
  \avar(\cZ_2;P) - \avar(\cZ_1;P)
  = \sum_{t\in\mathbb{T}} c_t^2 D_t(\cZ_1,\cZ_2;P) \geq 0,
  $$
  where for each $t\in \mathbb{T}$,
  \begin{align*}
      D_t&(\cZ_1,\cZ_2;P) \coloneq \bV_t(\cZ_2 ; P)-\bV_t(\cZ_1 ; P) \\
      &= \ex_P\,\Big[
          \pi_t(\cZ_1;P)
          \var(Y\given T=t,\cZ_1)
          \var\big(\pi_t(\cZ_2;P)^{-1}
              \big| \; T=t,\cZ_1 \big)\Big].
  \end{align*}
  Moreover, if $\cZ_2$ is a description of $\bW$ then $\cZ_1$ is $P$-valid if
  and only if $\cZ_2$ is $P$-valid.
\end{lem}
The lemma quantifies the efficiency lost in adjustment when including
information that is irrelevant for the outcome. 

We proceed to apply this lemma to the minimal information in $\bW$ that is
predictive of $Y$ conditionally on $T$. To define this information, we use the
regular conditional distribution function of $Y$ given $T=t,\bW=\bw$, which we
denote by
\begin{align*}
    F(y\given t,\bw; P) \coloneq \mathbb{P}_P (Y\leq y \given T=t,\bW=\bw), 
        \qquad y\in \real, t\in\mathbb{T}, \bw\in \mathbb{W}.
\end{align*}
See \citet[Sec. 8]{kallenberg2021foundations} for a rigorous treatment of
regular conditional distributions. We will in the following, by convention, take
\begin{equation} \label{eq:btw}
  b_t(\bW;P) = \int y\,\mathrm{d}F(y\given t, \bW; P)
\end{equation}
to be the version that is a measurable function of the regular conditional
distribution.

\begin{dfn}\label{def:OutcomeAlgebras}
  Define the $\sigma$-algebras
  \begin{align*}
      \mathcal{Q} &\coloneq \bigvee_{P\in\mathcal{P}} \mathcal{Q}_P,
          \qquad \mathcal{Q}_P \coloneq \sigma(F(y\given t,\bW; P) ; \, y\in \real, t\in \mathbb{T}), \\
      \mathcal{R} &\coloneq \bigvee_{P\in\mathcal{P}} \mathcal{R}_P,
          \qquad \mathcal{R}_P \coloneq \sigma(b_t(\bW; P) ; \, t\in \mathbb{T}).
  \end{align*}
\end{dfn}

\begin{figure}
    \centering
    \begin{tikzpicture}[node distance=1.5cm, thick, roundnode/.style={circle, draw, inner sep=1pt,minimum size=7mm},
      squarenode/.style={rectangle, draw, inner sep=1pt, minimum size=7mm}]
        \node[roundnode] (T) at (0,0) {$T$};
        \node (t) at (0,-0.5) {\scriptsize Treatment};
        \node[roundnode,dashed] (W) at (1.3,1.4) {$\mathcal{Q}$};
        \node (Z) at (0.55,1.4) {$\bW$};
        \node (z) at (0.5,1.15) {};
        \node[roundnode] (Y) at (2,0) {$Y$};
        \node (y) at (2,-0.5) {\scriptsize Outcome};
        
        \draw (1,1.4) ellipse (1cm and 0.5cm);
        
        \path [-latex,draw,thick] (W) edge [bend left = 0] node {} (T);
        \path [-latex,draw,thick] (z) edge [bend left = 0] node {} (T);
        \path [-latex,draw,thick] (W) edge [bend left = 0] node {} (Y);
        \path [-latex,draw,thick, blue] (T) edge [bend left = 0] node {} (Y);
    \end{tikzpicture}
    \caption{The $\sigma$-algebra $\mathcal{Q}$ given in
    Definition~\ref{def:OutcomeAlgebras} as a description of $\bW$, which may include 
    strictly less information than $\sigma(\bW)$ depending on $\mathcal{P}$.}
    \label{fig:COA}
\end{figure}
Note that $\mathcal{Q}_P$, $\mathcal{Q}$, $\mathcal{R}_P$ and $\mathcal{R}$ are
all descriptions of $\bW$, see Figure \ref{fig:COA} for a depiction of the
information contained in $\mathcal{Q}$. Note also that $\mathcal{R}_P \subseteq
\mathcal{Q}_P$ by the convention \eqref{eq:btw}.

In the following theorem we use $\ol{\cZ}^P$ to denote
the $\mathbb{P}_P$-completion of a $\sigma$-algebra $\cZ$, i.e.,
$\ol{\cZ}^P$ is the smallest $\sigma$-algebra containing $\cZ$ and all its
$\mathbb{P}_{P}$-negligible sets. We now state one of the main results of this
section. 
\begin{thm}\label{thm:COSefficiency} 
  Fix $P\in \mathcal{P}$. Then, under Assumption \ref{asm:positivity}, it holds
  that
  \begin{enumerate}
    \item[(i)] If $\cZ$ is a description of $\bW$ and $\mathcal{Q}_P\subseteq
    \ol{\cZ}^P$, then $\cZ$ is $P$-ODS. In particular, $\mathcal{Q}_P$ is a
    $P$-ODS description of $\bW$.
    \item[(ii)] If $\cZ$ is $P$-ODS description, then $\mathcal{Q}_P\subseteq
    \ol{\cZ}^P$.
    \item[(iii)] For any $P$-ODS description $\cZ$ it holds that
    \begin{equation}\label{eq:COApointwise}
        \avar(\cZ; P) - \avar(\mathcal{Q}_P; P) 
        = \sum_{t\in \mathbb{T}} 
            c_t^2 D_t(\mathcal{Q}_P,\cZ; P)
        \geq 0,
    \end{equation}
    where $D_t$ is given as in Lemma~\ref{lem:overadj}.
  \end{enumerate}
\end{thm}

Together parts (i) and (ii) state that a description of $\bW$ is $P$-ODS if and
only if its $\mathbb{P}_P$-completion contains $\mathcal{Q}_P$. Part (iii)
states that, under $P\in \mathcal{P}$, $\mathcal{Q}_P$ leads to the
smallest asymptotic variance among all
$P$-ODS descriptions.

\begin{cor}\label{cor:Qefficiency}
  Let $\cZ$ be a description of $\bW$. Then $\cZ$ is $\mathcal{P}$-ODS if and
  only if $\mathcal{Q}_P \subseteq \ol{\cZ}^P$ for all $P \in \mathcal{P}$.  
  A sufficient condition for $\cZ$ to be $\mathcal{P}$-ODS is that 
  $\mathcal{Q} \subseteq \ol{\cZ}^P$ for all $P \in \mathcal{P}$, in which case
  \begin{equation}\label{eq:COAuniform}
      \avar(\cZ; P) - \avar(\mathcal{Q}; P) 
      = \sum_{t\in \mathbb{T}} 
          c_t^2 D_t(\mathcal{Q},\cZ; P)\geq 0,
          \qquad P\in \mathcal{P}.
  \end{equation}
  In particular, $\mathcal{Q}$ is a $\mathcal{P}$-ODS description of $\bW$, and
  \eqref{eq:COAuniform} holds with $\cZ = \sigma(\bW)$.
\end{cor}
\begin{rmk}
  It is \textit{a priori} not obvious if $\mathcal{Q}$ is given by a
  representation, i.e., if  $\mathcal{Q}=\sigma(\phi(\bW))$ for some measurable
  mapping $\phi$. In Example \ref{exm:magnitude} it is, since the arguments can
  be reformulated to conclude that $\mathcal{Q} = \sigma(\|\bW\|)$.
\end{rmk}

Instead of working with the entire conditional distribution, it suffices to work
with the conditional mean when assuming, e.g., independent additive noise on the
outcome.
\begin{prop}\label{prop:COMSefficiency}
  For fixed $P\in \mathcal{P}$, then $\mathcal{Q}_P = \mathcal{R}_P$ if any of
  the following statements are true:
  \begin{itemize}
    \item $F(y\given t,\bW;P)$ is $\sigma(b_t(\bW;P))$-measurable for each
    $t\in \mathbb{T},y\in \real$.

    \item $Y$ is a binary outcome. 

    \item $Y$ has independent additive noise, i.e., $Y = b_T(\bW) +
        \varepsilon_Y$ with $\varepsilon_Y \indP T,\bW$.
  \end{itemize}
  If $\mathcal{Q}_P = \mathcal{R}_P$ holds, then \eqref{eq:COApointwise} also
  holds for any $P$-OMS description.
\end{prop}

\begin{rmk}
  When $\mathcal{Q}_P = \mathcal{R}_P$ for all $P \in \mathcal{P}$ we have
  $\mathcal{Q} = \mathcal{R}$. There is thus the same information in the
  $\sigma$-algebra $\mathcal{R}$ generated by the conditional means of the
  outcome as there is in the $\sigma$-algebra $\mathcal{Q}$ generated by the
  entire conditional distribution of the outcome. The three conditions in
  Proposition~\ref{prop:COMSefficiency} are sufficient but not necessary to
  ensure this.
\end{rmk}


We also have a result analogous to Lemma 4 of \citet{rotnitzky2020efficient}:

\begin{lem}[Supplementation with precision]\label{lem:underadj} Fix
  $P\in\mathcal{P}$ and let $\cZ_1\subseteq \cZ_2$ be descriptions of $\bW$ such
  that $T\indP \cZ_2 \given \cZ_1$. Then $\cZ_1$ is $P$-valid if and only if
  $\cZ_2$ is $P$-valid. Irrespectively, it always holds that
  \begin{align*}
      \avar(\cZ_1 ; P)-\avar(\cZ_2 ; P)
      = 
      \mathbf{c}^\top \var_P[\mathbf{R}(\cZ_1,\cZ_2;P)] \mathbf{c} \ge 0
  \end{align*}
  where $\mathbf{R}(\cZ_1,\cZ_2;P) \coloneq (R_t(\cZ_1,\cZ_2;P))_{t\in \mathbb{T}}$ with
  \begin{align*}
      R_t(\cZ_1,\cZ_2;P) 
          &\coloneq \pa{\frac{\one(T=t)}{\pi_t(\cZ_2;P)} - 1}
          \pa{b_t(\cZ_2;P) - b_t(\cZ_1;P)}.
  \end{align*}
  Writing $R_t = R_t(\cZ_1,\cZ_2;P)$, the components of the covariance matrix of
  $\mathbf{R}$ are given by
  \begin{align*}
      \var_P(R_t) 
          &= \ex_P \left[
              \pa{\frac{1}{\pi_t(\cZ_1;P)} - 1}
                  \var_P[b_t(\cZ_2;P)\given \cZ_1]\right], \\
      \cov_P(R_s,R_t) &= 
          - \ex_P[\cov_P(b_s(\cZ_2;P), b_t(\cZ_2;P)\given \cZ_1)].
  \end{align*}
\end{lem}

As a consequence, we obtain the well-known fact that the propensity score is a
valid adjustment if $\bW$ is, cf. Theorems 1--3 in \citet{rosenbaum1983central}.
\begin{cor}
    Let $\Pi_P=\sigma(\pi_t(\bW;P)\colon \, t\in \mathbb{T})$. If $\cZ$ is a
    description of $\bW$ containing $\Pi_P$, then $\cZ$ is $P$-valid and 
    $$
        \avar(\Pi_P ; P)-\avar(\cZ ; P)
        = \mathbf{c}^\top \var_P[\mathbf{R}(\mathcal{R}_P ,\cZ;P)] \mathbf{c} 
        \ge 0.
    $$
\end{cor}
The corollary asserts that while the information contained in the propensity
score is valid, it is asymptotically inefficient to adjust for in contrast to
all the information in $\bW$. This is in the same spirit as Theorem 2 by
\citet{hahn1998role}, which states that the efficiency bound $\avar(\bW; P)$
remains unchanged if the propensity is considered known. However, the
corollary also quantifies the difference of the asymptotic variances.

Corollary \ref{cor:Qefficiency} asserts that $\mathcal{Q}$ minimizes asymptotic
variance over all $\mathcal{P}$-ODS descriptions $\mathcal{Z}$ satisfying
$\mathcal{Q} \subseteq \ol{\mathcal{Z}}^P$, and Proposition
\ref{prop:COMSefficiency} asserts that in special cases, $\mathcal{Q}$ reduces
to $\mathcal{R}$. Since $\mathcal{R}$ is $\mathcal{P}$-OMS, hence
$\mathcal{P}$-valid, it is natural to ask if $\mathcal{R}$ is generally more
efficient than $\mathcal{Q}$. The following example shows that their asymptotic
variances may be incomparable uniformly over~$\mathcal{P}$. 

\begin{exm}\label{ex:symmetric} 
  Let $0<\delta<\frac12$ be fixed and let $\mathcal{P}$ be the collection of 
  data generating distributions that satisfy:
  \begin{itemize}
    \item $\bW \in [\delta,1-\delta]$ with a symmetric distribution, i.e.,
    $\bW\stackrel{\mathcal{D}}{=}1-\bW$.
    \item $T\in\{0,1\}$ with $\mathbb{P}(T = 1 \given \bW) = \bW$.
    \item $Y = T + g(|\bW-\frac{1}{2}|) + v(\bW)\varepsilon_Y$, where
    $\varepsilon_Y\ind{} (T,\bW)$, $\ex[\varepsilon_Y^2]<\infty$,
    $\ex[\varepsilon_Y] = 0$, and where $g\colon [0,\frac{1}{2}-\delta]\to
    \real$ and $v\colon [\delta,1-\delta]\to [0,\infty)$ are continuous
    functions.
  \end{itemize}
  Letting $\bZ = |\bW-\frac{1}{2}|$, it is easy to verify directly from
  Definition \ref{def:OutcomeAlgebras} that 
  $$
      \mathcal{Q} = \sigma(\bW) \neq \sigma(\bZ) = \mathcal{R}.
  $$
  It follows that $\bZ$ is $\mathcal{P}$-OMS but not $\mathcal{P}$-ODS. However,
  $\bZ$ is $\mathcal{P}_1$-ODS in the homoscedastic submodel $\mathcal{P}_1=
  \{P\in \mathcal{P}\given v\equiv 1\}$. In fact, it generates the
  $\sigma$-algebra $\mathcal{Q}$ within this submodel, i.e., $\sigma(\bZ) =
  \vee_{P\in \mathcal{P}_1}\mathcal{Q}_P$. Thus $\avar(\bZ; P) \leq \avar(\bW ;
  P)$ for all $P\in\mathcal{P}_1$. 

  We refer to the supplementary Section \ref{sec:symmetriccomputations} for
  complete calculations of the subsequent formulas.
  
  From symmetry it follows that $\pi_1(\bZ)=0.5$ and hence we conclude that
  $T\ind \bZ$. By Lemma \ref{lem:underadj}, it follows that $0$ (the trivial
  adjustment) is $\mathcal{P}$-valid, but with $\avar(\bZ; P) \leq \avar(0 ; P)$
  for all $P\in\mathcal{P}$.
  
  Alternatively, direct computation yields that
  \begin{align*}
      \bV_t(0)
      &= 2\var(g(\bZ)) + 2\, \ex[v(\bW)^2]\ex[\varepsilon_Y^2] \\
      \bV_t(\bZ)
      &= \var\left( g(\bZ)\right)  + 2\,\ex[v(\bW)^2] \ex[\varepsilon_Y^2] \\
      \bV_t(\bW) 
      &= \var\left( g(\bZ)\right) + \ex\left[v(\bW)^2/\bW\right] \ex[\varepsilon_Y^2].
  \end{align*}
  With $\Delta = \mu_t$, the first two equalities confirm
  that $\avar(\bZ; P) \leq \avar(0 ; P)$, $P\in\mathcal{P}$, and the last two
  yield that indeed $\avar(\bZ; P) \leq \avar(\bW ; P)$ for $P\in\mathcal{P}_1$
  by applying Jensen's inequality. In fact, these are strict inequalities
  whenever $g(\bZ)$ and $\varepsilon_Y$ are non-degenerate. 
  
  Finally, we show that it is possible for $\avar(\bZ; P) > \avar(\bW ; P)$ for
  $P\notin\mathcal{P}_1$. Let $\tilde{P}\in \mathcal{P}$
  be a data generating distribution with
  $\ex_{\tilde{P}}[\varepsilon_Y^2]>0$, $v(\bW)=\bW^2$,
  and with $\bW$ uniformly distributed on $[\delta,1-\delta]$. Then
  \begin{align*}
      2\,\ex_{\tilde{P}}[v(\bW)^2] 
          &= \,2\ex_{\tilde{P}}[\bW^{4}] 
          = 2\frac{(1-\delta)^{5}-\delta^{5}}{5(1-2\delta)} 
          \xrightarrow{\delta\to 0} \frac{2}{5} \\
      \ex_{\tilde{P}}\left[\frac{v(\bW)^2}{\bW}\right] 
          &= \ex_{\tilde{P}}[\bW^{3}] 
          = \frac{(1-\delta)^{4}-\delta^{4}}{4(1-2\delta)} 
          \xrightarrow{\delta\to 0} \frac{1}{4}.
  \end{align*}
  So for sufficiently small $\delta$, it holds that $\avar(\bZ;
  \tilde{P}) > \avar(\bW ;
  \tilde{P})$. The example can also be modified to work
  for other $\delta>0$ by taking $v$ to be a sufficiently large power of $\bW$.
\end{exm}

\begin{rmk}
  Following Theorem 1 of \citet{benkeser2020nonparametric}, it is stated that
  the asymptotic variance $\avar(\mathcal{R}_P; P)$ is generally smaller than
  $\avar(\bW; P)$. The example demonstrates that this requires some assumptions
  on the outcome distribution, such as the conditions in
  Proposition~\ref{prop:COMSefficiency}.
\end{rmk}
\begin{rmk}\label{rmk:minimalsufficiency} Suppose that
  $\mathcal{P}=\{f_\theta\cdot \nu \colon \theta \in \Theta\}$
  is a parametrized family of measures with densities
  $\{f_\theta\}_{\theta\in\Theta}$ with respect to a $\sigma$-finite measure
  $\nu$. Then informally, a \emph{sufficient
  sub-$\sigma$-algebra} is any subset of the observed information for
  which the remaining information is independent of $\theta\in \Theta$, see
  \citet{billingsley2017probability} for a formal definition. Superficially,
  this concept seems similar to that of a $\mathcal{P}$-ODS description. In
  contrast however, the latter is a subset of the covariate information
  $\sigma(\bW)$ rather than the observed information $\sigma(T,\bW,Y)$, and it
  concerns sufficiency for the outcome distribution rather than the entire data
  distribution. Moreover, the Rao-Blackwell theorem asserts that conditioning an
  estimator of $\theta$ on a sufficient sub-$\sigma$-algebra leads to an
  estimator that is never worse. Example~\ref{ex:symmetric} demonstrates that
  the situation is more delicate when considering statistical efficiency for
  adjustment.
\end{rmk}

\section{Estimation based on outcome-adapted representations}\label{sec:estimation}
In this section we develop a general estimation method based on the insights of
Section~\ref{sec:informationbounds}. For
simplicity, we present the special case $\Delta = \mu_t$, but the results can be
extended to general contrasts $\Delta$. Throughout the section we will be
working under a fixed distribution $P\in \mathcal{P}$, and it is hence omitted
from the notation on several occasions.

Given a $P$-OMS representation $\bZ = \phi(\bW)$, it holds that $\mu_t(\bZ;P) =
\mu_t$ and an estimator of $\mu_t$ based on $\bZ$ can, under regularity conditions, achieve an
asymptotic variance of $\bV_t(\bZ;P)$. Proposition~\ref{prop:COMSefficiency}
gives conditions under which it holds that $\bV_t(\bZ;P)\leq \bV_t(\bW; P)$.
This suggests that we can construct a better estimator by using the $P$-OMS
representation $\bZ$ rather than the full covariate $\bW$.

In general, such a representation is, of course, unknown, and we therefore
consider adjusting for an estimated representation $\widehat{\bZ} =
\widehat{\phi}(\bW)$. To this end, we henceforth use the convention that each of the
quantities $\pi_t,b_t,\mu_t$, and $\bV_t$ are defined conditionally on the
estimated representation $\widehat{\phi}$, e.g.,
$$
    \mu_t(\widehat{\bZ})
    = \ex[b_t(\widehat{\bZ})\given \widehat{\phi} \,]
    = \ex[\ex[Y\given T=t,\widehat{\bZ},\widehat{\phi}\,]\given \widehat{\phi}\,].
$$
For any estimator $\widehat{\mu}_{t}$ based on the estimated representation, the
estimation error can then be decomposed as 
\begin{equation}\label{eq:biasvardecomp}
    \widehat{\mu}_{t} - \mu_t
    = \underbrace{\widehat{\mu}_{t} - \mu_t(\widehat{\bZ})}_{\text{sampling error}}
        + \underbrace{\mu_t(\widehat{\bZ}) - \mu_t}_{
            \text{representation error}}
\end{equation}
While the previous section was primarily concerned with the asymptotic variance
of the sampling error, the representation error may lead to an adjustment bias
due to adjusting for $\widehat{\bZ}$ rather than $\bZ$ or $\bW$. We need to
consider both error terms, and we suggest an estimation procedure that
accommodates a trade-off between sampling and representation error.

\subsection{The DOPE procedure}
Our method augments the AIPW estimator, which we briefly survey. In Section
\ref{sec:adjustment} we introduced the propensity score, $\pi_t$, and the
outcome regression, $b_t$, as random variables. We now consider their function
counterparts obtained as regular conditional expectations:
\begin{align*} 
    m&\colon \mathbb{T}\times \mathbb{W} \to \real, 
        &m(t\given \bw) \coloneq \mathbb{P}(T=t\given \bW=\bw), \\
    g&\colon \mathbb{T}\times \mathbb{W} \to \real,
        &g(t, \bw)  \coloneq  \ex[Y\given T=t,\bW=\bw].
\end{align*}
Given estimates $(\widehat{m},\widehat{g})$ of the nuisance functions
$(m,g)$, the AIPW estimator of $\mu_t$ is given by
\begin{equation}\label{eq:AIPW}
    \widehat{\mu}_t^{\textsc{aipw}}(\widehat{m},\widehat{g}) \coloneq
        \mathbb{P}_n\Big[
            \widehat{g}(t,\bW) + \frac{\one(T=t)(Y-\widehat{g}(t,\bW))}{\widehat{m}(t\given\bW)}
        \Big],
\end{equation}
where $\mathbb{P}_n[\cdot]$ denotes the empirical mean over $n$ i.i.d.
observations from $P$. The AIPW estimator can achieve $\sqrt{n}$-consistency
under regularity conditions, of which the most restrictive usually involves
controlling the product error of the nuisance estimates, e.g., 
\begin{equation}\label{eq:productbias}
    n\cdot \mathbb{P}_n\big[(\widehat{m}(t\given \bW)-m(t\given \bW))^2\big]
        \cdot \mathbb{P}_n \big[(\widehat{g}(t, \bW)-g(t, \bW))^2 \big]
        \longrightarrow 0.
\end{equation}
If $(\widehat{m},\widehat{g})$ are estimated using the same data as
$\mathbb{P}_n$, then additional conditions such as Donsker class conditions are
typically required. Among others, \citet{chernozhukov2018} propose circumventing
Donsker class conditions by sample splitting techniques such as $K$-fold
cross-fitting.

The starting point for our procedure is to consider an outcome model that
factors through an intermediate representation:
\begin{align} \label{eq:factorizedmodel}
    g(t,\bw) =h(t, \phi(\bw)), \qquad t\in \mathbb{T}, \bw\in\mathbb{W},
\end{align}
where $\phi \colon \mathbb{W} \to \real^d$ is a measurable mapping, and where $h
\colon \mathbb{T}\times \real^d \to \real$ is an unknown function. With
$\mathcal{F} \subseteq \{ \phi \mid \phi \colon \mathbb{W} \to \real^d\}$
denoting a representation model, our outcome model estimator in
Algorithm~\ref{alg:generalalg} is of the form $\widehat{h}(t, \widehat{\phi})$ with
$\widehat{\phi} \in \mathcal{F}$. As above, we use the notation $\bZ = \phi(\bW)$
and $\widehat{\bZ} = \widehat{\phi}(\bW)$ for $\phi, \widehat{\phi} \in \mathcal{F}$. If a
particular covariate value $\bw\in \mathbb{W}$ is clear from the context, we
also use the implicit notation $\bz = \phi(\bw)$ and $\widehat{\bz} =
\widehat{\phi}(\bw)$.

\begin{exm}[Single-index model]\label{ex:singleindex} The \textit{(partial)
  single-index model} applies to $\bW \in 
  \real^k$ and assumes that \eqref{eq:factorizedmodel} holds with
  $\phi(\bw)=\bw^\top\theta$, where $\theta \in \Theta \subseteq
  \real^k$. In other words, it assumes that the outcome regression factors
  through the \textit{linear predictor} $\bZ = \bW^\top
  \theta$, that is,
  \begin{align} \label{eq:singleindex}
  Y = h(T, \bW^\top \theta) + \varepsilon_Y, 
        \qquad \ex[\varepsilon_Y \given T, \bW] = 0.
  \end{align}
  For each treatment $T=t$, the model extends the generalized linear model (GLM)
  by assuming that the (inverse) link function $h(t, \cdot)$ is
  unknown. 
    
  The \textit{semiparametric least squares} (SLS) estimator proposed by
  \citet{ichimura1993semiparametric} estimates $\theta$
  and $h$ by performing a two-step regression procedure.
  Alternative estimation procedures are given in
  \citet{powell1989semiparametric,delecroix2003efficient}. The parameter
  $\theta$ is generally only identifiable up to a scalar constant but this has 
  no impact when adjusting for $\bZ$ as $\sigma(\bZ) = \sigma(c\bZ)$ for any
  $c\in\real\setminus\{0\}$. 
\end{exm}

Given an estimator $\widehat{\phi} \in \mathcal{F}$, independent of
$(T, \bW, Y)$, we introduce:
\begin{align*} 
    f_{\widehat{\phi}}(t\given \bz) 
        &\coloneq \mathbb{P}(T=t \given \widehat{\phi}(\bW)=\bz,\widehat{\phi}\,), \\
    h_{\widehat{\phi}}(t, \bz) &\coloneq \ex[Y\given T=t, \widehat{\phi}(\bW)=\bz,\widehat{\phi}\,].
\end{align*}
In words, $f_{\widehat{\phi}}$ and $h_{\widehat{\phi}}$ are the theoretical propensity
score and outcome regression, respectively, when using the
estimated representation $\widehat{\bZ}=
\widehat{\phi}(\bW)$.
Our generic estimation procedure is described in Algorithm \ref{alg:generalalg},
where the estimator $\widehat{\phi}$ is computed as part of an outcome regression, and where
the propensity score is estimated based on $\widehat{\phi}$. For theoretical reasons, 
$\widehat{\phi}$ and $\widehat{h}_{\widehat{\phi}}$ are computed in Algorithm~\ref{alg:generalalg} 
by a two-step procedure using potentially different subsets of the data -- see 
the discussion below the algorithm. We
refer to the resulting estimator, denoted by
$\widehat{\mu}_t^{\textsc{DOPE}}$, as the \textit{Debiased Outcome-adapted
Propensity Estimator} (DOPE).

\begin{algorithm} \caption{Debiased Outcome-adapted Propensity Estimator} \label{alg:generalalg}
  \textbf{input}: observations $(T_i,\bW_i,Y_i)_{i\in[n]}$,
  index sets $\mathcal{I}_1,\mathcal{I}_2,\mathcal{I}_3 \subseteq [n]$\;
  \textbf{options}:
  representation model $\mathcal{F}$, and outcome and propensity score regression methods\;
  \Begin{
    regress outcomes $(Y_i)_{i\in \mathcal{I}_1}$ onto
    $(T_i,\bW_{i})_{i\in\mathcal{I}_1}$ to obtain $\widehat{\phi} \in \mathcal{F}$\;

    compute the estimated representations $\widehat{\bZ}_i = \widehat{\phi}(\bW_i)$ for 
    $i \in \mathcal{I}_2$;
  
    regress outcomes $(Y_i)_{i\in \mathcal{I}_2}$ onto
    $(T_i,\widehat{\bZ}_{i})_{i\in\mathcal{I}_2}$ to obtain
    $\widehat{h}_{\widehat{\phi}}$\;
    
    regress treatments $(T_i)_{i\in\mathcal{I}_2}$ onto
    $(\widehat{\bZ}_{i})_{i\in \mathcal{I}_2}$ to obtain
    $\widehat{f}_{\widehat{\phi}}$\;

    compute the estimated representations $\widehat{\bZ}_i = \widehat{\phi}(\bW_i)$ for 
    $i \in \mathcal{I}_3$;
    
    compute AIPW based on $(T_i,\widehat{\bZ}_i,Y_i)_{i\in\mathcal{I}_3}$ and
    nuisance estimates $(\widehat{f}_{\widehat{\phi}},
    \widehat{h}_{\widehat{\phi}})$:
    \begin{align*}
    \widehat{\mu}_t^{\textsc{DOPE}}
    =
        \frac{1}{|\mathcal{I}_3|}\sum_{i\in\mathcal{I}_3}\Big(
            \widehat{h}_{\widehat{\phi}}(t,\widehat{\bZ}_i) + 
            \frac{\one(T_i=t)(Y_i-\widehat{h}_{\widehat{\phi}}(t,\widehat{\bZ}_i))}{\widehat{f}_{\widehat{\phi}}(t\given\widehat{\bZ}_i)}
        \Big);
    \end{align*}
  }
  \Return{\textnormal{DOPE:} $\widehat{\mu}_t^{\textsc{DOPE}}$.}
\end{algorithm}

Algorithm \ref{alg:generalalg} is formulated such that $\mathcal{I}_1$,
$\mathcal{I}_2$, and $\mathcal{I}_3$ can be arbitrary subsets of $[n]$, and for
the asymptotic theory we assume that they are disjoint. However, in practical
applications it might be reasonable to take $\mathcal{I}_1 = \mathcal{I}_2$ or
to use the full sample for every estimation step, i.e., employing the algorithm
with $\mathcal{I}_1=\mathcal{I}_2=\mathcal{I}_3=[n]$. In those cases, we also
imagine that lines 4--6 are run simultaneously, cf. the SLS estimator in the
single-index model (Example \ref{ex:singleindex}). The supplementary
Section~\ref{sup:crossfitting} describes a full cross-fitting scheme based on
Algorithm~\ref{alg:generalalg}.

By defining 
\begin{align*}
    m_{\widehat{\phi}}(t, \bw) & = f_{\widehat{\phi}}(t \given \widehat{\phi}(\bw)), \\
    g_{\widehat{\phi}}(t, \bw) & = h_{\widehat{\phi}}(t, \widehat{\phi}(\bw))
\end{align*}
with corresponding estimators 
\begin{align*}
    \widehat{m}_{\widehat{\phi}}(t, \bw) & = \widehat{f}_{\widehat{\phi}}(t \given \widehat{\phi}(\bw)), \\
    \widehat{g}_{\widehat{\phi}}(t, \bw) & = \widehat{h}_{\widehat{\phi}}(t, \widehat{\phi}(\bw)),
\end{align*}
we can also write DOPE as 
\begin{align}
    \label{eq:DOPE}
   \widehat{\mu}_t^{\textsc{DOPE}}
    =
        \frac{1}{|\mathcal{I}_3|}\sum_{i\in\mathcal{I}_3}\Big(
            \widehat{g}_{\widehat{\phi}}(t,\bW_i) + 
            \frac{\one(T_i=t)(Y_i-\widehat{g}_{\widehat{\phi}}(t,\bW_i))}{\widehat{m}_{\widehat{\phi}}(t\given\bW_i)}
        \Big).
\end{align}
The representation of DOPE by \eqref{eq:DOPE} shows that it is a variant of the
general AIPW estimator given by \eqref{eq:AIPW} with $\widehat{m} =
\widehat{m}_{\widehat{\phi}}$ and $\widehat{g} = \widehat{g}_{\widehat{\phi}}$.
Thus DOPE forces both nuisance estimators to factor through the same
outcome-adapted representation. While Algorithm~\ref{alg:generalalg} most
accurately describes how DOPE is computed, \eqref{eq:DOPE} is more convenient
for stating and proving our theoretical results. 

\begin{rmk}\label{rmk:benkeser} 
\citet{benkeser2020nonparametric} use a similar idea as DOPE, but
their propensity factors through the final outcome regression function instead
of a general intermediate representation. That our general formulation of
Algorithm~\ref{alg:generalalg} contains their collaborative one-step estimator
as a special case is seen as follows. Suppose $T\in\{0,1\}$ is binary and
let the representation be given by
$$
    (z_1, z_2)  = \phi(\bw) = (g(0,\bw),g(1,\bw)).
$$ 
Then with $h(t, (z_1, z_2)) = (1 - t)z_1 + t z_2$ and any estimator
$\widehat{\phi}$ we have $\widehat{g}_{\widehat{\phi}}(t,\bw) = h(t,
\widehat{\phi}(\bw))$. In this case and with
$\mathcal{I}_1=\mathcal{I}_2=\mathcal{I}_3$, Algorithm~\ref{alg:generalalg}
yields the collaborative one-step estimator of \citet[App.
D]{benkeser2020nonparametric}. Accordingly, we refer to this special case as
DOPE-BCL (Benkeser, Cai and van der Laan).
\end{rmk}

\subsection{Asymptotics of DOPE}
We proceed to discuss the asymptotics of DOPE. For our theoretical analysis, we assume that the index sets in Algorithm \ref{alg:generalalg} are disjoint. That is, the theoretical analysis relies on sample splitting.
\begin{asm}\label{asm:samples} The observations $(T_i,\bW_i,Y_i)_{i\in[n]}$ used
    to compute $\widehat{\mu}_t^{\textsc{DOPE}}$ are i.i.d. with the same
    distribution as $(T,\bW,Y)$, and
    $[n]=\mathcal{I}_1\cup\mathcal{I}_2\cup\mathcal{I}_3$ is a partition such
    that $|\mathcal{I}_3| \to \infty$ as $n\to \infty$.
\end{asm}

In our simulations, employing sample splitting did not seem to enhance
performance, and hence we regard this as a theoretical convenience rather than a
practical necessity in all cases. Our results can likely also be established
under alternative assumptions that avoid sample splitting, in particular Donsker
class conditions.

\subsubsection{Estimation error conditionally on representation}
We now state some results on the asymptotic behavior of the DOPE estimator
using the estimated representation $\widehat{\bZ}$. We require the
following assumption.

\begin{asm}\label{asm:AIPWconv}
    For $(T,\bW,Y)\sim P$ satisfying the representation model~\eqref{eq:factorizedmodel}, it holds that:
    \begin{enumerate}
        \item[(i)] There exists $c>0$ such that
        $\max\{|\widehat{m}_{\widehat{\phi}}-\frac{1}{2}|,\sup_{\phi \in \mathcal{F}} |m_{\phi}-\frac{1}{2}|\}\leq
        \frac{1}{2}-c$ almost surely.
        \label{asm:strictoverlap}
        
        \item[(ii)] There exists $C>0$ such that $\ex[Y^2\given \bW, T] \leq C$.
        \label{asm:CVarbound}

        \item[(iii)] There exists $\delta>0$ such that
        $\ex\big[|Y|^{2+\delta}\big]<\infty$.
        \label{asm:varlowerbound}
        
        \item[(iv)] It holds that \(
                \mathcal{E}_{1,t}^{(n)} \coloneq \frac{1}{|\mathcal{I}_3|} \sum_{i\in \mathcal{I}_3} (\widehat{m}_{\widehat{\phi}}(t\given \bW_i) - m_{\widehat{\phi}}(t\given \bW_i))^2 \xrightarrow{P} 0
            \).
        \label{asm:propensityconsistency}
        
        \item[(v)] It holds that \(
                \mathcal{E}_{2,t}^{(n)} \coloneq\frac{1}{|\mathcal{I}_3|} \sum_{i\in \mathcal{I}_3} (\widehat{g}_{\widehat{\phi}}(t,\bW_i) - g_{\widehat{\phi}}(t,\bW_i))^2 \xrightarrow{P} 0
            \).
        \label{asm:ORconsistency}
        
        \item[(vi)] It holds that $|\mathcal{I}_3|\cdot\mathcal{E}_{1,t}^{(n)} \mathcal{E}_{2,t}^{(n)} 
                \xrightarrow{P} 0$.
        \label{asm:producterror}
    \end{enumerate}    
\end{asm}

Classical convergence results of the AIPW estimator are proven under
similar conditions, but with (vi) replaced by the stronger convergence in
\eqref{eq:productbias}. Condition (i) is a strict overlap condition for all
possible estimated representations and will certainly be satisfied if strict
overlap holds when conditioning on $\bW$. It is likely that one would only need
the strict overlap condition for representations selected with high probability
but we stick to the stricter condition here for simplicity. We establish
conditional asymptotic results under conditions on the conditional errors
$\mathcal{E}_{1,t}^{(n)}$ and $\mathcal{E}_{2,t}^{(n)}$. To the best of our
knowledge, the most similar results that we are aware of are those of
\citet{benkeser2020nonparametric}, and our proof techniques are most similar to
those of \citet{chernozhukov2018,lundborg2023perturbation}.
The following result states that DOPE is asymptotically Gaussian 
regardless of the estimated representation.

\begin{thm}[Asymptotics of DOPE] \label{thm:conditionalConvMain}
  Under Assumptions \ref{asm:samples} and \ref{asm:AIPWconv}, it holds that
  \begin{align*}
    \sqrt{|\mathcal{I}_3|} \cdot \bV_t(\widehat{\bZ})^{-1/2}\left(
   \widehat{\mu}_t^{\textsc{DOPE}} - \mu_t(\widehat{\bZ})
    \right) \xrightarrow{d} \mathrm{N}(0, 1).
  \end{align*}
    Furthermore, defining
    \begin{align}\label{eq:varestimator}
      \widehat{\mathcal{V}}_t 
      = \frac{1}{|\mathcal{I}_3|} \sum_{i\in \mathcal{I}_3} \widehat{u}_i^2
      - \Big(\frac{1}{|\mathcal{I}_3|} \sum_{i\in \mathcal{I}_3}
      \widehat{u}_{i}\Big)^2
    \end{align}
    where
    \begin{align*}
      \widehat{u}_i 
      = \widehat{g}_{\widehat{\phi}}(t,\bW_i) + \frac{\one(T=t)(Y_i-\widehat{g}_{\widehat{\phi}}(t,\bW_i))}{\widehat{m}_{\widehat{\phi}}(t\given \bW_i)},
    \end{align*}
    it holds that
    \begin{align*}
      \widehat{\mathcal{V}}_t 
      -\mathbb{V}_t(\widehat{\bZ})
      \xrightarrow{P} 0
    \end{align*}
    as $n\to \infty$.
\end{thm}

In other words, we can expect DOPE to have an asymptotic distribution, for the
estimated representation $\widehat{\phi}$, approximated by 
$$
  \widehat{\mu}_t^{\textsc{DOPE}} \given \widehat{\phi}
  \quad \stackrel{as.}{\sim} \quad 
  \mathrm{N}\Big(\mu_t(\widehat{\bZ}), \, 
  \frac{1}{|\mathcal{I}_3|}\mathbb{V}_t(\widehat{\bZ})\Big),
$$
and we can estimate $\mathbb{V}_t(\widehat{\bZ})$ by $\widehat{\mathcal{V}}_t$.
Note that if $|\mathcal{I}_3| = \lfloor n/3 \rfloor$, say, then the asymptotic
variance is
$\frac{3}{n}\mathbb{V}_t(\widehat{\bZ})$. Our
simulation study indicates that the asymptotic approximation may be valid in
some cases without the use of sample splitting, in which case
$|\mathcal{I}_3|=n$ and the asymptotic variance is
$\frac{1}{n}\mathbb{V}_t(\widehat{\bZ})$. A
direct implementation of the sample splitting procedure thus comes with an
efficiency cost. In the supplementary Section~\ref{sup:crossfitting} we discuss
how a cross-fitting procedure makes use of sample splitting
without an efficiency cost.

\subsubsection{Asymptotics of representation induced error}
We now turn to the discussion of the second term in \eqref{eq:biasvardecomp},
i.e., the difference between the adjusted mean for the estimated representation
$\widehat{\bZ}$ and the adjusted mean for the full covariate $\bW$. If the
representation belongs to a parametric class of functions and we use this for
estimation, under sufficient regularity, the delta method \citep[Thm.
3.8]{van2000asymptotic} describes the distribution of this error:

\begin{prop}\label{prop:deltamethod} Let $\mathcal{F} =
\{\phi_{\theta}\}_{\theta \in \Theta}$ with $\Theta \subseteq \mathbb{R}^p$.
Suppose there exists $\theta_P \in \Theta$ such that $(Y, \bW, Y) \sim P$
satisfies \eqref{eq:factorizedmodel} with $\phi \equiv \phi_{\theta_P} \in
\mathcal{F}$. Let $u\colon \Theta \to \real$ be the function given by $u(\theta)
= \mu_t(\phi_\theta(\bW);P)$ and assume that $u$ is differentiable in
$\theta_P$. Suppose that $\widehat{\theta}$ is an estimator of $\theta_P$ with rate
$r_n$ such that
    \begin{align*}
        r_n \cdot (\widehat{\theta} -\theta_P) 
        \xrightarrow{d} \mathrm{N}(0,\Sigma).
    \end{align*}
    If $\widehat{\phi} := \phi_{\widehat{\theta}}$, then
    \begin{align*}
        r_n \cdot (\mu_t(\widehat{\bZ}) - \mu_t)
        \xrightarrow{d} \mathrm{N}(0, \nabla u(\theta_P)^\top \Sigma \nabla u(\theta_P))
    \end{align*}
    as $n \to \infty$.
\end{prop}

The delta method requires that the adjusted mean,
$\mu_t(\phi_\theta(\bW);P)$,
is differentiable with respect to $\theta\in\Theta$. The theorem below showcases
that this is the case for the single-index model in
Example~\ref{ex:singleindex}.
\begin{thm}\label{thm:SIregularity} 
  
  Let $(T,\bW,Y)\sim P$ be given by the single-index model in
  Example~\ref{ex:singleindex} with $h_t(\cdot)\coloneq h(t,\cdot)\in
  C^1(\mathbb{R})$. Assume that $\bW$ has a distribution with density $p_{\bW}$
  with respect to Lebesgue measure on $\real^d$ and that $p_{\bW}$ is continuous
  almost everywhere with bounded support. Assume also that the propensity
  $m(t\mid \bw) = \mathbb{P}(T=t\given \bW=\bw)$ is continuous in $\bw$.
    
  Then $u\colon \real^d \to \real$, defined by $u(\theta) =
  \mu_t(\bW^\top \theta; P)$, is differentiable at
  $\theta=\theta_P$ with
  \begin{align*}
  \nabla u (\theta_P)=
  \ex_P\Big[
    h_t'(\bW^\top \theta_P)
  \Big(
  1
  -
  \frac{\mathbb{P}(T=t\given \bW)}{\mathbb{P}(T=t\given \bW^\top \theta_P)}
  \Big)\bW\Big].
  \end{align*}
\end{thm}
The theorem is stated with some restrictive assumptions that simplify the proof,
but these are likely not necessary. It should also be possible to extend this
result to more general models of the form \eqref{eq:factorizedmodel}, but we
leave such generalizations for future work. In fact, the proof technique of
Theorem~\ref{thm:SIregularity} has already found application in
\citet{gnecco2023boosted} in a different context.

The convergences implied in Theorem \ref{thm:conditionalConvMain} and
Proposition \ref{prop:deltamethod}, with $r_n=\sqrt{|\mathcal{I}_3|}$, suggest
that the MSE of DOPE is of order
\begin{align}\label{eq:asymptoticMSE}
    \ex[(\widehat{\mu}_t^{\textsc{DOPE}} - \mu_t)^2]
    &= \ex[(\widehat{\mu}_t^{\textsc{DOPE}} - \mu_t(\bZ_{\widehat{\theta}}))^2]
        + \ex[(\mu_t(\bZ_{\widehat{\theta}}) - \mu_t)^2] \nonumber \\
        &\qquad + 2\,\ex[(\widehat{\mu}_t^{\textsc{DOPE}} - \mu_t(\bZ_{\widehat{\theta}}))(\mu_t(\bZ_{\widehat{\theta}}) - \mu_t)] \nonumber \\
    &\approx \frac{1}{|\mathcal{I}_3|}\Big(\ex[\bV_t(\bZ_{\widehat{\theta}})] + 
        \nabla u(\theta_P)^\top \Sigma \nabla u(\theta_P)
        \Big)
        + o\left(|\mathcal{I}_3|^{-1}\right)
\end{align}
The informal approximation `$\approx$' can be turned into an equality by
establishing (or simply assuming) uniform integrability, which enables the
distributional convergences to be lifted to convergences of moments.

\begin{rmk}\label{rmk:confidenceinterval} 
  The expression in \eqref{eq:asymptoticMSE} suggests an approximate confidence
  interval of the form
  \begin{equation*}
      \widehat{\mu}_t^{\textsc{DOPE}} \pm \frac{z_{1-\alpha}}{\sqrt{|\mathcal{I}_3|}} (\widehat{\mathcal{V}}_t + \widehat{\nabla}^\top \widehat{\Sigma} \widehat{\nabla}),
  \end{equation*}
    where, $\widehat{\mathcal{V}}_t$ is as in \eqref{eq:varestimator}.
    $\widehat{\Sigma}$ is a consistent estimator of the asymptotic variance of
    $\sqrt{|\mathcal{I}_3|}(\widehat{\theta}-\theta_P)$, and where
    $\widehat{\nabla}$ is a consistent estimator of $\nabla_\theta
    \mu_t(\phi(\theta, \bW))\vert_{\theta=\theta_P}$.
    However, the requirement of constructing both $\widehat{\Sigma}$ and
    $\widehat{\nabla}$ adds further to the complexity of the methodology, so it
    might be preferable to rely on bootstrapping techniques in practice. In the
    simulation study we return to the question of inference and examine the
    coverage of the simpler interval that does not include the term
    $\widehat{\nabla}^\top \widehat{\Sigma} \widehat{\nabla}$, as well 
    as an interval based on bootstrap estimation of standard errors.
\end{rmk}

\section{Experiments}\label{sec:experiments} 
In this section, we showcase the performance of DOPE on simulated
and real data. All code is publicly available on
\href{https://github.com/ARLundborg/DOPE}{GitHub}\footnote{
\url{https://github.com/ARLundborg/DOPE}}.

\subsection{Simulation study}\label{sec:simulation} We present results of a
simulation study based on the single-index model from
Example~\ref{ex:singleindex}. We demonstrate the performance of DOPE, and compare with
various alternative estimators.

\subsubsection{Sampling scheme}
We simulated datasets consisting of $n$ i.i.d. copies of $(T,\bW,Y)$ sampled
according to the following scheme with $d=12$:
\begin{align}\label{eq:samplescheme}
    \bW = (W^1,\ldots,W^d)
        &\sim \mathrm{Unif}\bigl([0,1]^d\bigr) \nonumber \\
    T \given \bW
        &\sim \mathrm{Bern}( 0.01 + 0.98 \cdot \one(W^1 > 0.5)) \nonumber\\
    Y \given T,\bW, \beta_Y
        &\sim \mathrm{N}(h(T, \bW^\top \beta_Y ),1)
\end{align}
where $\beta_Y$ is sampled once for each dataset with
\[
    \beta_Y = (1,\tilde{\beta}_Y ), \qquad \tilde{\beta}_Y \sim \mathrm{N}\Bigl(0, \frac{1}{d-1}\mathbf{I}_{d-1}\Bigr),
\]
and where $n$, and $h$ are experimental parameters. The settings considered
were $n\in \{300,900,2700\}$ and with $h$ being one of
\begin{align}\label{eq:simulationlinks}
    \begin{array}{ll}
       h_{\text{lin}}(t,z) = t+3z,   & \quad
            h_{\text{square}} (t,z) = z^{1+t}, \\
       h_{\text{cbrt}} (t,z) = (2+t)z^{1/3},  & \quad 
            h_{\text{sin}} (t,z) = (3+t)\sin(\pi z).
    \end{array}
\end{align}
For each setting, $N=900$ datasets were simulated. We only present the results
for the linear and cube root link functions and relegate the results for the
other link functions to Section~\ref{sec:extra_simulations} in the supplementary
material. We experimented with simulations for other values of $d$ as well ($4$
and $36$) and found that the results were broadly similar to those reported here
for $d=12$. 

Note that while $\ex[T] = 0.01 + 0.98 \cdot \mathbb{P}(W^1 > 0.5) = 0.5$, the
propensity score $m(t\mid \bw) = 0.01 + 0.98 \cdot \one(w^1 > 0.5)$ takes the
rather extreme values $\{0.01,0.99\}$. Even though it technically satisfies
(strict) positivity, these extreme values of the propensity makes the adjustment
for $\bW$ a challenging task. For each dataset, the adjusted mean
$\mu_1$ (conditional on $\beta_Y$) was considered as the
target parameter, and the ground truth was numerically computed as the sample
mean of $10^7$ observations of $h(1,\bW^\top \beta_Y)$.

\subsubsection{Simulation estimators}\label{sec:simestimators} 
This section contains an overview of the estimators used in the simulation. For
a complete description see Section \ref{sup:simdetails} in the supplementary
material.

Two settings were considered for outcome regression (OR):
\begin{itemize}
    \item \textbf{Linear}: Ordinary Least Squares (OLS). 

    \item \textbf{Neural network}: A feedforward neural network with two hidden
    layers: a linear layer with one neuron, followed by a fully connected
    ReLU-layer with 100 neurons. The first layer is a linear bottleneck that
    enforces the single-index model, and we denote the weights by $\theta\in
    \real^d$. An illustration of the architecture can be found in the
    supplementary Section~\ref{sup:simdetails}. For a further discussion of
    leaning single-and multiple-index models with neural networks, see
    \citet{parkinson2023linear} and references therein.        
\end{itemize}
For propensity score estimation, logistic regression was used across all
settings. A ReLU-network with one hidden layer with 100 neurons was also
considered for estimation of the propensity score, but it did not enhance the
performance of the resulting ATE estimators in this setting. Random forests and
other methods were also initially used for both outcome regression and
propensity score estimation. They were subsequently excluded because they did
not seem to yield any noteworthy insights beyond what was observed for the
methods above.

For each outcome regression, two implementations were explored: a stratified
regression where $Y$ is regressed onto $\bW$ separately for each stratum
$T=0$ and $T=1$, and a joint regression where $Y$
is regressed onto $(T,\bW)$ simultaneously. In the case of joint regression, the
neural network architecture represents the regression function as a single index
model, given by $Y = h(\alpha T + \theta^\top \bW) + \varepsilon_Y$. This
representation differs from the description of the regression function specified
in the sample scheme~\eqref{eq:samplescheme}, which allows for a more complex
interaction between treatment and the index of the covariates. In this sense,
the joint regression is misspecified. We include both implementations
merely to highlight that both approaches are used in practice, but we
do not advocate for the stringent use of one or the other.

Based on these methods for nuisance estimation, we considered the following
estimators of the adjusted mean: 
\begin{itemize}
    \item The regression estimator 
    \(
      \widehat{\mu}_1^{\mathrm{reg}}
        \coloneq \mathbb{P}_n[\widehat{g}(1,\bW)] 
    \).
    \item The AIPW estimator given in Equation \eqref{eq:AIPW}.
    
    \item DOPE from Algorithm~\ref{alg:generalalg}, 
    where $\bZ_{\widehat{\theta}} = \bW^\top\widehat{\theta}$ and where    
    $\widehat{\theta}$ contains the weights of the first layer of the neural network
    designed for single-index regression. We refer to this estimator as
    DOPE-IDX.

    \item DOPE-BCL described in Remark~\ref{rmk:benkeser}, where the propensity
    score is based on the final outcome regression. See also \citet[App.
    D]{benkeser2020nonparametric}. Note that for the linear outcome regression 
    model DOPE-BCL coincides with DOPE-IDX.
\end{itemize}
We considered two versions of each DOPE estimator: one without sample splitting
and another using 3-fold cross-fitting.
For the latter, the final empirical mean is calculated based on the hold-out
fold, while the nuisance parameters are fitted using the remaining three folds.
For each fold $k=1, 2, 3$, this means employing
Algorithm~\ref{alg:generalalg} with $\mathcal{I}_1=\mathcal{I}_2 = [n]\setminus
J_k$ and $\mathcal{I}_3 = J_k$. The two versions showed comparable performance
for larger samples. In this section we only present the results for 
DOPE without sample splitting, which performed better overall in our
simulations. However, a comparison with the cross-fitted version can be found in
Section \ref{sup:simdetails} in the supplement.

Numerical results for the IPW estimator \( \widehat{\mu}_1^{\mathrm{IPW}}
\coloneq \mathbb{P}_n[\one(T=t)Y / \widehat{m}(1\given \bW)] \) were also
gathered. It performed poorly in our settings, and we have omitted the results
for the sake of clarity in the presentation.

\subsubsection{Empirical performance of estimators}
\begin{figure}
    \centering
    \includegraphics[width=\linewidth]{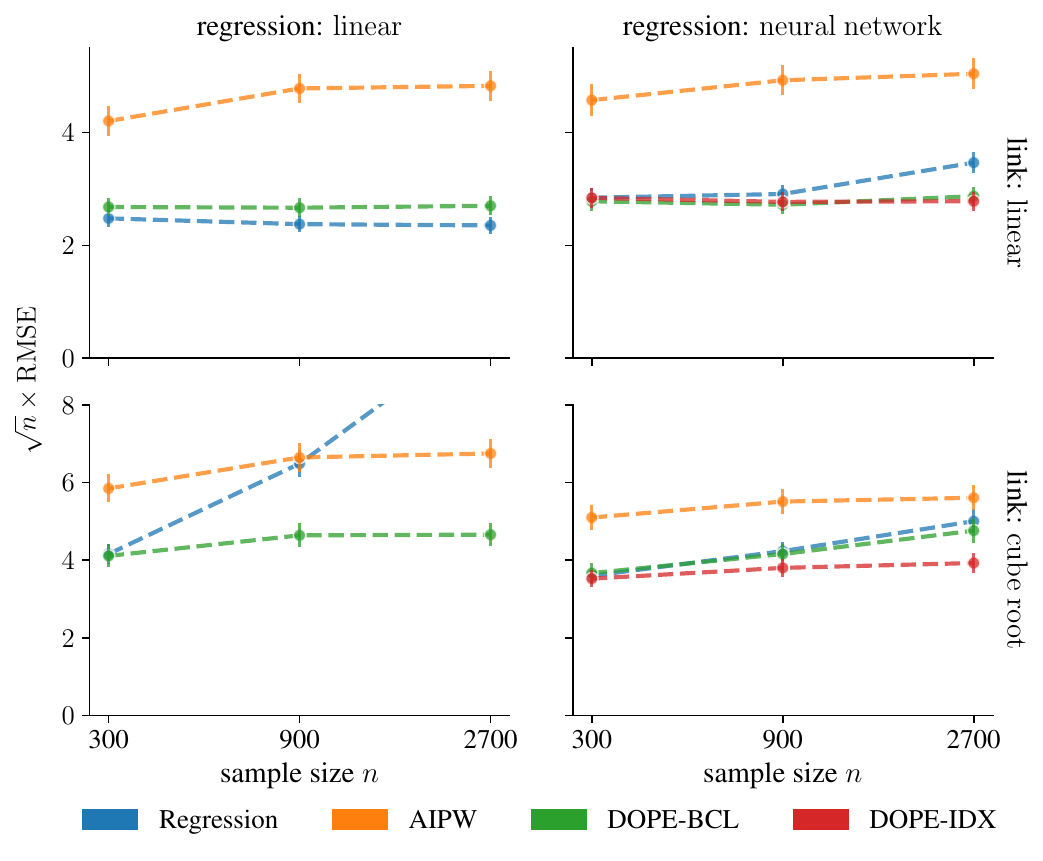}
    \caption{Root mean squared errors (RMSE) for various estimators of $\mu_1$
    plotted against sample size. Each data point is an average over 900
    datasets. The bars around each point correspond to asymptotic $95\%$
    confidence intervals based on the CLT. The dashed lines are only included as
    visual aids to make it easier to spot trends across sample sizes. For this
    plot, the outcome regression was fitted separately for each stratum
    $T=0$ and $T=1$.}
    \label{fig:RMSE-nocf}
\end{figure}

The results for stratified outcome regression are shown in Figure
\ref{fig:RMSE-nocf}. The rows of the plot correspond to different link functions
in \eqref{eq:simulationlinks} while each column denotes whether the outcome
regression is based on a linear regression or a single-index neural network. The
plot displays $\sqrt{n}$ times the root mean squared error as a function of the
sample size $n$. In the setting of the first row, where the link is linear, the
regression estimator $\widehat{\mu}_1^{\mathrm{reg}}$ with OLS performs best as
expected. The remaining estimators exhibit similar performance, except for AIPW,
which consistently performs poorly across all settings. This can be attributed
to extreme values of the propensity score, resulting in a large asymptotic
variance for AIPW. All estimators seem to maintain an approximate
$\sqrt{n}$-consistent RMSE for the linear link as anticipated.

For the cube root link in the second row (and the additional nonlinear links
included in the supplement), we observe that the OLS-based regression estimators
perform poorly and do not exhibit approximate $\sqrt{n}$-consistency.  For
neural network outcome regression, AIPW still performs the worst across all
nonlinear links. The regression estimator and DOPE estimators share similar
performance when used with the neural network, with DOPE-IDX being more accurate
overall. Since the neural network architecture is tailored to the single-index
model, it is not surprising that the outcome regression works well, and as a
result there is less need for debiasing. On the other hand, the debiasing
introduced in DOPE does not hurt the accuracy, and in fact, improves it in this
setting.

\begin{figure}
    \centering
    \includegraphics[width=\linewidth]{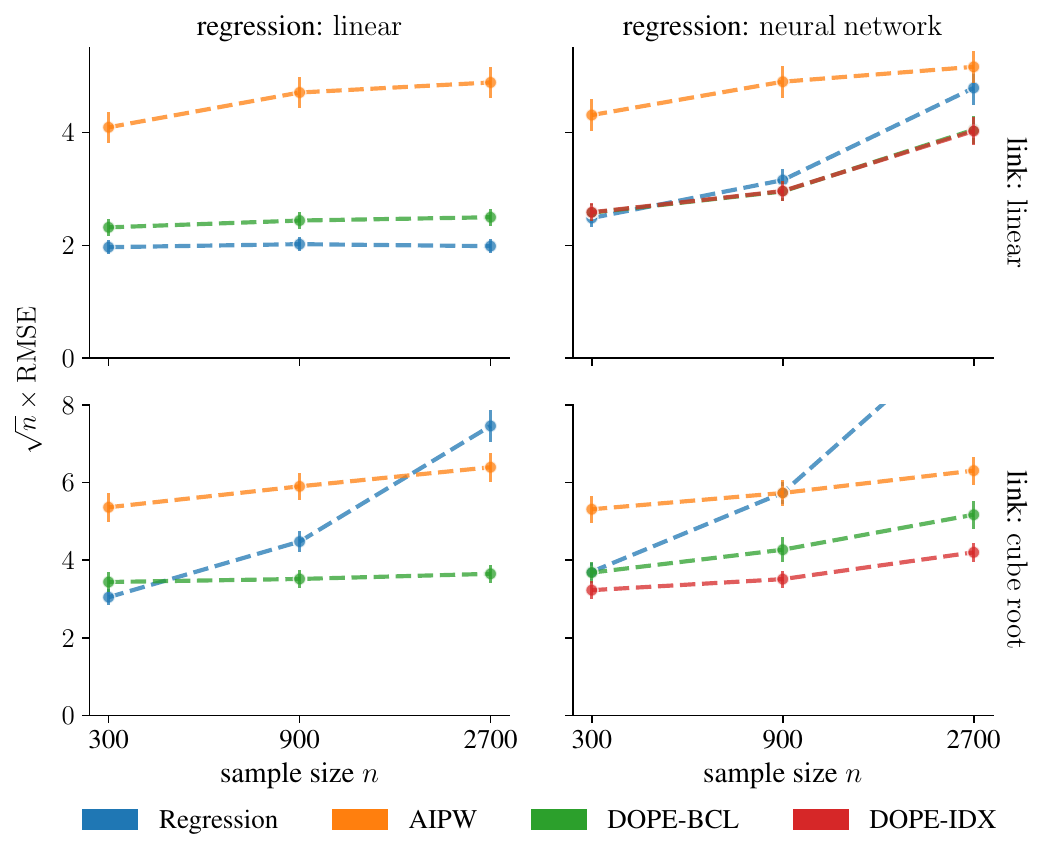}
    \caption{Root mean squared errors (RMSE) for various estimators of $\mu_1$
    plotted against sample size. Each data point is an average over 900
    datasets. The bars around each point correspond to asymptotic $95\%$
    confidence intervals based on the CLT. The dashed lines are only included as
    visual aids to make it easier to spot trends across sample sizes. For this
    plot, the outcome regression was fitted jointly onto $(T,\bW)$.}
    \label{fig:RMSE-joint}
\end{figure}

The results for joint regression of $Y$ on $(T,\bW)$ are shown in Figure
\ref{fig:RMSE-joint}. The results for the OLS-based estimators provide similar
insights as previously discussed, so we focus on the results for the neural
network based estimators. The jointly trained neural network is, in a sense,
misspecified for the single-index model (except for the linear link), as
discussed in Section~\ref{sec:simestimators}. Thus it is not surprising that
the regression estimator fails to adjust effectively for larger sample sizes.
What is somewhat unexpected, however, is that the precision of DOPE, especially
DOPE-IDX, does not appear to be compromised by the misspecified
outcome regression. We suspect that this could be attributed
to the joint regression producing more robust predictions for the rare treated
subjects with $W^1\leq 0.5$, for which $m(1\given \bW)=0.01$. The predictions
are more robust since the joint regression can leverage some of the information
from the many untreated subjects with $W^1\leq 0.5$ at the cost of introducing
systematic bias, which DOPE-IDX deals with in the debiasing step.
While this phenomenon is interesting, a thorough exploration of its exact
details, both numerically and theoretically, is a task we believe is better
suited for future research.

In summary, DOPE serves as a middle ground between the regression
estimator and the AIPW. It provides an additional safeguard against biased
outcome regression, all while avoiding the potential numerical instability
entailed by using standard inverse propensity weights.

\subsubsection{Inference}\label{sec:simulationinference} 

We now consider the question of inference on the adjusted mean using the
various estimators. We conduct a separate simulation study in the setting of the
previous section with the cube root link function, $n$ fixed at $2700$ and we
only consider joint estimation of the regression functions. In addition, we fix 
$\beta = (1, -2, 3, 0, \dots, 0) \in \mathbb{R}^{12}$ and repeat the experiment
$N=100$ times. We restrict attention to the estimators based on the neural
network.

We investigate confidence intervals for each of the methods discussed above of
the form $\widehat{\mu}_1 \pm z_{0.975} \mathrm{SE}$ where $z_{0.975}$ is the
$0.975$ quantile of the standard normal distribution and $\mathrm{SE}$ is an
estimate of the standard error of the estimator. One way to estimate the
standard error is by using the empirical asymptotic variance estimator, i.e.,
$\mathrm{SE} = \widehat{\mathcal{V}}_1^{1/2} n^{-1/2}$ where
$\widehat{\mathcal{V}}_1$ is defined in \eqref{eq:varestimator}. An alternative
approach is to use the bootstrap, where we resample the data with replacement
$200$ times and estimate $\mathrm{SE}$ as the empirical standard deviation of
the bootstrap estimates.

\begin{figure}
    \centering
    \includegraphics[width=\linewidth]{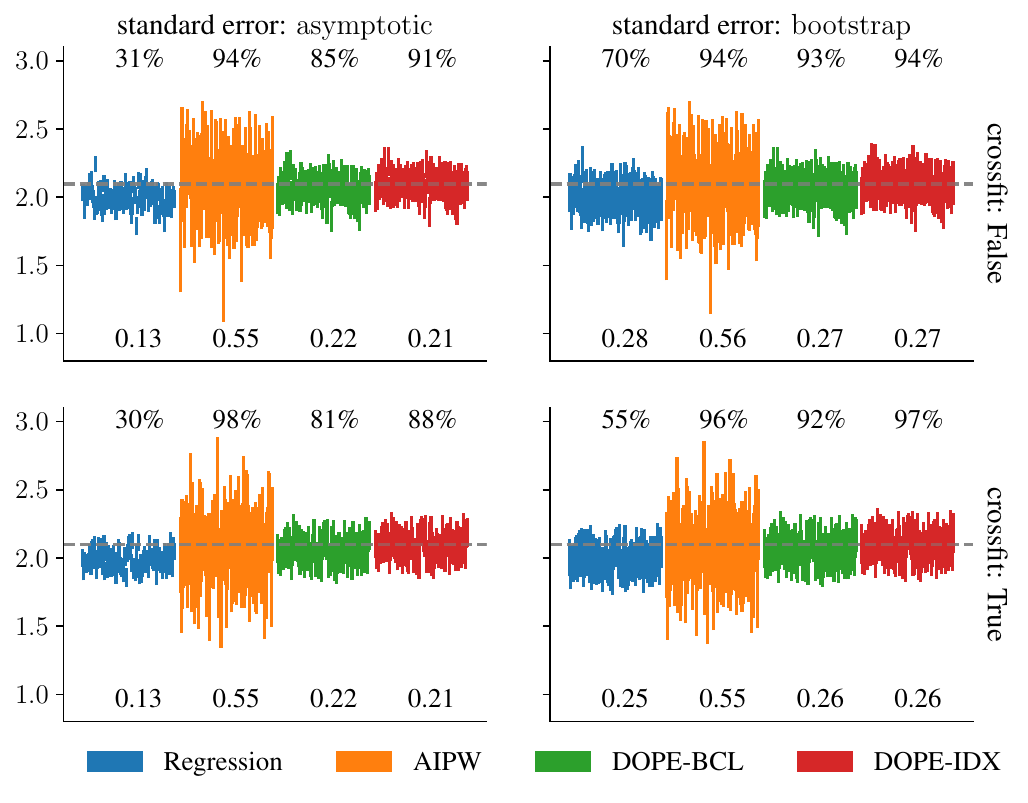}
    \caption{ $95\%$ nominal confidence intervals for the adjusted mean
    $\mu_1$ in the experiment described in
    Section~\ref{sec:simulationinference}. The text above each collection of
    intervals indicates the coverage rate out of the $100$ intervals. The text
    below indicates the median length of the intervals.} 
    \label{fig:Coverage-cf}
\end{figure}

The results from this experiment are shown in
Figure~\ref{fig:Coverage-cf}. We see that, except for the AIPW estimator, the
standard errors based on asymptotics yield anti-conservative confidence
intervals that do not achieve the nominal $95\%$ coverage. However, when using 
bootstrap standard errors, the coverage is close to the nominal level for all 
estimators. In every case, the AIPW intervals are much wider than the regression
and DOPE intervals. We also note that cross-fitting does in fact seem to reduce
the length of the intervals and when comparing mean squared
errors, as in the previous section, we do also see that cross-fitting improves
estimation in this setting. 

The results demonstrate an interesting statistical/computational
trade-off. We could use the computationally expensive bootstrap to obtain
precise coverage using any method, but we would pay a price in terms of width if
we use the AIPW intervals over the DOPE intervals. On the other hand, if we use
the intervals based on the asymptotic variance estimator, the DOPE intervals
would be slightly anti-conservative, but they would be much shorter than the
corresponding AIPW intervals. It is not surprising that the DOPE intervals based
on asymptotics do not achieve the nominal level as the asymptotic theory only
provides coverage for the adjusted mean given the \emph{estimated}
representation rather than the true representation. As mentioned in the theory
section, we could correct the DOPE intervals using the Delta method, but we see
the bootstrap as a more convenient way to achieve valid intervals.

\subsection{Application to NHANES data}
\label{sec:nhanes}
We consider the mortality dataset collected by the National Health and Nutrition
Examination Survey I Epidemiologic Followup Study \citep{cox1997plan},
henceforth referred to as the NHANES dataset. The dataset was initially
collected as in \citet{lundberg2020local}\footnote{ See
\url{https://github.com/suinleelab/treeexplainer-study} for their GitHub
repository. }. The dataset contains several baseline covariates, and our outcome
variable is the indicator of death at the end of study. To deal with missing
values, we considered both a complete mean imputation as in
\citet{lundberg2020local}, or a trimmed dataset where covariates with more than
$50\%$ of their values missing are dropped and the rest are mean imputed. The
latter approach reduces the number of covariates from $79$ to $65$. The final
results were similar for the two imputation methods, so we only report the
results for the mean imputed dataset here. In the supplementary
Section~\ref{sup:nhanesdetails} we show the results for the trimmed dataset.

The primary aim of this section is to evaluate the different estimation
methodologies on a realistic and real data example. For this purpose we consider
a treatment variable based on \textit{pulse pressure} and study its effect on
mortality. Pulse pressure is defined as the difference between systolic and
diastolic blood pressure (BP). High pulse pressure is not only used as an
indicator of disease but is also reported to increase the risk of cardiovascular
diseases \citep{franklin1999pulse}. We investigate the added effect of high
pulse pressure when adjusting for other baseline covariates, in particular
systolic BP and levels of white blood cells, hemoglobin, hematocrit, and
platelets. We do not adjust for diastolic BP, as it determines the pulse
pressure when combined with systolic BP, which would therefore lead to a
violation of positivity. 

\begin{figure}
    \centering
    \includegraphics[width=0.75\linewidth]{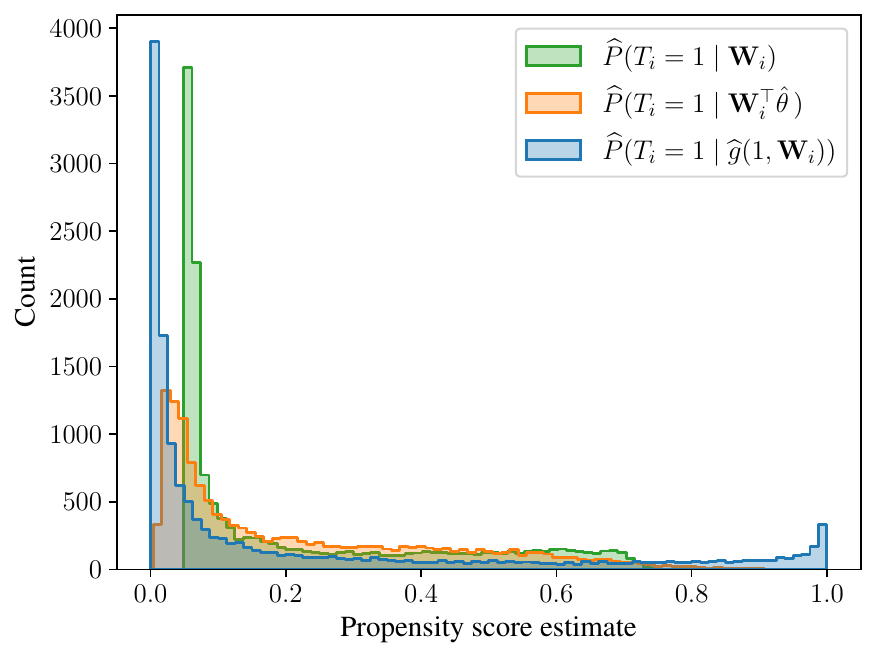}
    \caption{Distribution of estimated propensity scores based on the full
    covariate set $\bW$, the single-index $\bW^\top \widehat{\theta}$, and the
    outcome predictions $\widehat{g}(1,\bW)$, respectively.}
    \label{fig:NHANESpropensities}
\end{figure}

A pulse pressure of 40 mmHg is considered normal, and as the pressure increases
past 50 mmHg, the risk of cardiovascular diseases is reported to increase. Some
studies have used 60 mmHg as the threshold for high pulse pressure
\citep{homan2024physiology}, and thus we consider the following treatment
variable corresponding to high pulse pressure:
\begin{align*}
    T = \one\left( \textrm{pulse pressure} > 60 \textrm{ mmHg}\right).
\end{align*}

For the binary outcome regression we consider logistic regression and a variant
of the neural network from Section~\ref{sec:simestimators} with an additional
`sigmoid-activation' on the output layer. Based on 5-fold cross-validation,
logistic regression yields a slightly larger log-loss than the neural
network estimator.

Figure~\ref{fig:NHANESpropensities} shows the distribution of the estimated
propensity scores used in the AIPW, DOPE-IDX, and DOPE-BCL, based on neural
network regression. As expected, we see that the full covariate set yields
propensity scores that essentially violate positivity, with this effect being
less pronounced for DOPE-IDX and even less so for DOPE-BCL, where the scores are
comparatively less extreme.

\begin{table}
    \centering
    \begin{tabular}{lrrr}
\toprule
Estimator & Estimate & BS se & BS CI \\
\midrule
Regr. (Logistic) & $0.027$ & $0.009$ & $(0.012, 0.048)$ \\
DOPE-BCL (Logistic) & $0.024$ & $0.010$ & $(0.004, 0.040)$ \\
Naive contrast & $0.388$ & $0.010$ & $(0.369, 0.407)$ \\
DOPE-BCL (NN) & $0.010$ & $0.012$ & $(0.001, 0.043)$ \\
DOPE-IDX (NN) & $0.016$ & $0.013$ & $(0.001, 0.050)$ \\
Regr. (NN) & $0.022$ & $0.013$ & $(0.002, 0.049)$ \\
AIPW (Logistic) & $0.022$ & $0.016$ & $(-0.013, 0.051)$ \\
AIPW (NN) & $0.019$ & $0.017$ & $(-0.019, 0.048)$ \\
IPW (Logistic) & $-0.046$ & $0.027$ & $(-0.118, -0.010)$ \\
\bottomrule
\end{tabular}

    \vspace{\baselineskip}
    \caption{Adjusted effect estimates of unhealthy pulse pressure on mortality
      based on NHANES dataset. Estimators are sorted in increasing order after
      bootstrap (BS) standard error.}
    \label{tab:pulsepressure}
\end{table}
Table~\ref{tab:pulsepressure} shows the estimated treatment effects
$\widehat{\Delta} = \widehat{\mu}_1-\widehat{\mu}_0$ for various estimators
together with $95\%$ bootstrap confidence intervals. The estimators are sorted
according to bootstrap variance based on $B=1000$ bootstraps. Because we do not
know the true effect, we cannot directly decide which estimator is best. The
naive contrast provides a substantially larger estimate than all of the
adjustment estimators, which indicates that the added effect of high pulse
pressure on mortality when adjusting is different from the unadjusted effect.
The IPW estimator is the only estimator to yield a negative estimate of the
adjusted effect, and its bootstrap variance is also significantly larger than
the other estimators. Thus it is plausible, that the IPW estimator fails to
adjust appropriately. This is not surprising given extreme distribution of the
propensity weights, as shown in blue in Figure~\ref{fig:NHANESpropensities}.

The remaining estimators yield comparable estimates of the adjusted effect, with
the logistic regression based estimators having a marginally larger estimates
than their neural network counterparts. The (bootstrap) standard errors are
comparable for the DOPE and regression estimators, but the AIPW estimators have
slightly larger standard errors. As a result, the added effect of
high pulse pressure on mortality cannot be considered statistically significant
for the AIPW estimators, whereas it can for the other estimators. 

In summary, the results indicate that the choice treatment effect estimator
impacts the final estimate and confidence interval. While it is uncertain which
estimator is superior for this particular application, the DOPE estimators seem
to offer a reasonable and stable estimate. 

\section{Discussion}\label{sec:discussion} In this paper, we address the
challenges posed by complex data with unknown underlying structure, an
increasingly common scenario in observational studies. Specifically, we have
formulated a refined and general formalism for studying efficiency of covariate
adjustment. This formalism extends and builds upon the efficiency principles
derived from causal graphical models, in particular results provided by
\citet{rotnitzky2020efficient}. Our framework led to the identification of a
description of the covariate information that gives the theoretically most
efficient estimator of the adjusted mean. This efficient description is similar
to, but more general than, the graphically optimal adjustment set, and as
demonstrated by the single index model, the description can contain strictly
less information than the optimal adjustment set.

Based on our theoretical results, we introduced DOPE for more efficient
estimation of the ATE. DOPE attempts to learn a good description from data,
directed by our theory, that will improve upon efficiency. If we had oracle
knowledge about a graphically optimal adjustment set, DOPE might very well be
outperformed by an estimator that uses this knowledge, but this is not a good
comparison. DOPE does not use oracle knowledge but will attempt to learn enough
about the graphical structure to adjust efficiently. By taking the
representation model in DOPE to be the finite set of coordinate projections,
DOPE would be directly attempting to learn the optimal adjustment set. Though it
is beyond the scope of this paper, we believe that it would be very interesting
to explore how DOPE performs in a graphical modeling context, in particular if a
flexible representation model could improve on the efficiency compared to only
learning the optimal adjustment set.

In addition, we emphasize the following possible research directions:

\textbf{Extension to other targets and causal effects}: Our focus has
predominantly been on the adjusted mean $\mu_t =
\ex[\ex[Y\given T=t,\bW]]$. However, the extension of our methodologies to
continuous treatments, instrumental variables, or other causal effects, such as
the average treatment effect among the treated, is an area that warrants
in-depth exploration. We suspect that many of the fundamental ideas can be
modified to prove analogous results for such target parameters.  It would
also be interesting to consider conditional average treatment effects although
this is likely more difficult due to the non-existence of influence functions
for such functional targets.

\textbf{High-dimensional covariates}: While our theoretical results are
derived under fixed-dimensional covariates, extending these results to
high-dimensional settings where the number of covariates grows with the sample
size may be of practical interest. There are certain settings where using a
low-dimensional representation is already common practice, for example by using
OLS on the components selected by a LASSO regression \citep{lee2016exact}. Our
framework could help expand such practices to more general settings.

\textbf{Beyond neural networks}: While our examples and simulation study has
applied DOPE with neural networks, its compatibility with other regression
methods such as kernel regression and gradient boosting opens new avenues for
alternative adjustment estimators. Investigating such estimators can provide
practical insights and broaden the applicability of our approach. The Outcome
Highly Adapted Lasso \citep{ju2020robust} is a similar method in this direction
that leverages the Highly Adaptive Lasso \citep{benkeser2016highly} for robust
treatment effect estimation.

\textbf{Causal representation learning integration}: Another possible direction
for future research is the integration of our efficiency analysis into the
broader concept of causal representation learning \citep{scholkopf2021toward}.
This line of research, which is not concerned with specific downstream tasks,
could potentially benefit from some of the insights of our efficiency analysis
and the DOPE framework.

\textbf{Implications of sample splitting}: Our current asymptotic analysis
relies on the implementation of three distinct sample splits. Yet, our
simulation results suggest that sample splitting might not be necessary. Using
the notation of Algorithm~\ref{alg:generalalg}, we hypothesize that the
independence assumption $\mathcal{I}_3 \cap (\mathcal{I}_1 \cup \mathcal{I}_2)=
\emptyset$ could be replaced with Donsker class conditions. However, the
theoretical justifications for setting $\mathcal{I}_1$ identical to
$\mathcal{I}_2$ are not as evident, thus meriting additional exploration into
this potential simplification.

\textbf{Valid and efficient confidence intervals}: Confidence
intervals based on the conditional asymptotic standard errors  
may not provide adequate coverage of the unconditional adjusted mean, as shown
in the simulation study in Section~\ref{sec:simulationinference}. We
demonstrated, however, that bootstrap estimation of standard errors could result
in adequate coverage for DOPE. As indicated in
Remark~\ref{rmk:confidenceinterval}, it would be interesting to explore the
construction of valid confidence intervals in greater depths. This includes: (i)
theoretical justification of the bootstrap, (ii) adding a bias correction, (iii)
constructing a consistent estimator of the unconditional variance. To implement
(i), an approach could be to generalize Theorem~\ref{thm:SIregularity} to
establish general differentiability of the adjusted mean with respect to general
representations, and then construct an estimator of the gradient. However, one
must assess whether this bias correction could negate the efficiency benefits of
DOPE.

\section*{Acknowledgments}
We thank Leonard Henckel for helpful
discussions. AMC and NRH were supported by a research grant (NNF20OC0062897)
from Novo Nordisk Fonden. ARL was supported by a research grant (0069071)
from Novo Nordisk Fonden.

\bibliographystyle{agsm}
\bibliography{bibliography.bib}

@article{rosenbaum1983central,
  title={The central role of the propensity score in observational studies for causal effects},
  author={Rosenbaum, Paul R and Rubin, Donald B},
  journal={Biometrika},
  volume={70},
  number={1},
  pages={41--55},
  year={1983},
  publisher={Oxford University Press}
}

@inproceedings{veitch2020adapting,
  title={Adapting text embeddings for causal inference},
  author={Veitch, Victor and Sridhar, Dhanya and Blei, David},
  booktitle={Conference on Uncertainty in Artificial Intelligence},
  pages={919--928},
  year={2020},
  organization={PMLR}
}

@article{rotnitzky2020efficient,
  title={Efficient Adjustment Sets for Population Average Causal Treatment Effect Estimation in Graphical Models.},
  author={Rotnitzky, Andrea and Smucler, Ezequiel},
  journal={Journal of Machine Learning Research},
  volume={21},
  number={188},
  pages={1--86},
  year={2020}
}

@article{hahn1998role,
  title={On the role of the propensity score in efficient semiparametric estimation of average treatment effects},
  author={Hahn, Jinyong},
  journal={Econometrica},
  pages={315--331},
  year={1998},
  publisher={JSTOR}
}

@article{perkovic2018complete,
author  = {Emilija Perkovi\'c and Johannes Textor and Markus Kalisch and Marloes H. Maathuis},
  title   = {Complete Graphical Characterization and Construction of Adjustment Sets in Markov Equivalence Classes of Ancestral Graphs},
  journal = {Journal of Machine Learning Research},
  year    = {2018},
  volume  = {18},
  number  = {220},
  pages   = {1--62},
}

@incollection{emery2001vershik,
  title={On {V}ershik’s standardness criterion and {T}sirelson’s notion of cosiness},
  author={{\'E}mery, Michel and Schachermayer, Walter},
  booktitle={S{\'e}minaire de Probabilit{\'e}s XXXV},
  pages={265--305},
  year={2001},
  publisher={Springer}
}

@article{chernozhukov2018,
    author = {Chernozhukov, Victor and Chetverikov, Denis and Demirer, Mert and Duflo, Esther and Hansen, Christian and Newey, Whitney and Robins, James},
    title = "{Double/debiased machine learning for treatment and structural parameters}",
    journal = {The Econometrics Journal},
    volume = {21},
    number = {1},
    pages = {C1-C68},
    year = {2018},
    month = {01},
    issn = {1368-4221}
}

@article{henckel2022graphical,
  title={Graphical criteria for efficient total effect estimation via adjustment in causal linear models},
  author={Henckel, Leonard and Perkovi{\'c}, Emilija and Maathuis, Marloes H},
  journal={Journal of the Royal Statistical Society Series B: Statistical Methodology},
  volume={84},
  number={2},
  pages={579--599},
  year={2022},
  publisher={Oxford University Press}
}

@article{smucler2019unifying,
  title={A unifying approach for doubly-robust l1-regularized estimation of causal contrasts},
  author={Smucler, Ezequiel and Rotnitzky, Andrea and Robins, James M},
  journal={arXiv:1904.03737},
  year={2019}
}

@book{kolmogoroff1933grundbegriffe,
  author = {Kolmogoroff, Andrey},
  title = {Grundbegriffe der Wahrscheinlichkeitsrechnung},
  publisher = {Springer Berlin, Heidelberg},
  year = {1933},
  edition = {1},
  series = {Ergebnisse der Mathematik und Ihrer Grenzgebiete. 1. Folge},
  isbn = {978-3-642-49596-0},
  pages = {V, 62},
}

@book{peters2017elements,
  title={Elements of causal inference: foundations and learning algorithms},
  author={Peters, Jonas and Janzing, Dominik and Sch{\"o}lkopf, Bernhard},
  year={2017},
  publisher={The MIT Press}
}

@book{kallenberg2021foundations,
    year = 2021,
    publisher = {Springer International Publishing},
    author = {Olav Kallenberg},
    title = {Foundations of Modern Probability},
    edition = 3
}

@book{billingsley2017probability,
  title={Probability and measure},
  author={Billingsley, Patrick},
  year={2017},
  publisher={John Wiley \& Sons}
}

@book{pearl2009causality,
  title={Causality},
  author={Pearl, Judea},
  year={2009},
  publisher={Cambridge University Press}
}

@article{shortreed2017outcome,
  title={Outcome-adaptive lasso: variable selection for causal inference},
  author={Shortreed, Susan M and Ertefaie, Ashkan},
  journal={Biometrics},
  volume={73},
  number={4},
  pages={1111--1122},
  year={2017},
  publisher={Wiley Online Library}
}

@article{ju2020robust,
  title={Robust inference on the average treatment effect using the outcome highly adaptive lasso},
  author={Ju, Cheng and Benkeser, David and {van der Laan}, Mark J},
  journal={Biometrics},
  volume={76},
  number={1},
  pages={109--118},
  year={2020},
  publisher={Wiley Online Library}
}

@article{benkeser2020nonparametric,
author = {Benkeser, David and Cai, Weixin and {van der Laan}, Mark},
year = {2020},
month = {08},
pages = {484-495},
title = {A Nonparametric Super-Efficient Estimator of the Average Treatment Effect},
volume = {35},
journal = {Statistical Science}
}

@inproceedings{greenewald2021high,
  title={High-dimensional feature selection for sample efficient treatment effect estimation},
  author={Greenewald, Kristjan and Shanmugam, Karthikeyan and Katz, Dmitriy},
  booktitle={International Conference on Artificial Intelligence and Statistics},
  pages={2224--2232},
  year={2021},
  organization={PMLR}
}

@article{parkinson2023linear,
  title={Linear Neural Network Layers Promote Learning Single-and Multiple-Index Models},
  author={Parkinson, Suzanna and Ongie, Greg and Willett, Rebecca},
  journal={arXiv:2305.15598},
  year={2023}
}

@book{van2000asymptotic,
  title={Asymptotic statistics},
  author={van der Vaart, Aad W},
  volume={3},
  year={2000},
  publisher={Cambridge university press}
}

@article{hansen2008prognostic,
  title={The prognostic analogue of the propensity score},
  author={Hansen, Ben B},
  journal={Biometrika},
  volume={95},
  number={2},
  pages={481--488},
  year={2008},
  publisher={Oxford University Press}
}

@article{robins1995semiparametric,
  title={Semiparametric efficiency in multivariate regression models with missing data},
  author={Robins, James M and Rotnitzky, Andrea},
  journal={Journal of the American Statistical Association},
  volume={90},
  number={429},
  pages={122--129},
  year={1995},
  publisher={Taylor \& Francis}
}

@book{van2011targeted,
  title={Targeted learning: causal inference for observational and experimental data},
  author={{van der Laan}, Mark J and Rose, Sherri},
  volume={4},
  year={2011},
  publisher={Springer}
}

@article{lundberg2020local,
  title={From local explanations to global understanding with explainable {AI} for trees},
  author={Lundberg, Scott M and Erion, Gabriel and Chen, Hugh and DeGrave, Alex and Prutkin, Jordan M and Nair, Bala and Katz, Ronit and Himmelfarb, Jonathan and Bansal, Nisha and Lee, Su-In},
  journal={Nature machine intelligence},
  volume={2},
  number={1},
  pages={56--67},
  year={2020},
  publisher={Nature Publishing Group UK London}
}

@article{franklin1999pulse,
  title={Is pulse pressure useful in predicting risk for coronary heart disease? {T}he {F}ramingham Heart Study},
  author={Franklin, Stanley S and Khan, Shehzad A and Wong, Nathan D and Larson, Martin G and Levy, Daniel},
  journal={Circulation},
  volume={100},
  number={4},
  pages={354--360},
  year={1999},
  publisher={Am Heart Assoc}
}

@article{shi2019adapting,
  title={Adapting neural networks for the estimation of treatment effects},
  author={Shi, Claudia and Blei, David and Veitch, Victor},
  journal={Advances in neural information processing systems},
  volume={32},
  year={2019}
}

@article{balde2023reader,
  title={Reader reaction to “{O}utcome-adaptive lasso: Variable selection for causal inference” by {S}hortreed and {E}rtefaie (2017)},
  author={Bald{\'e}, Ismaila and Yang, Yi Archer and Lefebvre, Genevi{\`e}ve},
  journal={Biometrics},
  volume={79},
  number={1},
  pages={514--520},
  year={2023},
  publisher={Wiley Online Library}
}

@article{devlin2018bert,
  title={{BERT}: Pre-training of deep bidirectional transformers for language understanding},
  author={Devlin, Jacob and Chang, Ming-Wei and Lee, Kenton and Toutanova, Kristina},
  journal={arXiv:1810.04805},
  year={2018}
}

@article{forre2023mathematical,
  title={A Mathematical Introduction to Causality},
  author={Forr{\'e}, Patrick and Mooij, Joris M},
  year={2023}
}

@article{guo2022confounder,
  title={Confounder Selection: Objectives and Approaches},
  author={Guo, F Richard and Lundborg, Anton Rask and Zhao, Qingyuan},
  journal={arXiv:2208.13871},
  year={2022}
}

@article{lundborg2023perturbation,
  title={Perturbation-based Analysis of Compositional Data},
  author={Lundborg, Anton Rask and Pfister, Niklas},
  journal={arXiv:2311.18501},
  year={2023}
}

@article{powell1989semiparametric,
  title={Semiparametric estimation of index coefficients},
  author={Powell, James L and Stock, James H and Stoker, Thomas M},
  journal={Econometrica: Journal of the Econometric Society},
  pages={1403--1430},
  year={1989},
  publisher={JSTOR}
}

@article{newey1994,
  title = {The {{Asymptotic Variance}} of {{Semiparametric Estimators}}},
  author = {Newey, Whitney K.},
  year = 1994,
  journal = {Econometrica},
  volume = {62},
  number = {6},
  eprint = {2951752},
  eprinttype = {jstor},
  pages = {1349--1382},
  publisher = {[Wiley, Econometric Society]},
  issn = {0012-9682}
}

@article{ichimura1993semiparametric,
  title={Semiparametric least squares ({SLS}) and weighted {SLS} estimation of single-index models},
  author={Ichimura, Hidehiko},
  journal={Journal of Econometrics},
  volume={58},
  number={1-2},
  pages={71--120},
  year={1993},
  publisher={Elsevier}
}

@article{delecroix2003efficient,
  title={Efficient estimation in conditional single-index regression},
  author={Delecroix, Michel and H{\"a}rdle, Wolfgang and Hristache, Marian},
  journal={Journal of Multivariate Analysis},
  volume={86},
  number={2},
  pages={213--226},
  year={2003},
  publisher={Elsevier}
}

@article{cox1997plan,
  title={Plan and Operation of the {NHANES I} Epidemiologic Follow-up Study, 1992},
  author={Cox, Christine S. and Feldman, Jacob J. and Golden, Cordell D. and Lane, Madelyn A. and Madans, Jennifer H. and Mussolino, Michael E. and Rothwell, Sandra T.},
  journal={Vital and health statistics.},
  year={1997},
  month={12},
  publisher={National Center for Health Statistics (U.S.)},
  series={DHHS Publication}
}

@book{van1995python,
  title={Python reference manual},
  author={Van Rossum, Guido and Drake, Fred L and others},
  volume={111},
  year={1995},
  publisher={Centrum voor Wiskunde en Informatica Amsterdam}
}

@incollection{paszke2019pytorch,
title = {PyTorch: An Imperative Style, High-Performance Deep Learning Library},
author = {Paszke, Adam and Gross, Sam and Massa, Francisco and Lerer, Adam and Bradbury, James and Chanan, Gregory and Killeen, Trevor and Lin, Zeming and Gimelshein, Natalia and Antiga, Luca and Desmaison, Alban and Kopf, Andreas and Yang, Edward and DeVito, Zachary and Raison, Martin and Tejani, Alykhan and Chilamkurthy, Sasank and Steiner, Benoit and Fang, Lu and Bai, Junjie and Chintala, Soumith},
booktitle = {Advances in Neural Information Processing Systems 32},
pages = {8024--8035},
year = {2019},
publisher = {Curran Associates, Inc.}
}

@article{gnecco2023boosted,
  title={Boosted Control Functions},
  author={Gnecco, Nicola and Peters, Jonas and Engelke, Sebastian and Pfister, Niklas},
  journal={arXiv:2310.05805},
  year={2023}
}

@article{van2010collaborative,
  title={Collaborative double robust targeted maximum likelihood estimation},
  author={{van der Laan}, Mark J and Gruber, Susan},
  journal={The international journal of biostatistics},
  volume={6},
  number={1},
  year={2010},
  publisher={De Gruyter}
}

@article{ju2019scalable,
  title={Scalable collaborative targeted learning for high-dimensional data},
  author={Ju, Cheng and Gruber, Susan and Lendle, Samuel D and Chambaz, Antoine and Franklin, Jessica M and Wyss, Richard and Schneeweiss, Sebastian and {van der Laan}, Mark J},
  journal={Statistical methods in medical research},
  volume={28},
  number={2},
  pages={532--554},
  year={2019},
  publisher={SAGE Publications Sage UK: London, England}
}

@article{homan2024physiology,
  title = {Physiology, Pulse Pressure},
  author = {Homan, T.D. and Bordes, S.J. and Cichowski, E.},
  year = {2024},
  howpublished = {In: StatPearls [Internet]},
  publisher = {StatPearls Publishing},
  address = {Treasure Island (FL)},
  note = {[Updated 2023 Jul 10]}
}

@article{scholkopf2021toward,
  title={Toward causal representation learning},
  author={Sch{\"o}lkopf, Bernhard and Locatello, Francesco and Bauer, Stefan and Ke, Nan Rosemary and Kalchbrenner, Nal and Goyal, Anirudh and Bengio, Yoshua},
  journal={Proceedings of the IEEE},
  volume={109},
  number={5},
  pages={612--634},
  year={2021},
  publisher={IEEE}
}

@article{shah2020hardness,
  title={The hardness of conditional independence testing and the generalised covariance measure},
  author={Shah, Rajen D and Peters, Jonas},
journal = {The Annals of Statistics},
	number = {3},
	pages = {1514 -- 1538},
	volume = {48},
  year={2020}
}

@article{christgau2023nonparametric,
  title={Nonparametric conditional local independence testing},
  author={Christgau, Alexander Mangulad and Petersen, Lasse and Hansen, Niels Richard},
  journal={The Annals of Statistics},
  volume={51},
  number={5},
  pages={2116--2144},
  year={2023},
  publisher={Institute of Mathematical Statistics}
}

@article{zivich2021machine,
  title={Machine learning for causal inference: on the use of cross-fit estimators},
  author={Zivich, Paul N and Breskin, Alexander},
  journal={Epidemiology (Cambridge, Mass.)},
  volume={32},
  number={3},
  pages={393},
  year={2021},
  publisher={NIH Public Access}
}

@article{scikit-learn,
  title={Scikit-learn: Machine Learning in {P}ython},
  author={Pedregosa, F. and Varoquaux, G. and Gramfort, A. and Michel, V.
          and Thirion, B. and Grisel, O. and Blondel, M. and Prettenhofer, P.
          and Weiss, R. and Dubourg, V. and Vanderplas, J. and Passos, A. and
          Cournapeau, D. and Brucher, M. and Perrot, M. and Duchesnay, E.},
  journal={Journal of Machine Learning Research},
  volume={12},
  pages={2825--2830},
  year={2011}
}

@inproceedings{benkeser2016highly,
  title={The highly adaptive lasso estimator},
  author={Benkeser, David and {van der Laan}, Mark},
  booktitle={2016 IEEE international conference on data science and advanced analytics (DSAA)},
  pages={689--696},
  year={2016},
  organization={IEEE}
}

@article{bojer2021kaggle,
  title={Kaggle forecasting competitions: An overlooked learning opportunity},
  author={Bojer, Casper Solheim and Meldgaard, Jens Peder},
  journal={International Journal of Forecasting},
  volume={37},
  number={2},
  pages={587--603},
  year={2021},
  publisher={Elsevier}
}

@article{uhler2013geometry,
  title={Geometry of the faithfulness assumption in causal inference},
  author={Uhler, Caroline and Raskutti, Garvesh and B{\"u}hlmann, Peter and Yu, Bin},
  journal={The Annals of Statistics},
  pages={436--463},
  year={2013},
  publisher={JSTOR}
}

@article{chickering2004large,
  title={Large-sample learning of Bayesian networks is {NP}-hard},
  author={Chickering, Max and Heckerman, David and Meek, Chris},
  journal={Journal of Machine Learning Research},
  volume={5},
  pages={1287--1330},
  year={2004}
}

@book{bickel1993efficient,
  title={Efficient and adaptive estimation for semiparametric models},
  author={Bickel, Peter J and Klaassen, Chris AJ and Bickel, Peter J and Ritov, Ya’acov and Klaassen, J and Wellner, Jon A and Ritov, YA'Acov},
  volume={4},
  year={1993},
  publisher={Springer}
}

@inbook{bolthausen2002lectures,
  title={Lectures on Probability Theory and Statistics},
  subtitle={Ecole d'Été de Probabilités de Saint-Flour XXIX - 1999},
  author={{van der Vaart}, Aad},
  chapter={III},
  series={Lecture Notes in Mathematics},
  volume={1781},
  year={2002},
  publisher={Springer Berlin, Heidelberg},
  pages={323-466},
  edition={1}
}

@article{hines2022demystifying,
  title={Demystifying statistical learning based on efficient influence functions},
  author={Hines, Oliver and Dukes, Oliver and Diaz-Ordaz, Karla and Vansteelandt, Stijn},
  journal={The American Statistician},
  volume={76},
  number={3},
  pages={292--304},
  year={2022},
  publisher={Taylor \& Francis}
}

@article{young2024rose,
  title={ROSE Random Forests for Robust Semiparametric Efficient Estimation},
  author={Young, Elliot H and Shah, Rajen D},
  journal={arXiv preprint arXiv:2410.03471},
  year={2024}
}

@article{lecam1970assumptions,
  title={On the assumptions used to prove asymptotic normality of maximum likelihood estimates},
  author={{Le Cam}, Lucien},
  journal={The Annals of Mathematical Statistics},
  volume={41},
  number={3},
  pages={802--828},
  year={1970},
  publisher={JSTOR}
}

@article{lee2016exact,
author = {Jason D. Lee and Dennis L. Sun and Yuekai Sun and Jonathan E. Taylor},
title = {{Exact post-selection inference, with application to the lasso}},
volume = {44},
journal = {The Annals of Statistics},
number = {3},
publisher = {Institute of Mathematical Statistics},
pages = {907 -- 927},
year = {2016}
}

\clearpage
\appendix
\appendixname{Supplement to ``Efficient adjustment for complex covariates: Gaining efficiency with DOPE''}
The supplementary material is organized as follows:
    Section~\ref{sec:efficiencybounds} contains a description and derivation of
    some semiparametric efficiency bounds for adjusted means.
    Section~\ref{sup:proofs} contains proofs of the results in the main text and
    a few auxiliary results; Section~\ref{sup:simdetails} contains details on
    the experiments in Section~\ref{sec:experiments}; and
    Section~\ref{sup:crossfitting} contains a description ofa
    cross-fit variant of the DOPE algorithm. 

\setcounter{section}{0}
\setcounter{algocf}{0}
\setcounter{table}{0}
\renewcommand\thesection{\Alph{section}}
\renewcommand\thealgocf{\Alph{section}.\arabic{algocf}}
\renewcommand\thetable{\Alph{section}.\arabic{table}}

\section{Efficiency bounds for representation
adjustment}\label{sec:efficiencybounds}

The target estimand is throughout the adjusted mean $\mu_t = \mu_t(\bW; P)$, which is 
well defined assuming only positivity (Assumption~\ref{asm:positivity}).
\citet{hahn1998role} showed that $\bV_t(\bW;P)$ is the
semiparametric efficiency bound for RAL estimators of $\mu_t$ when no further model 
assumptions are imposed. 

Our main results in Section~\ref{sec:informationbounds}
are on the asymptotic \emph{relative} efficiency of particular estimators under model assumptions pertaining to descriptions of $\bW$.  
It is natural to ask if any of those estimators is semiparametrically efficient, that is, 
if the asymptotic variance of one such estimator equals a lower bound.
In other words, we ask if $\bV_t(\mathcal{Z}; P)$ is a lower bound for the asymptotic 
variance of RAL estimators in a model $\mathcal{P}$, for which $\mathcal{Z}$ is 
$\mathcal{P}$-valid, $\mathcal{P}$-OMS or $\mathcal{P}$-ODS. To approach this 
question we include in this section the following:
\begin{itemize}
    \item We state and prove rigorously two related but different results on 
    asymptotic efficiency bounds for the submodel of distributions $P$ that have a given 
    and common $P$-ODS representation $\mathbf{Z}$.
     \item We discuss in more detail the relation between efficiency 
     bounds and the main results obtained in the paper.
\end{itemize}

We emphasize that it is far more 
difficult to rigorously establish lower efficiency bounds, that depend 
heavily on the specific model $\mathcal{P}$, than comparing the 
relative efficiency of particular estimators. 

\subsection{Asymptotic efficiency bounds} \label{sec:asymptoticefficiencybounds}
The lower efficiency bounds we state below require some additional structure on 
the models considered. We let $(T,\bW,Y)$ denote a template
observation taking values in the sample space $\mathbb{T}\times \mathbb{W} \times \mathbb{R}$.
We suppose that $\nu_\mathbb{T}$, $\nu_\mathbb{W}$, and $\nu_\mathbb{R}$, are
$\sigma$-finite measures on $\mathbb{T}$, $\mathbb{W}$, and $\mathbb{R}$,
respectively, and we set $\nu = \nu_\mathbb{T} \otimes \nu_\mathbb{W} \otimes
\nu_\mathbb{R}$. We let $\mathcal{P}_\nu$ denote the set of probability measures
that are absolutely continuous with respect to $\nu$. 

Now let $\phi \colon \mathbb{W} \to \real^k$ be a given measurable map, and let $\bZ=\phi(\bW)$
be the corresponding representation. We then consider the model given as:
\begin{align} \label{eq:P-ods}
    \mathcal{P}^{\textsc{ods}} \coloneq \big\{P \in \mathcal{P}_\nu \given \, 
        \forall t\in \mathbb{T}\colon \mathbb{P}_P(0<\pi_t(\bW;P)<1)=1,\,
        \ex_P[Y^2]<\infty, \, \bZ \text{ is $P$-ODS} 
    \,\big\}
\end{align}

The canonical asymptotic efficiency bound for the estimand $\mu_t$ relative to 
the model  $\mathcal{P}^{\textsc{ods}}$ is, loosely speaking, the
supremum of Cramér-Rao lower bounds taken over all finite-dimensional parametric
submodels. To compute this supremum, it is typically sufficient to consider
one-dimensional submodels, i.e., paths $[0,1) \ni \epsilon \longmapsto
P_\epsilon \in \mathcal{P}_\nu$. Some care must be taken, though, as the paths should be
restricted to depend ``smoothly'' on the parameter $\epsilon$. 

With a suitable notion of path differentiability, the \emph{influence function} 
is defined as the Riesz representer of the gradient of $\mu_t$, and the asymptotic 
efficiency bound can then be computed as the variance of the influence
function \citep[Chapter 3.3]{bickel1993efficient}. 
The weakest smoothness assumption is \textit{differentiability in quadratic mean (DQM)}, or
simply \textit{regular}, see \citet{lecam1970assumptions} or Definition 2.1 in
\citet{bickel1993efficient}, and using this notion of differentiability 
our first result states that $\psi_t(\bZ;P)$ is, indeed, the influence function 
of $\mu_t$ in the smaller model
\begin{align}
    \mathcal{P}_{c,C}^{\textsc{ods}}
        = \big\{ P\in \mathcal{P}^{\textsc{ods}} \given
            \forall t \in \mathbb{T}: 
        \mathbb{P}_{P}(\pi_t(\bW; P) \in [c,1-c]) = 1; \mathbb{P}_{P}(|Y|\leq C) = 1 
        \big\} 
\end{align}
where $c,C>0$ are fixed constants. 

\begin{thm}\label{thm:dqmpaths} The estimand $\mu_t$, considered as a map
    $\mathcal{P}_{c,C}^{\textsc{ods}}\to \real$, is pathwise differentiable at
    $P\in \mathcal{P}_{c,C}^{\textsc{ods}}$ with respect to all regular paths
    and with influence function $\psi_t(\bZ;P)$. 
\end{thm}
As a consequence of Theorem~\ref{thm:dqmpaths}, the variance
$\mathbb{V}_t(\bZ;P) = \var_P(\psi_t(\bZ;P))$ is the supremum of all Cramér-Rao
lower bounds for $\mu_t$ over all regular parametric submodels \cite[Thm.
3.3.1]{bickel1993efficient}. In the proof we construct a one-dimensional
submodel with $\mathbb{V}_t(\bZ;P)$ as its Cramér-Rao lower bound, i.e., the
supremum is attained for this submodel. 

Another lower bound result is the
following: for any $P\in \mathcal{P}_{c,C}^{\textsc{ods}}$, subconvex loss
function $\ell \colon \real \to [0,\infty)$, and sequence of estimators $T_n$, 
\begin{align*}
    \inf_{\delta > 0} \liminf_{n\to \infty} \sup_{Q \colon \mathrm{TV}(P,Q)<\delta}
    \ex_{Q}[\ell(\sqrt{n}(T_n - \mu_t(Q)))]
         \geq \int \ell \mathrm{d}\mathrm{N}(0,\mathbb{V}_t(\bZ;P))
\end{align*}
where $\mathrm{TV}$ denotes the total variation distance. See Theorem III.2.5
by \citet{bolthausen2002lectures} and Theorem 25.21 by \citet{van2000asymptotic}
for further details.

While DQM is sufficient to prove general results in
nonparametric theory, \cite{bolthausen2002lectures} points out
that the collection of regular paths is often ``too big'' and impractical to
work with for concrete examples. For example, a regular path does not
necessarily yield a pointwise differentiable likelihood, and hence the product
rule cannot be used to compute the score. Therefore further restrictions are
required to formulate a rigorous result. Recent influential papers, such as
\cite{hines2022demystifying,chernozhukov2018}, focus on paths obtained from
convex combinations, i.e., paths of the form $\epsilon \mapsto (1-\epsilon)P +
\epsilon \Tilde{P}$. This more restricted collection of paths
corresponds to a weaker notion of pathwise differentiability, but it can be
convenient for computing the functional form of the influence
function. 

The model $\mathcal{P}^{\textsc{ods}}$ is, however, not closed under convex mixtures,
and we therefore need a more flexible class of paths for this particular model.  
The following theorem establishes the
pathwise differentiability over a fairly general type of paths, similar to the
ones considered by \cite{hahn1998role}.
\begin{thm}\label{thm:smoothpaths}
    Define the collection of paths
    \begin{align*}
    \Gamma&\coloneq \Big\{
        (P_{\epsilon})_{\epsilon\in [0,1)} \subseteq \mathcal{P}^{\textsc{ods}} \given
        \epsilon \mapsto \frac{\mathrm{d}P_\epsilon}{\mathrm{d}\nu}
        \text{   regular and $\nu$-a.e. $C^1$}\Big\}.
    \end{align*}
    Then, at any distribution $P\in \mathcal{P}^{\textsc{ods}}$, the estimand
    $\mu_t$ is pathwise differentiable with respect to $\Gamma$ with influence
    function $\psi_t(\bZ;P)$.
\end{thm}
Note that Theorem~\ref{thm:smoothpaths} does not imply
Theorem~\ref{thm:dqmpaths} and vice-versa, and their proofs are different. We
will give a full proof of Theorem~\ref{thm:dqmpaths} in Section~\ref{sec:proofA1} 
and merely sketch the proof of Theorem~\ref{thm:smoothpaths} in Section~\ref{sec:proofA2}.
Note also that the efficiency bound is
independent of the choice of dominating measure $\nu$. It is plausible that by invoking
the space of square root measures, the assumption of a dominating measure can be
removed, but at the cost of a significantly more technical argumentation.

From the special case where $\phi$ is the identity map, we recover the classical
efficiency bound $\mathbb{V}_t(\bW;P)$ for estimation of $\mu_t$ without  
additional model assumptions as established by \cite{hahn1998role}. The results above 
may be unsurprising, but we found it worthwhile stating 
them in a precise way, and we believe that Theorem~\ref{thm:dqmpaths} is a novel result.

\subsection{How our results relate to asymptotic efficiency bounds}

Suppose $\mathcal{P}$ is some model, that is, a collection of probability 
measures on $\mathbb{T}\times \mathbb{W} \times \mathbb{R}$ fulfilling 
Assumption \ref{asm:positivity}. Our results in Section~\ref{sec:informationbounds} 
then characterize potential efficiency gains by adjusting for a $P$-ODS 
description instead of adjusting for $\bW$. Theorem \ref{thm:COSefficiency}
shows that $\mathcal{Q}_P$ is effectively the minimal $P$-ODS description 
with the smallest possible asymptotic variance -- for the estimators we consider.
However, it would require oracle knowledge about $P$, or additional model 
assumptions, to use $\mathcal{Q}_P$ for actual estimation. 

Corollary~\ref{cor:Qefficiency} shows that $\mathcal{Q}$ is $\mathcal{P}$-ODS 
and establishes a potential efficiency gain for the model $\mathcal{P}$ 
by adjusting for $\mathcal{Q}$ instead of adjusting for $\bW$. If 
there exists a representation $\bZ = \phi(\bW)$ with $\sigma(\bZ) = \mathcal{Q}$, 
we have that 
$$
    \mathcal{P} \subseteq \mathcal{P}^{\textsc{ods}}
$$
where $\mathcal{P}^{\textsc{ods}}$ is defined by \eqref{eq:P-ods} in terms of $\bZ$. 
The efficiency bounds discussed in Section~\ref{sec:asymptoticefficiencybounds} show 
that in the larger model $\mathcal{P}^{\textsc{ods}}$, adjusting for 
$\mathcal{Q}$ gives an estimator that, indeed, attains the asymptotic efficiency 
bound.

Since $\mathbb{V}_t(\mathcal{Q};P) \leq \mathbb{V}_t(\bW;P)$ by 
Corollary~\ref{cor:Qefficiency} the obvious suggestion is to always adjust for 
$\mathcal{Q}$ instead of $\bW$. We never lose anything and we gain  
asymptotic efficiency if the inequality is sharp for some $P$. 
This might appear to contradict
that $\mathbb{V}_t(\bW;P)$ is the efficiency bound in an unrestricted model, 
but this is, of course, not the case. On the contrary, if 
$\mathbb{V}_t(\mathcal{Q};P) < \mathbb{V}_t(\bW;P)$ for some $P$ we can 
argue that the model $\mathcal{P}$ that induces the description $\mathcal{Q}$
must impose additional structural model assumptions. Thus a valid interpretation 
of our results is that they simply expose how a potential efficiency gain 
is encoded by additional model assumptions. 

However, if the model $\mathcal{P}$ imposes additional assumptions compared to 
$\mathcal{P}^{\textsc{ods}}$, it is quite possible that the
efficiency bound for $\mathcal{P}$ is strictly smaller than $\mathbb{V}_t(\mathcal{Q};P)$. 
Thus we do not 
claim that estimators with asymptotic variance $\mathbb{V}_t(\mathcal{Q};P)$ are
among the most efficient estimators for general models $\mathcal{P}$, and we regard it as an 
open problem to compute asymptotic efficiency bounds for general 
models $\mathcal{P}$ with a $\mathcal{P}$-ODS description $\mathcal{Z}$.  
For models with only $\mathcal{P}$-OMS or $\mathcal{P}$-valid
descriptions it is even more so unknown what the asymptotic efficiency bounds are.

\subsection{Proof of Theorem~\ref{thm:dqmpaths}} \label{sec:proofA1}
We may assume without loss of generality that $\mathbb{T}=\{0,1\}$ and that
$t=1$, which can be achieved by replacing $T$ with the variable $\one(T=t)$.
Then we may also assume that $\nu_{\mathbb{T}}$ is the counting measure on
$\{0,1\}$.

Fix $P_0 \in \mathcal{P}_{c,C}^{\textsc{ods}}$ and consider an arbitrary DQM
curve $\epsilon\mapsto P_\epsilon\in \mathcal{P}_{c,C}^{\textsc{ods}}$ through
$P_0$, and let 
$$
s_\epsilon = \left(\frac{\mathrm{d}P_\epsilon}{\mathrm{d}\nu}\right)^{\frac{1}{2}} \in L^2(\nu), 
\qquad \epsilon \in [0,1),
$$
denote its square root density. By definition of DQM, there exists a function
$\tau\in L^2(P_0)$ (which is equivalent to $\tau s_0 \in L^2(\nu)$) such that
\begin{equation}\label{eq:DQM}
    \int \left(\frac{s_\epsilon-s_0}{\epsilon} - \frac{1}{2} \tau s_0\right)^2 \mathrm{d}\nu 
        \longrightarrow 0.
\end{equation}
Note that $\tau$ is unique, and any measurable function $\tau$ satisfying the
above property will automatically satisfy $\ex_{P_0}[\tau(T,\bW,Y)^2]<\infty$
and $\ex_{P_0}[\tau(T,\bW,Y)]=0$. We refer to $\tau$ the tangent associated to
curve $\epsilon \mapsto P_\epsilon$.

Recall that for $P\in \mathcal{P}_{c,C}^{\textsc{ods}}$, the estimand $\mu_1$ is
given by
\begin{align*}
    \mu_1(P) 
    = \ex_{P}[\ex_{P}[Y\given T=1, \bW]]
    = \ex_{P}[\ex_{P}[Y\given T=1, \bZ]].
\end{align*}
It is convenient to define $Q_\epsilon(\bz) \coloneq \ex_{P_\epsilon}[Y\given
T=1,\bZ=\bz]$ and $\varpi_\epsilon(\bz) = \mathbb{P}_{P_\epsilon}(T=1\given
\bZ=\bz)$, so that $\mu_1(P_\epsilon) = \ex_{P_\epsilon}[Q_\epsilon(\bZ)]$.

The claim is that $\mu_1$ is pathwise differentiable at $P_0$ with respect to
all regular paths in $\mathcal{P}_{c,C}^{\textsc{ods}}$, and with influence
function
\begin{align*}
    \psi(t,\bz,y) =
    Q_0(\bz) + \underbrace{\frac{t(y-Q_0(\bz))}{\varpi_0(\bz)}}_{\eqcolon \zeta(t,\bz,y)} - \mu_1(P_0)
\end{align*}
From the definition of $\mathcal{P}_{c,C}^{\textsc{ods}}$ it follows that
$C_\psi \coloneq \|\psi\|_{L^\infty(\nu)}\leq 2C(1+c^{-1})<\infty$.

We employ the following observation, which is used explicitly in the proof of
Theorem 1 in \citet{young2024rose}, and is implicitly used in the proof of
Proposition A.5.2 in \citet{bickel1993efficient}:

If $(h_{\epsilon})_{\epsilon\in[0,1)}\subset L^2(\nu)$ satisfies that
$\limsup_{\epsilon\to 0} \|h_{\epsilon}\|_{L^2(\nu)}<\infty$, then the
Cauchy-Schwartz inequality implies that
\begin{align}\label{eq:RajenTrick}
    \left\lvert
    \int h_{\epsilon} \left(\frac{s_\epsilon-s_0}{\epsilon} - \frac{1}{2} \tau s_0\right)\mathrm{d}\nu
    \right\rvert
    \leq \|h_{\epsilon}\|_{L^2(\nu)} 
    \left\|\frac{s_\epsilon-s_0}{\epsilon} - \frac{1}{2} \tau s_0\right\|_{L^2(\nu)}  \longrightarrow 0.
\end{align}
We will apply \eqref{eq:RajenTrick} with
\begin{align*}
    h_{\epsilon} = \psi \cdot (s_\epsilon + s_0),
\end{align*}
which is possible because 
$
    \|h_{\epsilon}\|_{L^2(\nu)} \leq C_\psi \|s_\epsilon +s_0\| _{L^2(\nu)} \leq 2C_\psi
$.
Since 
$$
    |\inner{\psi \tau s_0}{s_\epsilon - s_0}_{L^2(\nu)}|
    \leq  C_\psi \|\tau\|_{L^2(P_0)} \|s_\epsilon - s_0\|_{L^2(\nu)} \to 0,
    \qquad \epsilon \to 0,
$$ 
it follows that, as $\epsilon \to 0$,
\begin{align*}
    R_\epsilon \coloneq \frac{1}{2}\int h_{\epsilon} \tau s_0 \mathrm{d}\nu
    = \frac{1}{2}\inner{\psi \tau s_0}{s_\epsilon + s_0}_{L^2(\nu)}
    \longrightarrow 
    \inner{\psi \tau s_0}{s_0}_{L^2(\nu)} = \inner{\psi}{\tau}_{L^2(P_0)}.
\end{align*}
We now turn our attention to
\begin{align*}
    L_\epsilon \coloneq \int \frac{h_{\epsilon}(s_\epsilon-s_0)}{\epsilon} \mathrm{d}\nu
    &= \frac{1}{\epsilon}\int \psi (s_\epsilon^2-s_0^2) \mathrm{d}\nu
    = \frac{1}{\epsilon} \int \psi \mathrm{d}P_\epsilon \\
    &= \frac{\ex_{P_\epsilon}[Q_0(\bZ) + \zeta(T,\bZ,Y)] - \mu_1(P_0)
    }{\epsilon}
\end{align*}
From iterated expectations it follows that
\begin{align*}
    \ex_{P_\epsilon} \left[
        \zeta(T,\bZ,Y) 
    \right]
    = \ex_{P_\epsilon} \left[
        \frac{T(Q_\epsilon(\bZ)-Q_0(\bZ))}{\varpi_0(\bZ)}
    \right]
    = \ex_{P_\epsilon} \left[
        \frac{\varpi_\epsilon(\bZ)(Q_\epsilon(\bZ)-Q_0(\bZ))}{\varpi_0(\bZ)}
    \right]
\end{align*}
and hence
\begin{equation}\label{eq:remainder}
  \begin{aligned}
    |\ex_{P_\epsilon}[Q_0(\bW) + \psi_2(T,\bZ,Y)] - \mu_1(P_\epsilon)|
    &\leq
    \Big|\Big\langle\frac{\varpi_\epsilon-\varpi_0}{\varpi_0}, 
        Q_\epsilon - Q_0\Big\rangle_{L^2(P_\epsilon)}\Big|
    \\
    &\leq c^{-1}
        \|\varpi_\epsilon-\varpi_0\|_{L^2(P_\epsilon)}
        \|Q_\epsilon-Q_0\|_{L^2(P_\epsilon)}
  \end{aligned}
\end{equation}
To continue, we will bound each of the factors in the upper bound separately
below.

Let $\wp_\epsilon$ denote the density of $\bW$ under $P_\epsilon$, and let
$\pi_\epsilon(\bw)=\mathbb{P}_{P_\epsilon}(T=1\given \bW=\bw)$. By the tower
property and the Doob-Dynkin lemma, we see that $\varpi_\epsilon(\bz)$ is a
version of $\ex_{P_\epsilon}[\pi_\epsilon(\bW)\given \bZ=\bz]$. From conditional
Jensens inequality, we then obtain that
$$
    \|\varpi_\epsilon-\varpi_0\|_{L^2(P_\epsilon)} 
    \leq \|\pi_\epsilon-\pi_0\|_{L^2(P_\epsilon)} 
$$
By applying difference of squares and the bound
$\sqrt{\pi_\epsilon}+\sqrt{\pi_0}\leq 2$, we see that
\begin{align*}
    (\pi_\epsilon-\pi_0)^2\wp_\epsilon \leq 
    4 (\sqrt{\pi_\epsilon\wp_\epsilon} - \sqrt{\pi_0\wp_\epsilon})^2,
\end{align*}
from which we obtain that
\begin{align*}
    \|\pi_\epsilon-\pi_0\|_{L^2(P_\epsilon)} 
    &\leq 2 \|\sqrt{\pi_\epsilon\wp_\epsilon} -
        \sqrt{\pi_0\wp_\epsilon}\|_{L^2(\nu)} \\
    &\leq 2 \|\sqrt{\pi_\epsilon\wp_\epsilon} -
        \sqrt{\pi_0\wp_0}\|_{L^2(\nu)} 
        +2 \|\sqrt{\pi_0\wp_\epsilon} -
        \sqrt{\pi_0\wp_0}\|_{L^2(\nu)}.
\end{align*}
Now since 
$$
    \pi_\epsilon(\bw)\wp_\epsilon(\bw)
    = \int s_\epsilon^2(1,\bw,y)\nu_{\real}(\mathrm{d}y)
$$
we obtain from the reverse triangle inequality for $L^2(\nu_\real)$ 
\begin{align*}
    \|\sqrt{\pi_\epsilon\wp_\epsilon} &-
        \sqrt{\pi_0\wp_0}\|_{L^2(\nu)}^2 \\
    &= \int \Big( \Big(\int s_\epsilon^2(1,\bw,y)\nu_{\real}(\mathrm{d}y)\Big)^{1/2}
    -\Big(\int s_0^2(1,\bw,y)\nu_{\real}(\mathrm{d}y)\Big)^{1/2}
    \Big)^2 \nu_{\mathbb{W}}(\mathrm{d}\bw)\\
    &\leq \int \int (s_\epsilon(1,\bw,y)-s_0(1,\bw,y))^2\nu_{\real}(\mathrm{d}y)
    \nu_{\mathbb{W}}(\mathrm{d}\bw) \\
    &\leq \|s_\epsilon - s_0\|_{L^2(\nu)}^2
\end{align*}
For the other term we similarly have that
\begin{equation}
  \label{eq:varpibound2}
  \begin{aligned}
    \|\sqrt{\pi_0\wp_\epsilon} -
        \sqrt{\pi_0\wp_0}\|_{L^2(\nu)}^2
    &\leq \|\sqrt{\wp_\epsilon} -\sqrt{\wp_0}\|_{L^2(\nu)}^2 \\
    &=  \int \Big(\|s_\epsilon(\cdot,\bw,\cdot)\|_{L^2(\nu_{\mathbb{T}}\otimes \nu_{\real})}- \|s_0(\cdot, \bw,\cdot)\|_{L^2(\nu_{\mathbb{T}}\otimes \nu_{\real})}\Big)^2 \nu_{\mathbb{W}}(\mathrm{d}\bw) \\
    &\leq \int \|s_\epsilon(\cdot,\bw,\cdot)- s_0(\cdot,\bw,\cdot)\|_{L^2(\nu_{\mathbb{T}}\otimes \nu_{\real})}^2 \nu_{\mathbb{W}}(\mathrm{d}\bw) \\
    &= \|s_\epsilon - s_0\|_{L^2(\nu)}^2.
    \end{aligned}
\end{equation}
Together the inequalities show that
\begin{equation}
  \label{eq:varpibound}
  \|\varpi_\epsilon-\varpi_0\|_{L^2(P_\epsilon)} \leq 4 \|s_\epsilon -
s_0\|_{L^2(\nu)}.
\end{equation}

We now turn to the second factor in the upper bound of \eqref{eq:remainder}.
Let $f_\epsilon(\cdot \given \bw)$ denote the conditional density for $Y\given
T=1,\bW=\bw$ under $P_\epsilon$, and let $g_\epsilon(\bw)
=\ex_{P_\epsilon}[Y\given T=1,\bW=\bw]$. Since $\bZ$ is $P_\epsilon$-ODS, we
first note that 
$$
\|Q_\epsilon-Q_0\|_{L^2(P_\epsilon)} = \|g_\epsilon-g_0\|_{L^2(P_\epsilon)}.
$$
Using that $\mathbb{P}_{P_\epsilon}(|Y|<C)=1$, we also note that
\begin{align*}
    |g_\epsilon(\bw) - g_0(\bw)|
    &= \left| \int y (f_\epsilon(y\given \bw) - f_0(y\given \bw)) 
    \nu_\real(\mathrm{d}y)\right| \\
    &\leq C \left\langle
        \big|\sqrt{f_\epsilon}(\cdot \given \bw)
            -\sqrt{f_0}(\cdot \given \bw)\big|,
        \sqrt{f_\epsilon(\cdot \given \bw)}
            +\sqrt{f_0(\cdot \given \bw)}
        \right\rangle_{L^2(\nu_\real)} \\
    &\leq 2C \|\sqrt{f_\epsilon}(\cdot \given \bw)
            -\sqrt{f_0}(\cdot \given \bw) \|_{L^2(\nu_\real)}.
\end{align*}
Combining the above we obtain
\begin{align*}
    \frac{1}{2C}\|Q_\epsilon-Q_0\|_{L^2(P_\epsilon)}
    &\leq \|\sqrt{f_\epsilon \wp_\epsilon} -\sqrt{f_0 \wp_\epsilon}\|_{L^2(\nu)} \\
    &= \| \pi_\epsilon^{-1/2}(s_\epsilon(1,\cdot,\cdot) - \sqrt{f_0\pi_\epsilon \wp_\epsilon})\|_{L^2(\nu)} \\
    &\leq c^{-1/2}\| s_\epsilon(1,\cdot,\cdot) - \sqrt{f_0\pi_\epsilon \wp_\epsilon}\|_{L^2(\nu)} \\
    &\leq c^{-1/2}\| s_\epsilon(1,\cdot,\cdot) - s_0(1,\cdot,\cdot)\|_{L^2(\nu)}
        + c^{-1/2}\| \sqrt{f_0}( \sqrt{\pi_0 \wp_0} - \sqrt{\pi_\epsilon \wp_\epsilon})\|_{L^2(\nu)}\\
    &\leq 2 c^{-1/2} \|s_\epsilon - s_0\|_{L^2(\nu)},
\end{align*}
where the last inequality uses that
$$
\|\sqrt{f_0}( \sqrt{\pi_0 \wp_0} - \sqrt{\pi_\epsilon \wp_\epsilon})\|_{L^2(\nu)} 
= \|\sqrt{\pi_0 \wp_0}-\sqrt{\pi_\epsilon \wp_\epsilon}\|_{L^2(\nu)} \leq 
\|s_\epsilon - s_0\|_{L^2(\nu)},
$$
and where this last inequality was established in \eqref{eq:varpibound2}. We
conclude that 
\begin{equation}
  \label{eq:Qbound}
  \|Q_\epsilon-Q_0\|_{L^2(P_\epsilon)} \leq 4 C c^{-1/2} \|s_\epsilon - s_0\|_{L^2(\nu)}.
\end{equation}

We now return to the upper bound in \eqref{eq:remainder} and combine it with
\eqref{eq:varpibound} and \eqref{eq:Qbound} to obtain that 
\begin{align*}
  |L_\epsilon- \frac{\mu_1(P_\epsilon)-\mu_1(P_0)}{\epsilon}| = O\left(  \frac{\|s_\epsilon - s_0\|_{L^2(\nu)}^2}{\epsilon}\right).
\end{align*}
From the QMD property of path $\epsilon\mapsto P_\epsilon$, an application of 
the triangle inequality yields that
\begin{align*}
    \limsup_{\epsilon\to 0} \frac{\|s_\epsilon - s_0\|_{L^2(\nu)}}{\epsilon}
    = \frac{1}{2}\|\tau s_0\|_{L^2(\nu)},
\end{align*}
whence 
\[
  \frac{\|s_\epsilon - s_0\|_{L^2(\nu)}^2}{\epsilon} = O(\epsilon) \to 0.
\]
Combined with \eqref{eq:RajenTrick}, we finally conclude that 
\begin{align*}
    \lim_{\epsilon\to 0}\frac{\mu_1(P_\epsilon)-\mu_1(P_0)}{\epsilon}
    = \lim_{\epsilon\to 0}L_\epsilon
    = \lim_{\epsilon\to 0}R_\epsilon
    = \inner{\psi}{\tau}_{L^2(P_0)}.
\end{align*}
As $\epsilon \mapsto P_\epsilon$ was an arbitrary QMD path, this shows that 
$\mu_1$ is pathwise differentiable with derivative 
$\inner{\psi}{\cdot}_{L^2(P_0)}$ along the maximal tangent set of $\mathcal{P}$. 

We now show that $\psi$ is the efficient influence function, which follows if we
can show that it can be realized as a tangent to a regular path. To this end, 
define first the auxiliary function $\kappa(t) = 2(1+e^{-2t})^{-1}$, which 
satisfies $\kappa(0) = \kappa'(0) = 1$ and $\kappa \leq 2$. Now define the path 
$\epsilon \mapsto P_\epsilon$ with density
\begin{align*}
    \frac{\mathrm{d}P_\epsilon}{\mathrm{d}\nu}(t,\bw,y)
    = s_\epsilon^2(t,\bw,y)
    = f_\epsilon(y \given t,\bw) \pi_0(\bw)^t(1-\pi_0(\bw))^{1-t} \wp_\epsilon(\bw)
\end{align*}
where, with $\bz = \phi(\bw)$,
\begin{align*}
    f_\epsilon(y \given t,\bw)
        \coloneq \frac{\kappa(\epsilon \zeta(t,\bz,y))f_0(y\given t,\bz)}{c_\epsilon^a (t,\bz)}, \quad
    c_\epsilon^a(t,\bz)
        \coloneq
        \int \kappa(\epsilon \zeta(t,\bz,y))f_0(y\given t,\bz) \nu_{\real}(\mathrm{d}y)
\end{align*}
and 
\begin{align*}
    \wp_\epsilon(\bw) &\coloneq \frac{\kappa(\epsilon (Q_0(\bz) - \mu_1(P_0)))
        \wp_0(\bw)}{
        c_\epsilon^b
        },
\end{align*}
where $c_\epsilon^b$ is the appropriate normalization constant. 
We verify that this indeed defines a regular parametric model. 
First note that $P_\epsilon \in \mathcal{P}_\nu$ by construction,
and in fact, $P_\epsilon \in \mathcal{P}_{c,C}^{\textsc{ods}}$ since
\begin{itemize}
    \item $\mathbb{P}_{P_\epsilon}(|Y|<C)=1$ follows from conditioning and using
    that $f_0 = 0 \implies f_\epsilon =0$.

    \item $\bZ$ is $P_\epsilon$-ODS since $f_\epsilon$ depends on $\bw$ only 
    via $\bz=\phi(\bw)$.

    \item The conditional distribution for $T\given \bW$ is the same for 
    $P_\epsilon$ and $P_0$.
\end{itemize}
By dominated convergence, we see that
\begin{align*}
    \frac{\mathrm{d}}{\mathrm{d}\epsilon}c_\epsilon^a(t,\bw)\vert_{\epsilon=0}
    = \ex_{P_0}[\psi_2(T,\bW,Y)\given T=t,\bW=\bw]
    = 0 
\end{align*}
and that
\begin{align*}
    \frac{\mathrm{d}}{\mathrm{d}\epsilon}c_\epsilon^b(t,\bw)\vert_{\epsilon=0}
    = \int (Q_0(\bw)-\mu_1(P))\wp_0(\bw) \nu_{\mathbb{W}}(\mathrm{d}\bw)
    = 0.
\end{align*}
From this it follows that the mapping $\epsilon \mapsto s_\epsilon$ is $C^1$, and that
\begin{align*}
    \frac{\mathrm{d}}{\mathrm{d}\epsilon} \log(s_\epsilon^2)\vert_{\epsilon=0}
    = \frac{\mathrm{d}}{\mathrm{d}\epsilon} \log(f_\epsilon) \vert_{\epsilon=0}
    + \frac{\mathrm{d}}{\mathrm{d}\epsilon} \log(\wp_\epsilon) \vert_{\epsilon=0} 
    = \zeta + Q_0 - \mu_1(P_0) = \tau.
\end{align*}
Finally, since the information is 
$\ex_{P_{\epsilon}}[\psi_\epsilon(T,W,Y)^2] \in (0,C_\psi)$, 
we conclude from Proposition~2.1 in \citet{bickel1993efficient} that 
$\epsilon \mapsto P_\epsilon$ is a regular path with $\tau$ as its tangent.
\hfill \qedsymbol{}

\subsection{Sketch proof of Theorem~\ref{thm:smoothpaths}} \label{sec:proofA2}
Just as in the proof of Theorem~\ref{thm:dqmpaths}, we assume without loss of 
generality that $\mathbb{T}=\{0,1\}$ and that $t=1$, which can be achieved by 
replacing $T$ with the variable $\one(T=t)$. 
For a generic $\bw\in\mathbb{W}$, we denote its image under $\phi$ with the 
implicit notation $\bz =\phi(\bw)$. Fix $P_0 \in \mathcal{P}$ and consider an 
arbitrary regular curve 
$$
    (-1,1) \ni \epsilon \mapsto P_\epsilon \in \mathcal{P}^\textsc{ods},
$$ 
and the corresponding curve of densities 
$p_\epsilon = \frac{\mathrm{d}P_\epsilon}{\mathrm{d}\nu}$. Since $\bZ=\phi(\bW)$
is $P_{\epsilon}\,$-ODS, the conditional density of $Y\given T=t,\bW=\bw$ is 
equal to the conditional density of $Y\given T=t,\bZ=\bz$ under $P_\epsilon$.
Hence the density can be written as
\begin{align}\label{eq:genericdensity}
    p_\epsilon(t,\bw,y) = 
    \big(f_\epsilon(y\given 1,\bz) \pi_\epsilon(\bw)\big)^t
    \big(f_\epsilon(y\given 0,\bz) (1-\pi_\epsilon(\bw))\big)^{1-t} 
    \wp_\epsilon(\bw)
\end{align}
where $f_\epsilon(\cdot\given t,\bz)$ denotes the conditional density of 
$Y\given T=t,\bZ=\bz$ under $P_\epsilon$, 
$\pi_\epsilon(\bw)=\mathbb{P}_{P_\epsilon}(T=1 \given \bW=\bw)$ 
denotes the conditional probability, and $\wp_\epsilon(\bw)$ the density of $\bW$ under $P_\epsilon$.

On $\{(t,\bw,y)\in \mathbb{T}\times \mathbb{W} \times \mathbb{R} 
\given p_0(t,\bw,y)>0\}$, the score of the generic curve 
\eqref{eq:genericdensity} at $P_0$ is given by
\begin{align}
    s(t,\bw,y) 
    &\coloneq \frac{\mathrm{d}}{\mathrm{d}\epsilon} \log p_\epsilon(t,\bw,y) \vert_{\epsilon=0} \nonumber \\
    &= t s^1(y\given \bz) + (1-t)s^0 (y\given \bz) 
        + \frac{(t-\pi_0(\bw))\dot{\pi}_0(\bw)}{\pi_0(\bw)(1-\pi_0(\bw))} 
        + \frac{\dot{\wp}_0(\bw)}{\wp_0(\bw)} \label{eq:genericscore}
\end{align}
where $s^t(y\given \bz)= \frac{\mathrm{d}}{\mathrm{d}\epsilon} 
\log f_\epsilon(y \given t,\bz) \vert_{\epsilon=0}$ for $t\in \{0,1\}$,
$\dot{\pi}_0(\bw) = \frac{\mathrm{d}}{\mathrm{d}\epsilon} \pi_\epsilon(\bw) \vert_{\epsilon=0}$, 
and $\dot{\wp}_0(\bw) = \frac{\mathrm{d}}{\mathrm{d}\epsilon} \wp_\epsilon(\bw) \vert_{\epsilon=0}$.

Let
$$
    \mathcal{U}\subset 
    L^1(\nu_{\real}\otimes \phi(\nu_{\mathbb{W}})) \times 
    L^1(\nu_{\real}\otimes \phi(\nu_{\mathbb{W}})) \times 
    L^1(\phi(\nu_{\mathbb{W}})) \times
    L^1(\phi(\nu_{\mathbb{W}})) 
$$
be the subset defined by the restriction that for each 
$(u_1, u_2, u_3, u_4)\in \mathcal{U}$ it holds that
$$
    \int u_1(y\given \bz) f_0(y\given 1,\bz)  \mathrm{d}\nu_{\real}(y)
    = \int u_2(y\given \bz) f_0(y\given 0,\bz) \mathrm{d}\nu_{\real}(y)
    =0
$$
for $\phi(\nu_{\mathbb{W}})$-almost all $\bz$, and 
$\int u_4(\bw) \wp_0(\bw) \mathrm{d}\nu_{\mathbb{W}}(\bw)=0$. 

We now argue that the tangent space of $\mathcal{P}^{\textsc{ods}}$ at 
$P_0$ is given by
\begin{align*}
    \dot{\mathcal{P}}^{\textsc{ods}}_{P_0} = \Big\{
        (t,\bw,y) \mapsto 
        t \cdot u_1(y\given \bz)
        + (1-t) u_2(y\given \bz)
        + u_3(\bw) (t-\pi_0(\bw)) 
        + u_4(\bw) \\
        \colon (u_1, u_2, u_3, u_4)\in \mathcal{U}
    \Big\}.
\end{align*}

To see that the generic score \eqref{eq:genericscore} satisfies these
restrictions, note that
\begin{align*}
    \int s^1(y\given \bz) f_0(y\given 1,\bz)  \mathrm{d}\nu_{\real}(y)
    = \int \frac{\mathrm{d}}{\mathrm{d}\epsilon} f_\epsilon(y\given 1,\bz)\vert_{\epsilon=0} \mathrm{d}\nu_{\real}(y)
    = 0,
\end{align*}
by interchanging derivative and expectation. The argument that
$\ex_{P_0}[\dot{\wp}(\bW)/\wp(\bW)]=0$ is similar.

Denote the outcome regression under $P_\epsilon$ conditionally on $A=1$ by
$$
    Q_\epsilon^1(w) = \int yf_\epsilon(y\given 1,w) \mathrm{d}y 
$$
and the adjusted mean on the curve $t\mapsto P_\epsilon$ is then
\begin{align*}
    \Psi(P_\epsilon) = \int Q_\epsilon^1(w)\wp_\epsilon(w)\mathrm{d}w 
        = \int \int yf_\epsilon(y\given 1,w) \mathrm{d}\nu(y) \wp_\epsilon(w) \mathrm{d}\nu(w).
\end{align*}
Differentiating we see that
\begin{align*}
    \frac{\mathrm{d}}{\mathrm{d}\epsilon}\Psi(P_\epsilon)\vert_{\epsilon=0}
    &= \int ys^1(y\given w) f_0(y\given 1,w)\wp_0(w) \mathrm{d}\nu(w,y)
        + \int Q_0^1(w) \dot{\wp}(w) \mathrm{d}\nu(w) \\
    &= \ex[\ex[Ys^1(Y\given W) \given A=1,W]] 
        + \ex[Q_0^1(W) \frac{\dot{\wp}(W)}{\wp(W)}]
\end{align*}
Define the function 
\begin{align*}
    \psi(a,w,y) = Q_0^1(w) + \frac{a(y-Q_0^1(w))}{\pi_0(w)} - \int Q_0^1(w) \wp_0(w)\mathrm{d}\nu(w)
\end{align*}
The last term is deterministic, and since $s(A,W,Y)$ is mean zero, it follows 
that
\begin{align*}
    \ex_{P_0}[\psi(A,W,Y)s(A,W,Y)]
    &= \ex_{P_0}[Q_0^1(W)s(A,W,Y)] 
        + \ex_{P_0}\left[\frac{A(Y-Q_0^1(W))}{\pi_0(W)}s(A,W,Y)\right] \\
    &\eqcolon T_1 + T_2
\end{align*}
For $T_1$, we note that 
$\ex[As^1(Y\given W) \given W] = \ex[(1-A)S^0(Y\given W) \given W] =0$, and hence
\begin{align*}
    T_1 = \ex_{P_0}[Q_0^1(W)s(A,W,Y)] 
    &= \ex[Q_0^1(W)\frac{\dot{\wp}(W)}{\wp(W)}]
\end{align*}
For the $T_2$, note that fraction has mean zero conditionally on $(A,W)$, and 
hence we have that
\begin{align*}
    T_2 = \ex\left[\frac{A(Y-Q_0^1(W))}{\pi_0(W)}s(A,W,Y)\right]  
    &= \ex\left[\frac{A(Y-Q_0^1(W))s^1(Y\given W)}{\pi_0(W)}\right] \\
    &= \ex\left[\ex\left[Ys^1(Y\given W) \given A=1,W\right]\right] 
\end{align*}
Combined we obtain that
\begin{align*}
    \ex_{P_0}[\psi(A,W,Y)s(A,W,Y)]
    = T_1 + T_2 =
    \frac{\mathrm{d}}{\mathrm{d}\epsilon}\Psi(P_\epsilon)\vert_{\epsilon=0}
\end{align*}
As the curve $t\mapsto P_\epsilon$ was generic, this shows that $\Psi$ is 
pathwise differentiable with gradient $\langle \psi, \cdot \rangle_{P_0}$ at 
$P_0$.


\section{Auxiliary results and proofs}\label{sup:proofs}
The proposition below relaxes the criteria for being a valid adjustment set 
in a causal graphical model. It follows from the results of 
\citet{perkovic2018complete}, which was pointed out to the authors by 
Leonard Henckel, and is stated here for completeness.
\begin{prop}\label{prop:adjset}
  Let $T,\bZ$, and $Y$ be pairwise disjoint node sets in a DAG 
  $\mathcal{D}=(\mathbf{V},\mathbf{E})$, and let $\mathcal{M}(\mathcal{D})$ 
  denote the collection of continuous distributions that are  Markovian with 
  respect to $\mathcal{D}$ and with $\ex|Y|<\infty$. Then $\bZ$ is an adjustment
   set relative to $(T,Y)$ if and only if for all 
   $P\in \mathcal{M}(\mathcal{D})$ and all $t$
  \begin{align*}
      \ex_P[Y\given \mathrm{do}(T = t)] 
      = \ex_P\left[
          \frac{\one(T=t)Y}{\mathbb{P}_P(T=t\given \mathrm{pa}_{\mathcal{D}}(T))}
          \right]
      = \ex_P[\ex_P[Y\given T=t, \bZ]]
  \end{align*}
\end{prop}

\begin{proof}
The `only if' direction is straightforward if we use an alternative 
characterization of $\bZ$ being an adjustment set: for any density $p$ of a 
distribution $P\in \mathcal{P}$ it holds that
$$
    p(y\given \mathrm{do}(T=t)) 
    = \int p(\mathbf{y}|\bz,t) p(\bz)\mathrm{d}\bz,
$$
where the `do-operator' (and notation) is defined as in, e.g., 
\citet{peters2017elements}. Dominated convergence then yields 
\begin{align*}
    \ex_P[Y\given \mathrm{do}(T=t)] 
    &= \ex_P[\ex_P[Y\given T=t,\bZ]],
\end{align*}
which is equivalent to the `only if' part. On the contrary, assume that 
$\bZ$ is \emph{not} an adjustment set for $(T,Y)$. Then Theorem~56 and the 
proof of Theorem~57 in \citet{perkovic2018complete} imply the existence of 
a Gaussian distribution $\tilde{P}\in \mathcal{P}$ for
which 
\begin{align*}
    \ex_{\tilde{P}}[Y\given \mathrm{do}(T = 1)] \neq \ex_{\tilde{P}}[\ex_{\tilde{P}}[Y\given T=1, \bZ]]
\end{align*}
This implies the other direction.
\end{proof}

The following lemma is a generalization of Lemma 27 in 
\citet{rotnitzky2020efficient}.
\begin{lem}\label{lem:invpropexp}
    For any $\sigma$-algebras $\cZ_1\subseteq \cZ_2 \subseteq \sigma(\bW)$
    it holds, under Assumption~\ref{asm:positivity}, that
    $$
        \ex_P\left[
            \pi_t(\cZ_2; P)^{-1} \given T=t,\cZ_1
        \right]
         = \pi_t(\cZ_1; P)^{-1}
    $$
    for all $P\in\mathcal{P}$.
\end{lem}
\begin{proof}
    Direct computation yields
    \begin{align*}
        &\ex_P\left[
            \pi_t(\cZ_2; P)^{-1} \given T=t,\cZ_1
        \right] \mathbb{P}_P(T=t\given\cZ_1) \\
        &= 
        \ex_P\left[
            \frac{\one(T=t)}{\pi_t(\cZ_2; P)} \given \cZ_1
        \right] 
        =
        \ex_P\left[
            \frac{\ex_P[\one(T=t)\given\cZ_2]}{\pi_t(\cZ_2; P)} \given \cZ_1
        \right]
        = 1.
    \end{align*}
\end{proof}

\subsection{Proof of Lemma \ref{lem:overadj}}
Since $P\in \mathcal{P}$ is fixed, we suppress it from the notation in 
the following computations.

From $Y\ind \cZ_2 \given \cZ_1,T$ and $\cZ_1\subseteq \cZ_2$, it follows 
that $\ex[Y\given\cZ_1,T] = \ex[Y\given\cZ_2,T]$. Hence, for each 
$t\in \mathbb{T}$,
\begin{equation}\label{eq:overadjeq1}
    b_t(\cZ_1) = \ex[Y\given\cZ_1,T=t] 
    = \ex[Y\given\cZ_2,T=t]=b_t(\cZ_2).
\end{equation} 
Therefore $\mu_t(\cZ_1) = \ex[b_t(\cZ_1)] =
\ex[b_t(\cZ_2)] = \mu_t(\cZ_2)$. If $\mathcal{Z}_2$ is 
description of $\bW$, this identity shows that $\cZ_2$ is $\mathcal{P}$-valid 
if and only if $\cZ_1$ is $\mathcal{P}$-valid. 
To compare the asymptotic variances for $\cZ_2$ 
with $\cZ_1$ we use the law of total variance. Note first that
\begin{align*}
    \ex[\psi_{t}(\cZ_2)\given T,Y,\cZ_1] 
    &= \ex\Big[\one(T=t)(Y-b_t(\cZ_2))
        \frac{1}{\pi_t(\cZ_2)} + b_t(\cZ_2)-\mu_t \given T,Y,\cZ_1 
             \Big] \\
    &= \one(T=t)(Y-b_t(\cZ_1))\,
        \ex\Big[\frac{1}{\pi_t(\cZ_2)}\given T=t,Y,\cZ_1 \Big]
            + b_t(\cZ_1)-\mu_t \\
    &= \one(T=t)(Y-b_t(\cZ_1))
        \ex\left[\frac{1}{\pi_t(\cZ_2)}\given T=t,\cZ_1 \right]
            + b_t(\cZ_1)-\mu_t \\
    &= \one(T=t)(Y-b_t(\cZ_1))
        \frac{1}{\pi_t(\cZ_1)} + b_t(\cZ_1)-\mu_t
        = \psi_t(\cZ_1).
\end{align*}
Second equality is due to $\cZ_1$-measurability and \eqref{eq:overadjeq1}, 
whereas the third equality follows from $Y\ind \cZ_2 \given \cZ_1,T$ and the
last equality is due to Lemma~\ref{lem:invpropexp}. On the other hand,
\begin{align*}
    &\phantom{=}\ex\big[\var(\psi_{t}(\cZ_2)\given T,Y,\cZ_1)\big]\\
    &= \ex\Big[\var\Big(\one(T=t)(Y-b_t(\cZ_1))
        \frac{1}{\pi_t(\cZ_2)} + b_t(\cZ_1)-\mu_t\; \big| \; T,Y,\cZ_1 
             \Big)\Big] \\
    &= \ex\Big[\one(T=t)(Y-b_t(\cZ_1))^2\var\Big(
        \frac{1}{\pi_t(\cZ_2)} \; \big| \; T=t,Y,\cZ_1 
             \Big)\Big] \\
    &= \ex\Big[\pi_t(\cZ_1)\var(Y \given T=t,\cZ_1)\var\Big(
        \frac{1}{\pi_t(\cZ_2)} \; \big| \; T=t,\cZ_1 
             \Big)\Big].
\end{align*}
Combined we have that
\begin{align*}
    \bV_t(\cZ_2)
    &= \var(\psi_{t}(\cZ_2))
    = \var(\ex[\psi_{t}(\cZ_2) \given T,Y,\cZ_1]) 
        + \ex[\var[\psi_{t}(\cZ_2) \given T,Y,\cZ_1]] \\
    &= \bV_t(\cZ_1) 
        + \ex\Big[\pi_t(\cZ_1)\var(Y \given T=t,\cZ_1)\var\Big(
        \frac{1}{\pi_t(\cZ_2)}\Big| \; T=t,\cZ_1 
             \Big)\Big].
\end{align*}
To prove the last part of the lemma, we first note that $\psi_t(\cZ_2)$ and 
$\psi_{t'}(\cZ_2)$ are conditionally uncorrelated for $t\neq t'$. This follows 
from \eqref{eq:overadjeq1}, from which the conditional covariance reduces to
\begin{align*}
    \cov(\psi_t(\cZ_2), \psi_{t'}(\cZ_2) \given T,Y,\cZ_1)
    = \one(T=t)\one(T=t')
    \cov(
        *,* \given T,Y,\cZ_1
    )
    = 0.
\end{align*}
Thus, letting $\psi(\cZ_2) = (\psi_t(\cZ_2))_{t\in \mathbb{T}}$ and using the
computation for $\bV_t$, we obtain that:
\begin{align*}
    &\bV_\Delta(\cZ_2) = \var(\mathbf{c}^\top \psi(\cZ_2)) \\
    &= \var(\ex[\mathbf{c}^\top \psi(\cZ_2) \given T,Y,\cZ_1]) 
        +  \ex[\var[\mathbf{c}^\top \psi(\cZ_2) \given T,Y,\cZ_1]] \\
    &= \bV_\Delta(\cZ_1) 
        + \sum_{t\in \mathbb{T}} c_t^2 \, \ex\Big[\pi_t(\cZ_1)\var(Y \given T=t,\cZ_1)\var\Big(
        \frac{1}{\pi_t(\cZ_2)}\Big| \; T=t,\cZ_1 
             \Big)\Big].
\end{align*}
This concludes the proof. \hfill \qedsymbol

\subsection{Proof of Theorem \ref{thm:COSefficiency}}
(i) Assume that $\cZ\subseteq \sigma(\bW)$ is a description of $\bW$ such that
$\mathcal{Q}_P\subseteq \ol{\cZ}^P$. Observe that almost surely,
$$
    \mathbb{P}_P(Y\leq y \given T, \bW) 
    = \sum_{t\in \mathbb{T}} \one(T=t)F(y\mid t,\bW;P)
$$ 
is $\sigma(T,\mathcal{Q}_P)$-measurable and hence also 
$\sigma(T,\ol{\cZ}^P)$-measurable.
It follows that $\mathbb{P}_P(Y\leq y \given T, \bW)$ is also a version of
$\mathbb{P}_P(Y\leq y \given T, \overline{\cZ}^P)$. As a consequence, 
$\mathbb{P}_P(Y\leq y \given T, \bW) = \mathbb{P}_P(Y\leq y \given T, \cZ)$ 
almost surely. From Doob's characterization of conditional independence 
\citep[Theorem 8.9]{kallenberg2021foundations}, we conclude that 
$Y\indP \bW \mid T,\cZ$, or equivalently that $\cZ$ is $P$-ODS. 
The `in particular' follows from setting $\cZ=\mathcal{Q}_P$.

(ii) Assume that $\cZ \subseteq \sigma(\bW)$ is $P$-ODS. Using Doob's 
characterization of conditional independence again, 
$\mathbb{P}_P(Y\leq y \given T, \bW) = \mathbb{P}_P(Y\leq y \given T, \cZ)$
almost surely. Under Assumption~\ref{asm:positivity}, this entails that
$F(y\given T=t,\bW;P) = \mathbb{P}_P(Y\leq y \given T=t, \cZ)$ almost surely, 
and hence $F(y\given T=t,\bW;P)$ must be $\ol{\cZ}^P$-measurable. As the 
generators of $\mathcal{Q}_P$ are $\ol{\cZ}^P$-measurable, we conclude that 
$\mathcal{Q}_P \subseteq \ol{\cZ}^P$.

(iii) Let $\cZ$ be a $P$-ODS description. 
From (i) and (ii) it follows that $\mathcal{Q}_P\subseteq \ol{\cZ}^P$, and since 
$\cZ \subseteq \sigma(\bW)$ it 
also holds that $Y\indP \ol{\cZ}^P \mid T,\mathcal{Q}_P$. 
Since $\bV_\Delta(\cZ; P) = \bV_\Delta(\ol{\cZ}^P; P)$, Lemma~\ref{lem:overadj} 
gives the desired conclusion. 
\hfill \qed{}

\subsection{Proof of Corollary \ref{cor:Qefficiency}}
Theorem~\ref{thm:COSefficiency} (i,ii) implies that
$\cZ$ is $\mathcal{P}$-ODS if and only if~$\ol{\cZ}^P$ contains $\mathcal{Q}_P$
for all $P \in \mathcal{P}$.

If $\mathcal{Q} \subseteq \ol{\cZ}^P$ for all $P \in \mathcal{P}$ then 
$\mathcal{Q}_P \subseteq \ol{\cZ}^P$ for all $P \in \mathcal{P}$ by definition, 
and $\cZ$ is $\mathcal{P}$-ODS. Equation \eqref{eq:COAuniform} follows from the 
same argument as Theorem~\ref{thm:COSefficiency} $(iii)$. That is, 
$\mathcal{Q} \subseteq \ol{\cZ}^P$, and since $\mathcal{Q}$ is $P$-ODS and 
$\cZ \subseteq \sigma(\bW)$ we have $Y \indP \ol{\cZ}^P \mid T,\mathcal{Q}$. 
Lemma~\ref{lem:overadj} can be applied for each $P\in \mathcal{P}$ to obtain the
desired conclusion. 
\hfill \qed{}

\subsection{Proof of Proposition \ref{prop:COMSefficiency}}
Since $\mathcal{R}_P \subseteq \mathcal{Q}_P$ always holds, so it suffices to 
prove that $\mathcal{Q}_P \subseteq \mathcal{R}_P$. 
If $F(y\given t,\bW;P)$ is $\sigma(b_t(\bW;P))$-measurable, then 
$\mathcal{Q}_P \subseteq \mathcal{R}_P$ follows directly from definition. 
If $Y$ is a binary outcome, then the conditional distribution is determined by 
the conditional mean, i.e., $F(y\given t,\bW;P)$ is 
$\sigma(b_t(\bW;P))$-measurable. If $Y = b_T(\bW) + \varepsilon_Y$ with
$\varepsilon_Y \ind T, \bW$, then $F(y\given t,\bw;P) = F_P(y - b_t(\bw))$ where
$F_P$ is the distribution function for the distribution of $\varepsilon_Y$. 
Again, $F(y\given t,\bW;P)$ becomes $\sigma(b_t(\bW;P))$-measurable.

To prove the last part, let $\cZ$ be a $P$-OMS description. Since $b_t(\bW;P) =
b_t(\cZ;P)$ is $\cZ$-measurable for every $t\in \mathbb{T}$, it follows that
$\mathcal{R}_P\subseteq \ol{\cZ}^P$. Hence the argument of
Theorem~\ref{thm:COSefficiency} (iii) establishes
Equation~\eqref{eq:COApointwise} when $\mathcal{Q}_P = \mathcal{R}_P$. \hfill
\qed{}

\subsection{Proof of Lemma~\ref{lem:underadj}}

Since $P\in \mathcal{P}$ is fixed throughout this proof, we suppress it from
notation. From $\cZ_1\subseteq \cZ_2 \subseteq \sigma(\bW)$ and $T\indP \cZ_2
\given \cZ_1$, it follows that $\pi_t(\cZ_1) = \pi_t(\cZ_2)$, and hence 
$$
\mu_t(\cZ_1)
        =\ex[\pi_t(\cZ_1)^{-1}\one(T=t)Y]
        =\ex[\pi_t(\cZ_2)^{-1}\one(T=t)Y]
        = \mu_t(\cZ_2).
$$
This establishes the first part. For the second part, we use that $\pi_t(\cZ_1)
= \pi_t(\cZ_2)$ and $\mu_t(\cZ_1) =
\mu_t(\cZ_2)$ to see that
\begin{align*}
    \psi_t(\cZ_1) - \psi_t(\cZ_2)
        &= b_t(\cZ_1)-b_t(\cZ_2) 
        + \frac{\one(T=t)}{\pi_t(\cZ_2)}(b_t(\cZ_2)-b_t(\cZ_1)) \\
        &= \Big(\frac{\one(T=t)}{\pi_t(\cZ_2)}-1\Big)(b_t(\cZ_2)-b_t(\cZ_1))
        \eqcolon R_t(\cZ_1,\cZ_2) .
\end{align*}
Since $\ex[R_t(\cZ_1,\cZ_2)  \given \cZ_2]=0$ we have that
$\ex[\psi_t(\cZ_2)R_t(\cZ_1,\cZ_2) ]=0$, which means that $\psi_t(\cZ_2)$ and
$R_t(\cZ_1,\cZ_2) $ are uncorrelated. Note that from $\pi_t(\cZ_1) =
\pi_t(\cZ_2)$ it also follows that 
\begin{align*}
    b_t(\cZ_1)
    = \ex\Big[\frac{Y\one(T=t)}{\pi_t(\cZ_1)}\given \cZ_1\Big] 
    = \ex\Big[\ex\Big[\frac{Y\one(T=t)}{\pi_t(\cZ_2)}\given \cZ_2\Big]
        \given \cZ_1\Big]
    = \ex[b_t(\cZ_2) \given \cZ_1]
\end{align*}
This means that $\ex[(b_t(\cZ_2)-b_t(\cZ_1))^2\given \cZ_1]= \var(b_t(\cZ_2)\given \cZ_1)$. 
These observations let us compute that
\begin{align} \label{eq:varRtcomputation}
    \bV_t(\cZ_1)
    &= \var(\psi_t(\cZ_1)) \nonumber\\
    &= \var(\psi_t(\cZ_2)) 
        +\var\left(R_t(\cZ_1,\cZ_2) \right)  \nonumber\\
    &= \var(\psi_t(\cZ_2)) 
        +\ex[\ex[R_t(\cZ_1,\cZ_2)^2\given \cZ_2]]  \nonumber\\
    &= \bV_t(\cZ_2)
        + \ex\Big[
        (b_t(\cZ_2)-b_t(\cZ_1))^2
        \ex\Big[\frac{\one(T=t)^2}{\pi_t(\cZ_1)^2}
            -\frac{2\one(T=t)}{\pi_t(\cZ_1)}+1\given \cZ_2\Big]\Big] \nonumber\\
    &= \bV_t(\cZ_2)
        + \ex\Big[
        \var(b_t(\cZ_2)\given \cZ_1)
        \Big(\frac{1}{\pi_t(\cZ_1)}-1\Big)\Big].
\end{align}
This establishes the formula in the case $\Delta = \mu_t$. 

For general $\Delta$, note first that $\mathbf{R}$ and
$(\psi_t(\cZ_2))_{t\in\mathbb{T}}$ are uncorrelated since $\ex[\mathbf{R}\given
\cZ_2]=0$. Therefore we have
\begin{align*}
    \avar(\cZ_1) 
        &= \var\Big(\sum_{t\in\mathbb{T}}c_t\psi_t(\cZ_1)\Big) \\
        &= \var\Big(\sum_{t\in\mathbb{T}}c_t\psi_t(\cZ_2)\Big) +
        \var\Big(\sum_{t\in\mathbb{T}}c_t R_t(\cZ_1,\cZ_2)\Big) \\
        &= \avar(\cZ_2) + \mathbf{c}^\top \var\Big( \mathbf{R}\Big)\mathbf{c}.
\end{align*}
It now remains to establish the covariance expressions for the last term, since
the expression of $\var\left(R_t(\cZ_1,\cZ_2) \right)$ was found in
\eqref{eq:varRtcomputation}. To this end, note first that for any
$s,t\in\mathbb{T}$ with $s\neq t$,
\begin{align*}
    \ex\left[
        \Big(\frac{\one(T=s)}{\pi_s(\cZ_2)}-1\Big)\Big(\frac{\one(T=t)}{\pi_t(\cZ_2)}-1\Big)
        \given \cZ_2
    \right]
    = -1.
\end{align*}
Thus we finally conclude that
\begin{align*}
    \cov(R_s,R_t)
    &=
        -\ex\left[
            (b_s(\cZ_2)-b_s(\cZ_1))
            (b_t(\cZ_2)-b_t(\cZ_1))
        \right] \\
    &=  -\ex\left[
            \cov(b_s(\cZ_2),b_t(\cZ_2) \given \cZ_1)
        \right],
\end{align*}
as desired.
\hfill \qed{}

\subsection{Computations for Example \ref{ex:symmetric}}\label{sec:symmetriccomputations}
    Using symmetry, we see that
    \begin{align*}
        \pi_1(\bZ) &= \ex[T \mid \bZ ] = \ex[\bW \mid \bZ] \\
        &= 
            (0.5-\bZ)\mathbb{P}(\bW\leq 0.5\given \bZ)
             + (0.5+\bZ)\mathbb{P}(\bW > 0.5\given \bZ)\
        = \frac{1}{2}.
    \end{align*}
    From $T\ind \bZ$ we observe that
    \begin{align*}
        b_t(\bW) &= t + \ex[g(\bZ)\given  T=t,\bW] = t + g(\bZ), \\
        b_t(\bZ) &= t + \ex[g(\bZ)\given  T=t,\bZ] = t + g(\bZ), \\
        b_t(0) &= t + \ex[g(\bZ)\given T=t] = t + \ex[g(\bZ)].
    \end{align*}
    Plugging these expressions into the asymptotic variance yields:
    \begin{align*}
        \bV_t(0)  
        &= \var\left( b_t(0) + \frac{\one(T=t)}{\pi_t(0)}(Y-b_t(0))\right)\\
        &= \pi_t(0)^{-2}\var\left(\one(T=t)(Y-b_t(0))\right)\\
        &= 4\ex[\one(T=t)( g(\bZ) - \ex[ g(\bZ)]+v(\bW)\varepsilon_Y)^2] \\
        &= 4\ex[\bW( g(\bZ) - \ex[ g(\bZ)])^2]
            +2\ex[v(\bW)\varepsilon_Y^2] \\
        &= 2\var(g(\bZ)) + 2 \ex[v(\bW)^2]\ex[\varepsilon_Y^2] 
    \end{align*}
    \begin{align*}
        \bV_t(\bZ) 
        &= \var\left( b_t(\bZ) + \frac{\one(T=t)}{\pi_t(\bZ)}(Y-b_t(\bZ))\right) \\
        &= \var\left( g(\bZ) + 2\one(T=t)v(\bW)\varepsilon_Y\right) \\
        &= \var\left( g(\bZ)\right)  + 2\ex[v(\bW)^2] \ex[\varepsilon_Y^2] 
    \end{align*}
    \begin{align*}
        \bV_t(\bW) 
        &= \var\left( b_t(\bW) + \frac{\one(T=t)}{\pi_t(\bW)}(Y-b_t(\bW))\right) \\
        &= \var\left( g(\bZ) + \frac{\one(T=t)v(\bW)\varepsilon_Y}{\bW}\right) \\
        &= \var\left( g(\bZ)\right)  + \ex\left[\frac{v(\bW)^2}{\bW}\right] \ex[\varepsilon_Y^2].
    \end{align*}
\hfill \qed{}


\section{Asymptotic results}

To describe the asymptotics of the first term of \eqref{eq:biasvardecomp}, we
consider the decomposition
\begin{align}\label{eq:DOPEdecomp}
    \sqrt{|\mathcal{I}_3|}\widehat{\mu}_t^{\textsc{DOPE}}
    &= U_{\widehat{\phi}}^{(n)} + R_1 + R_2 + R_3,
\end{align}
where
\begin{align*}
    U_{\widehat{\phi}}^{(n)} &\coloneq \frac{1}{\sqrt{|\mathcal{I}_3|}}\sum_{i\in \mathcal{I}_3} u_{\widehat{\phi}, i}, 
    &u_{\widehat{\phi}, i} \coloneq g_{\widehat{\phi}}(t,\bW_i) + \frac{\one(T_i=t)(Y_i-g_{\widehat{\phi}}(t,\bW_i))}{m_{\widehat{\phi}}(t\given \bW_i)},\\
    R_1 &\coloneq \frac{1}{\sqrt{|\mathcal{I}_3|}}\sum_{i\in \mathcal{I}_3} r_i^1,
    &r_i^1 \coloneq 
        \left(\widehat{g}_{\widehat{\phi}}(t,\bW_i)-g_{\widehat{\phi}}(t,\bW_i)\right)
        \left(1-\frac{\one(T_i=t)}{m_{\widehat{\phi}}(t\given \bW_i)}\right), \\
    R_2 &\coloneq  \frac{1}{\sqrt{|\mathcal{I}_3|}}\sum_{i\in \mathcal{I}_3} r_i^2,
    &r_i^2 \coloneq  (Y_i-g_{\widehat{\phi}}(t,\bW_i))\left(
            \frac{\one(T_i=t)}{\widehat{m}_{\widehat{\phi}}(t\given \bW_i)}
            -\frac{\one(T_i=t)}{m_{\widehat{\phi}}(t\given \bW_i)}\right), \\
    R_3 &\coloneq  \frac{1}{\sqrt{|\mathcal{I}_3|}}\sum_{i\in \mathcal{I}_3} r_i^3,
    &r_i^3 \coloneq  
        \big(\widehat{g}_{\widehat{\phi}}(t,\bW_i)
            -g_{\widehat{\phi}}(t,\bW_i)\big)\left(
        \frac{\one(T_i=t)}{\widehat{m}_{\widehat{\phi}}(t\given \bW_i)}
            -\frac{\one(T_i=t)}{m_{\widehat{\phi}}(t\given \bW_i)}\right).
\end{align*}
We define the analogous quantity $U_{\phi}^{(n)}$ by replacing $\widehat{\phi}$ with
a fixed $\phi$ in the definition of $U_{\widehat{\phi}}^{(n)}$. In what follows, we
show that the oracle term, $U_{\widehat{\phi}}^{(n)}$, drives the asymptotic limit,
and that the terms $R_1,R_2,R_3$ are remainder terms, subject to the conditions
stated below.

\subsection{Proof of Theorem \ref{thm:conditionalConvMain}}

The theorem follows from the following results:
\begin{thm}\label{thm:conditionalAIPWconv}
    Under Assumptions \ref{asm:samples} and \ref{asm:AIPWconv}, it holds that
    \begin{align*}
        \mathbb{V}_t(\widehat{\bZ})^{-1/2}
        \cdot
        U_{\widehat{\phi}}^{(n)} \xrightarrow{d} \mathrm{N}(0,1),
        \qquad \text{and} \qquad 
        R_i \xrightarrow{P} 0, \quad i=1,2,3,
    \end{align*}
    as $n\to \infty$.
\end{thm}

\begin{thm}\label{thm:varconsistent}
    Under Assumptions \ref{asm:samples} and \ref{asm:AIPWconv}, it holds that
    \begin{align*}
    \widehat{\mathcal{V}}_t
    -\mathbb{V}_t(\widehat{\bZ}) 
    \xrightarrow{P} 0
    \end{align*}
    as $n\to \infty$.
\end{thm}

\subsection*{Proof of Theorem \ref{thm:conditionalAIPWconv}}    
We first prove that $R_i \xrightarrow{P} 0$ for $i=1,2,3$. The third term is
taken care of by combining Assumption \ref{asm:AIPWconv} (i,iv) with
Cauchy-Schwarz:
\begin{align*}
    |R_3| 
    &\leq \sqrt{|\mathcal{I}_3|} 
        \sqrt{\frac{1}{|\mathcal{I}_3|} \sum_{i\in \mathcal{I}_3} 
        \big(\widehat{m}_{\widehat{\phi}}(t\given \bW_i)^{-1} 
            - m_{\widehat{\phi}}(t\given \bW_i)^{-1}\big)^2}
        \sqrt{\mathcal{E}_{2,t}^{(n)}} \\
    &\leq \ell_c \sqrt{|\mathcal{I}_3| 
        \mathcal{E}_{1,t}^{(n)}\mathcal{E}_{2,t}^{(n)}} \xrightarrow{P} 0,
\end{align*}
where $\ell_c>0$ is the Lipschitz constant of $x \mapsto x^{-1}$ 
on $[c,1-c]$. 

For the first term, $R_1$, note that conditionally on $\widehat{\phi}$ and the
observed estimated representations $(\widehat{\bZ}_i)_{i\in\mathcal{I}_3}$, the
summands are conditionally i.i.d. due to Assumption~\ref{asm:samples}. They are
also conditionally mean zero since
\begin{align*}
    \ex[r_i^1 \given (\widehat{\bZ}_j)_{j\in\mathcal{I}_3}, \widehat{\phi}\,]
    &= 
    \left(\widehat{g}_{\widehat{\phi}}(t,\bW_i)-g_{\widehat{\phi}}(t,\bW_i)\right)
        \Big(1-\frac{\ex[\one(T_i=t)\given (\widehat{\bZ}_j)_{j \in\mathcal{I}_3}, \widehat{\phi}\,]}{m_{\widehat{\phi}}(t \given \bW_i)}\Big) \\
    &= 
    \left(\widehat{g}_{\widehat{\phi}}(t,\bW_i)-g_{\widehat{\phi}}(t,\bW_i)\right)
        \Big(1-\frac{\ex[\one(T_i=t)\given\widehat{\bZ}_i, \widehat{\phi}]}{m_{\widehat{\phi}}(t \given \bW_i)}\Big)
     = 0,
\end{align*}
for each $i\in \mathcal{I}_3$.
Using Assumption \ref{asm:AIPWconv} (i), we can bound the conditional variance by
\begin{align*}
    \var\left(r_i^1
        \given
        (\widehat{\bZ}_j)_{j\in\mathcal{I}_3}, \widehat{\phi} \,
        \right) 
    &=\frac{\left(
        \widehat{g}_{\widehat{\phi}}(t,\bW_i)-g_{\widehat{\phi}}(t,\bW_i)\right)^2
        }{m_{\widehat{\phi}}(t \given \bW_i)^2} 
        \var\left(
        \one(T=t)
        \given
        (\widehat{\bZ}_j)_{j\in\mathcal{I}_3}, \widehat{\phi}
        \right) \\
    &=\frac{\left(
        \widehat{g}_{\widehat{\phi}}(t,\bW_i)-g_{\widehat{\phi}}(t,\bW_i)\right)^2
        }{m_{\widehat{\phi}}(t \given \bW_i)} 
        (1-m_{\widehat{\phi}}(t \given \bW_i)) \\
    &\leq \frac{(1-c)}{c}
    \left(\widehat{g}_{\widehat{\phi}}(t,\bW_i)-g_{\widehat{\phi}}(t,\bW_i)\right)^2
\end{align*}
Now it follows from Assumption~\ref{asm:AIPWconv}~(v) and
the conditional Chebyshev inequality that for any $\epsilon > 0$, we have
\begin{align*}
    \mathbb{P}\left(|R_1| >\epsilon\given  (\widehat{\bZ}_j)_{j\in\mathcal{I}_3}, \widehat{\phi}\,\right)
    &\leq \epsilon^{-2}\var(R_1\given  (\widehat{\bZ}_j)_{j\in\mathcal{I}_3}, \widehat{\phi}) 
    \leq \frac{1-c}{\epsilon^2 c} \mathcal{E}_{2,t}^{(n)}
    \xrightarrow{P} 0,
\end{align*}
from which we conclude that $R_1\xrightarrow{P}0$.

The analysis of $R_2$ is similar. Observe that the summands are conditionally
i.i.d. given $
(\widehat{\bZ}_i)_{i\in\mathcal{I}_3},\widehat{\phi}$,
and $(T_i)_{i\in\mathcal{I}_3}$. They are also conditionally
mean zero since
\begin{align*}
  \ex\left[r_i^2
    \given (\widehat{\bZ}_j, T_j)_{j \in \mathcal{I}_3}, \widehat{\phi}
    \right]
   & =
  \left(\ex[Y_i \given (\widehat{\bZ}_j, T_j)_{j \in \mathcal{I}_3}, \widehat{\phi}]-g_{\widehat{\phi}}(t,\bW_i)\right)\left(
  \frac{\one(T_i=t)}{\widehat{m}_{\widehat{\phi}}(\bW_i)}
  -\frac{\one(T_i=t)}{m_{\widehat{\phi}}(\bW_i)}\right) \\
   & = 0,
\end{align*}
where we use that on the event $(T_i = t)$,
$$
\ex[Y_i \given \widehat{\bZ}_i, T_i, \widehat{\phi}\,] 
= \ex[Y_i \given \widehat{\bZ}_i, T_i=t, \widehat{\phi}\,] 
= g_{\widehat{\phi}}(t,\bW_i).
$$ 
To bound the conditional variance, note first that Assumption~\ref{asm:samples}
and Assumption~\ref{asm:AIPWconv}~(ii) imply
\begin{align*}
    \var\big(Y_i\given \widehat{\bZ}_i, T_i, \widehat{\phi}\,\big)
    \leq \ex\big[Y_i^2\given \widehat{\bZ}_i, T_i, \widehat{\phi}\,\big]
    = \ex[\ex\big[Y_i^2\given \bW_i, T_i, \widehat{\phi}\,\big]
        \given \widehat{\bZ}_i, T_i, \widehat{\phi}\,\big]
    \leq C,
\end{align*}
which in return implies that
\begin{align*}
    \var\left(
        r_i^2
        \given (\widehat{\bZ}_j, T_j)_{j \in \mathcal{I}_3}, \widehat{\phi}
    \right) 
    &= \one(T_i=t) \left(
        \frac{1}{\widehat{m}_{\widehat{\phi}}(\bW_i)}
        -\frac{1}{m_{\widehat{\phi}}(\bW_i)}
        \right)^2
        \var\left(Y_i\given \widehat{\bZ}_i, T_i, \widehat{\phi}\,\right) \\
    &\leq C \left(
        \frac{1}{\widehat{m}_{\widehat{\phi}}(\bW_i)}
        -\frac{1}{m_{\widehat{\phi}}(\bW_i)}
        \right)^2.
\end{align*}
It now follows from the conditional Chebyshev's inequality that
\begin{align*}
    \mathbb{P}(|R_2|>\epsilon \given (\widehat{\bZ}_j, T_j)_{j \in \mathcal{I}_3}, \widehat{\phi}) 
        &\leq \epsilon^{-2}\var(R_2 \given (\widehat{\bZ}_j, T_j)_{j \in \mathcal{I}_3}, \widehat{\phi}) \\
        &\leq \frac{C}{|\mathcal{I}_3|}\sum_{i\in\mathcal{I}_3}\left(
            \frac{1}{\widehat{m}_{\widehat{\phi}}(\bW_i)}
            -\frac{1}{m_{\widehat{\phi}}(\bW_i)}
            \right)^2
        \leq C \ell_c^2  \mathcal{E}_{1,t}^{(n)}.
\end{align*}
Assumptions~\ref{asm:AIPWconv}~(iii) implies
convergence to zero in probability, and from this we conclude that also $R_2
\xrightarrow{P}0$. 

We now turn to the discussion of the oracle term $U_{\widehat{\phi}}^{(n)}$.
From Assumption~\ref{asm:samples} it follows that
$(T_i,\bW_i,Y_i)_{i\in\mathcal{I}_3}\ind \widehat{\phi}$, and hence the
conditional distribution $U_{\widehat{\phi}}^{(n)} \given \widehat{\phi}=\phi$
is the same as that of $U_{\phi}^{(n)}$. We show that this distribution is
asymptotically Gaussian uniformly over all $\phi \in \mathcal{F}$. To see this,
definen for each $\phi$ the representation
$\bZ_{\phi} := \phi(\bW)$ and note that, the terms of $U_\phi^{(n)}$ are i.i.d.
with mean $\mu_t(\bZ_{\phi})$ and variance
$\mathbb{V}_t(\bZ_{\phi})$. Repeated applications of
the fact that $|a+b|^{p} \leq 2^{p-1}(|a|^{p} + |b|^{p})$ for any $a, b, \in
\mathbb{R}$ and $p\geq 1$ together with Assumption~\ref{asm:AIPWconv}~(i) yield
that 
\begin{align*}
  |u_{{\phi, i}}|^{2+\delta} &\leq 2^{1+{\delta}}\left(|g_\phi(t,\bW_i)|^{2+\delta} + c^{-(2+\delta)}\left|(Y_i-g_\phi(t,\bW_i))\right|^{2+\delta} \right)\\
   & \leq \left(2^{1+{\delta}} + 2^{2(1+\delta)}c^{-(2+\delta)}\right)|g_\phi(t,\bW_i)|^{2+\delta} + c^{-(2+\delta)}2^{2(1+\delta)} \left|Y_i\right|^{2+\delta}.
\end{align*}
We conclude that $\ex[|u_{{\phi, i}}|^{2+\delta}]$ is uniformly bounded over all
$\phi \in \mathcal{F}$ by conditional Jensen's inequality and Assumption~\ref{asm:AIPWconv}~(iii).

Thus we can apply the
Lindeberg-Feller CLT, as stated in \citet[Lem. 18]{shah2020hardness}, to
conclude that
\begin{equation*}
    S_{\phi}^{(n)} \coloneq \sqrt{|\mathcal{I}_3|}\mathbb{V}_t(\bZ_{\phi})^{-1/2}(U_{\phi}^{(n)} - \mu_t(\bZ_{\phi})) 
        \xrightarrow{d/\mathcal{F}} \mathrm{N}(0,1),
\end{equation*}
where the convergence holds uniformly over $\phi
\in \mathcal{F}$. With
$\Phi(\cdot)$ denoting the CDF of a standard normal and $P_{\widehat{\phi}}$
denoting the distribution of $\widehat{\phi}$, this implies that
\begin{align}\label{eq:uniformimpliescondtional}
    \sup_{t\in\real}|\mathbb{P}(S_{\widehat{\phi}}^{(n)} \leq t) - \Phi(t)|
    &= \sup_{t\in\real} \left\lvert\int (\mathbb{P}(S_{\widehat{\phi}}^{(n)} \leq t \given \widehat{\phi}=\phi)
    -\Phi(t)) \, P_{\widehat{\phi}}(\mathrm{d}\phi)\right\rvert 
    \nonumber \\
    &\leq \sup_{\phi \in \mathcal{F}} \sup_{t\in\real} |\mathbb{P}(S_{\phi}^{(n)}\leq t)-\Phi(t)|
    \longrightarrow 0.
\end{align}
This shows that $S_{\widehat{\phi}}^{(n)} \xrightarrow{d}
\mathrm{N}(0,1)$ as desired. The last part of the theorem is a simple
consequence of Slutsky's theorem. \hfill \qedsymbol{}


\subsection*{Proof of Theorem \ref{thm:varconsistent}} 

For each $i\in \mathcal{I}_3$, let $\widehat{u}_i =
u_{\widehat{\phi},i} + r_i^1 + r_i^2 + r_i^3$ be the
decomposition from \eqref{eq:DOPEdecomp}. For the squared sum, it is immediate
from Theorem \ref{thm:conditionalAIPWconv} that
\begin{align*}
    \left(\frac{1}{|\mathcal{I}_3|}
        \sum_{i\in \mathcal{I}_3}\widehat{u}_{i}\right)^2 
    = \left(\frac{1}{|\mathcal{I}_3|}
        \sum_{i\in \mathcal{I}_3} u_{\widehat{\phi},i}\right)^2 + o_P(1).
\end{align*}
For the sum of squares, we first expand the squares as
\begin{align}\label{eq:uhatexpansion}
    \widehat{u}_i^2 = u_{\widehat{\phi},i}^2 + 2(r_i^1 + r_i^2 + r_i^3) u_{\widehat{\phi},i} + (r_i^1 + r_i^2 + r_i^3)^2.
\end{align}
We show that the last two terms are convergent to zero in probability. For the
cross-term, we note by Cauchy-Schwarz that
\begin{align*}
    \Big|\frac{1}{|\mathcal{I}_3|}
        \sum_{i\in \mathcal{I}_3} (r_i^1 + r_i^2 + r_i^3) u_{\widehat{\phi},i} \Big|
    \leq 
    \left(\frac{1}{|\mathcal{I}_3|}
        \sum_{i\in \mathcal{I}_3} (r_i^1 + r_i^2 + r_i^3)^2\right)^{1/2}
    \left(\frac{1}{|\mathcal{I}_3|}
        \sum_{i\in \mathcal{I}_3} u_{\widehat{\phi},i}^2\right)^{1/2}.
\end{align*}
We show later that the sum $\frac{1}{|\mathcal{I}_3|}\sum_{i\in \mathcal{I}_3}
u_{\widehat{\phi},i}^2$ is convergent in probability. Thus, to show that
the last two terms of \eqref{eq:uhatexpansion} converge to zero, it suffices to
show that $\frac{1}{|\mathcal{I}_3|}\sum_{i\in \mathcal{I}_3} (r_i^1 + r_i^2 +
r_i^3)^2 \xrightarrow{P}0$. To this end, we observe that
\begin{align*}
    \frac{1}{|\mathcal{I}_3|}\sum_{i\in \mathcal{I}_3} (r_i^1 + r_i^2 + r_i^3)^2
    \leq 
    \frac{3}{|\mathcal{I}_3|}\sum_{i\in \mathcal{I}_3} (r_i^1)^2+(r_i^2)^2+(r_i^3)^2
\end{align*}
The last term is handled similarly to $R_3$ in the proof of
Theorem~\ref{thm:conditionalAIPWconv},
\begin{align*}
    \frac{1}{|\mathcal{I}_3|}\sum_{i\in \mathcal{I}_3} (r_i^3)^2
    &= 
        \frac{1}{|\mathcal{I}_3|}\sum_{i\in \mathcal{I}_3}
        \left(\widehat{g}_{\widehat{\phi}}(t,\bW_i)-g_{\widehat{\phi}}(t,\bW_i)\right)^2
        \left(\frac{\one(T_i=t)}{\widehat{m}_{\widehat{\phi}}(t \given \bW_i)}-\frac{\one(T_i=t)}{m_{\widehat{\phi}}(t \given \bW_i)}\right)^2 \\
    &\leq  \ell_c^2 n\mathcal{E}_1^{(n)}\mathcal{E}_{2,t}^{(n)} \xrightarrow{P} 0,
\end{align*}
where we have used the naive inequality $\sum a_i^2b_i^2 \leq (\sum a_i^2)(\sum
b_i^2)$ rather than Cauchy-Schwarz. From the proof of Theorem
\ref{thm:conditionalAIPWconv}, we have that 
\begin{align*}
    \ex[(r_i^1)^2 \given (\widehat{\bZ}_j)_{j\in\mathcal{I}_3}, \widehat{\phi}] = \var[r_i^1 \given (\widehat{\bZ}_j)_{j\in\mathcal{I}_3}, \widehat{\phi}]\leq \frac{1-c}{c}\left(\widehat{g}_{\widehat{\phi}}(t,\bW_i)-g_{\widehat{\phi}}(t,\bW_i)\right)^2,
\end{align*}
and hence the conditional Markov's inequality yields
\begin{align*}
    \mathbb{P}\left(\frac{1}{|\mathcal{I}_3|}\sum_{i\in \mathcal{I}_3}(r_i^1)^2>\epsilon \given (\widehat{\bZ}_j)_{j\in\mathcal{I}_3}, \widehat{\phi}\right) 
    &\leq \epsilon^{-1}\ex\left[\frac{1}{|\mathcal{I}_3|}\sum_{i\in \mathcal{I}_3}(r_i^1)^2 \given (\widehat{\bZ}_j)_{j\in\mathcal{I}_3}, \widehat{\phi}\right] \\
    &= \epsilon^{-1}\frac{1}{n}\sum_{i=1}^n \ex\left[(r_i^1)^2 \given (\widehat{\bZ}_j)_{j\in\mathcal{I}_3}, \widehat{\phi}\right] \\
    &\leq \frac{1-c}{\epsilon c}\mathcal{E}_{2,t}^{(n)} \xrightarrow{P} 0.
\end{align*}
Thus we also conclude that $\frac{1}{|\mathcal{I}_3|}\sum_{i\in
\mathcal{I}_3}(r_i^1)^2 \xrightarrow{P} 0$. Analogously, the final remainder
$\frac{1}{|\mathcal{I}_3|}\sum_{i\in \mathcal{I}_3} (r_i^2)^2$ can be shown to
converge to zero in probability by leveraging the argument for
$R_2\xrightarrow{P}0$ in the proof of Theorem \ref{thm:conditionalAIPWconv}.

Combining the arguments above we conclude that 
\begin{align*}
    \widehat{\mathcal{V}}_t(\widehat{g},\widehat{m}) 
    = \frac{1}{|\mathcal{I}_3|}\sum_{i\in \mathcal{I}_3} u_{\widehat{\phi},i}^2
        - \Big(\frac{1}{|\mathcal{I}_3|}
        \sum_{i\in \mathcal{I}_3}u_{\widehat{\phi},i}\Big)^2
        + o_P(1) 
\end{align*}
As noted in the proof of Theorem~\ref{thm:conditionalAIPWconv}, for each
$\phi \in \mathcal{F}$ the terms $\{u_{\phi,i}\}_{i\in \mathcal{I}_3}$ are i.i.d
with $\ex[|u_{\phi},i|^{2+\delta}]$ uniformly bounded. Hence,
the uniform law of large numbers, as stated in \citet[Lem.
19]{shah2020hardness}, implies that
\begin{align*}
    \frac{1}{|\mathcal{I}_3|}\sum_{i\in \mathcal{I}_3} u_{\phi,i}
    - \ex[u_{\phi,1}] \xrightarrow{P/\mathcal{F}} 0,
    \qquad
    \frac{1}{|\mathcal{I}_3|}\sum_{i\in \mathcal{I}_3} u_{\phi,i}^2 
    - \ex[u_{\phi,1}^2] \xrightarrow{P/\mathcal{F}} 0,
\end{align*}
where the convergence in probability holds uniformly over $\phi \in \mathcal{F}$.
Since $\var(u_{\phi,1})=\var(\psi_t(\bZ_{\phi}))=\mathbb{V}_t(\bZ_{\phi})$,
this lets us conclude that
\begin{align*}
  \widetilde{S}_{\phi}^{(n)} \coloneq
  \frac{1}{|\mathcal{I}_3|}\sum_{i\in \mathcal{I}_3} u_{\phi,i}^2
    - \Big(\frac{1}{|\mathcal{I}_3|}\sum_{i\in \mathcal{I}_3} u_{\phi,i}\Big)^2 
    - \mathbb{V}_t(\bZ_{\phi})
    \xrightarrow{P/\mathcal{F}} 0.
\end{align*}
We now use that convergence in distribution is equivalent to convergence in
probability for deterministic limit variables, see \citet[Cor.
B.4.]{christgau2023nonparametric} for a general uniform version of the
statement. Since Assumption~\ref{asm:samples} implies that the conditional
distribution $\widetilde{S}_{\widehat{\phi}}^{(n)}\given
\widehat{\phi}=\phi$ is the same as
$\widetilde{S}_{\phi}^{(n)}$, the computation in
\eqref{eq:uniformimpliescondtional} with $\Phi(\cdot)=\one(\cdot \geq 0)$ yields
that $\widetilde{S}_{\widehat{\phi}}^{(n)}\xrightarrow{d}0$. This lets us
conclude that
\begin{equation*}
    \widehat{\mathcal{V}}_t(\widehat{g},\widehat{m}) 
    = \widetilde{S}_{\widehat{\phi}}^{(n)} + o_P(1) \xrightarrow{P} 0,
\end{equation*}
which finishes the proof.
\hfill \qedsymbol{}

\subsection{Proof of Theorem \ref{thm:SIregularity}} 
We fix $P\in\mathcal{P}$ and suppress it from notation throughout the proof.
Given $\theta\in \real^d$, the single-index model assumption
\eqref{eq:singleindex} implies that
\begin{align*}
    b_t(\bW^\top \theta) 
    = \ex[Y \given T=t, \bW^\top \theta]
    = \ex[h_t(\bW^\top \theta_P) \given T=t, \bW^\top \theta].
\end{align*}
Now since $h_t \in C^1(\mathbb{R})$, we may write
\begin{align*}
    h_t(\bW^\top \theta) 
    &= h_t(\bW^\top \theta_P) + R(\theta,\theta_P),\\
    R(\theta,\theta_P) 
    &=
    h_t'(\bW^\top \theta_P) \bW^\top (\theta - \theta_P)
    + r(\bW^\top(\theta - \theta_P))\bW^\top (\theta - \theta_P)
\end{align*}
for a continuous function $r\colon \mathbb{R} \to \mathbb{R}$ satisfying that
$\lim_{\epsilon \to 0}r(\epsilon) = 0$. It follows that
\begin{align*}
    \ex[h_t(\bW^\top \theta_P) \given T=t, \bW^\top \theta] 
    = h_t(\bW^\top \theta) 
    - \ex[R(\theta,\theta_P) \given T=t, \bW^\top \theta],
\end{align*}
The theoretical bias induced by adjusting for $\bW^\top\theta$ instead of $\bW$,
or equivalently $\bW^\top \theta_P$, is therefore
\begin{align*}
  \mu_t(\bW^\top \theta;P) - \mu_t(\bW)  
    &= \mu_t(\bW^\top \theta;P) - \mu_t(\bW^\top \theta_P)\\
    &= \ex\left[h_t(\bW^\top \theta) 
        - \ex[R(\theta,\theta_P) \given T=t, \bW^\top \theta]
        - h_t(\bW^\top \theta_P)
        \right] \\
    &= \ex[
        R(\theta,\theta_P)
        -\ex[R(\theta,\theta_P)\given T=t,\bW^\top \theta]] \\
    &= 
    \ex\left[
        R(\theta,\theta_P) \cdot
        \left(
        1
        -
        \frac{\one(T=t)}{\mathbb{P}(T=t\given \bW^\top \theta)}
        \right)
    \right] \\
    &= 
        C(\theta,\theta_P)^\top (\theta - \theta_P),
\end{align*}
where 
\begin{align*}
    C(\theta;\theta_P) = 
    \ex\left[
        \left(
            h_t'(\bW^\top \theta)
            + r(\bW^\top(\theta - \theta_P))
        \right)
        \left(
        1
        -
        \frac{\one(T=t)}{\mathbb{P}(T=t\given \bW^\top \theta)}
        \right)
        \bW
    \right].
\end{align*}
To show that $u$ is differentiable at $\theta = \theta_P$, it therefore suffices
to show that the mapping $\theta \mapsto C(\theta;\theta_P)$ is continuous at
$\theta_P$. 

We first show that $\mathbb{P}(T=t \given \bW^\top \theta)$ is continuous at
$\theta_P$ almost surely, which follows after applying a coordinate change such
that $\theta$ becomes a basis vector. To be more precise, we may choose a
neighborhood $\mathcal{U}\subset \real^d$ of $\theta_P$ and $d-1$ continuous
functions 
\[
    q_2, \ldots, q_d \colon \mathcal{U}\longrightarrow \mathcal{S}^{d-1}
\]
such that $Q(\theta) \coloneq (\lVert\theta \rVert^{-1}\theta,
q_2(\theta),\ldots q_d(\theta))$ is an orthogonal matrix for every $\theta\in
\mathcal{U}$. For example, the first $d$ vectors of the Gram-Schmidt process
applied to $(\theta,\mathbf{e}_1,\ldots, \mathbf{e}_d)$ is continuous and yields
an orthonormal basis for $\theta\in \real^d \setminus
\operatorname{span}(\mathbf{e}_d)$, so this works in case $\theta_P\neq \pm
\mathbf{e}_d$. Let $Z(\theta) = Q(\theta)^\top \bW$ and note that $Z(\theta)_1 =
\lVert\theta \rVert^{-1} \theta^\top \bW =\lVert\theta \rVert^{-1}\bW^\top
\theta$.
Then by iterated expectations,
\begin{align*}
    &\mathbb{P}(T=t \given \bW^\top \theta) 
    = \ex[m(t\given \bW)\given \bW^\top \theta] \\
    &= \ex[m(t\given Q(\theta)Z(\theta))\given Z(\theta)_1] \\
    &= 
    \frac{
        \int_{\real^{d-1}} 
        m(t\given \widetilde{\bW}(\theta,z))
        p_{\bW}(\widetilde{\bW}(\theta,z))
        \mathrm{d}z
    }{
        \int_{\real^{d-1}}
        p_{\bW}(\widetilde{\bW}(\theta,z))
        \mathrm{d}z
    },
    \qquad 
    \widetilde{\bW}(\theta,z) \coloneq 
    Q(\theta)
    \begin{pmatrix}
        Z(\theta)_1 \\ z
    \end{pmatrix}
\end{align*}
Each integrand is bounded and continuous over $\theta\in \mathcal{U}$. From
dominated convergence it follows that $\mathbb{P}(T=t \given \bW^\top \theta)$
is continuous at $\theta_P$ almost surely. It follows that almost surely,
\begin{align*}
    &\left(
            h_t'(\bW^\top \theta)
            + r(\bW^\top(\theta - \theta_P))
    \right)
    \left(
    1
    -
    \frac{\one(T=t)}{\mathbb{P}(T=t\given \bW^\top \theta)}
    \right) 
    \bW \\
    &\longrightarrow
    h_t'(\bW^\top \theta_P)
    \left(
    1
    -
    \frac{\one(T=t)}{\mathbb{P}(T=t\given \bW^\top \theta_P)}
    \right)\bW,
    \qquad \text{for }\theta \to \theta_P. 
\end{align*}
By dominated convergence again we conclude that
\begin{align*}
    C(\theta;\theta_P) \longrightarrow
    \ex\left[
     h_t'(\bW^\top \theta_P)
    \left(
    1
    -
    \frac{\one(T=t)}{\mathbb{P}(T=t\given \bW^\top \theta_P)}
    \right)\bW\right], \qquad \theta \to \theta_P.
\end{align*}
This shows that $\theta \mapsto C(\theta;\theta_P)$ is continuous at $\theta_P$,
and hence we conclude that $u$ is differentiable at $\theta=\theta_P$ with
gradient
\begin{align*}
    \nabla u (\theta_P) &=
    \ex\left[
     h_t'(\bW^\top \theta_P)
    \left(
    1
    -
    \frac{\one(T=t)}{\mathbb{P}(T=t\given \bW^\top \theta_P)}
    \right)\bW\right] \\
    &=
    \ex\left[
     h_t'(\bW^\top \theta_P)
    \left(
    1
    -
    \frac{\mathbb{P}(T=t\given \bW)}{\mathbb{P}(T=t\given \bW^\top \theta_P)}
    \right)\bW\right].
\end{align*}
\hfill \qedsymbol{}

\section{Details of simulation study}\label{sup:simdetails}
\subsection{Details of estimators}
Our experiments were conducted in Python \citep{van1995python}. Linear and
logistic regression was imported from the \texttt{scikit-learn} package
\citep{scikit-learn}, and the neural network for DOPE-IDX was
implemented using \texttt{pytorch} \citep{paszke2019pytorch}.

\begin{figure}
    \centering
    \includegraphics[width=0.6\linewidth]{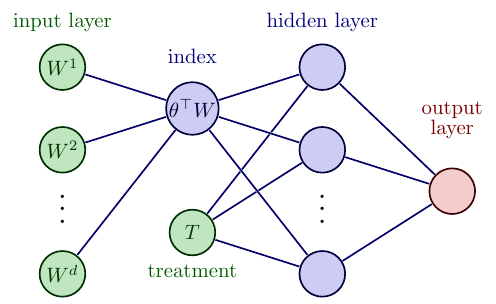}
    \caption{Neural network architecture for single-index model.}
    \label{fig:NNarchitecture}
\end{figure}
The neural network architecture is illustrated in
Figure~\ref{fig:NNarchitecture}. The network was optimized using MSE loss and
the ADAM optimizer with \texttt{lr=1e-3} and \texttt{n\_iter=1200} in the
simulation experiment. For the NHANES application the settings were similar, but
with BCELoss and \texttt{n\_iter=3000}.
The Logistic regression for the propensity
score was fitted without $\ell_2$-penalty and optimized using the \texttt{lbfgs}
optimizer. The propensity score was clipped to the interval $(0.01,0.99)$ for
all estimators of the adjusted mean.

\subsection{Additional results from the simulation study in Section~\ref{sec:simulation}}
\label{sec:extra_simulations}

Figures~\ref{fig:RMSE_stratified_square_sin} and \ref{fig:RMSE_joint_square_sin}
below show the results of the simulation study in Section~\ref{sec:simulation}
with the square and sin link functions with stratified and joint regression,
respectively. The results follow the same patterns described in the main text.

\begin{figure}
  \centering
  \includegraphics[width=\linewidth]{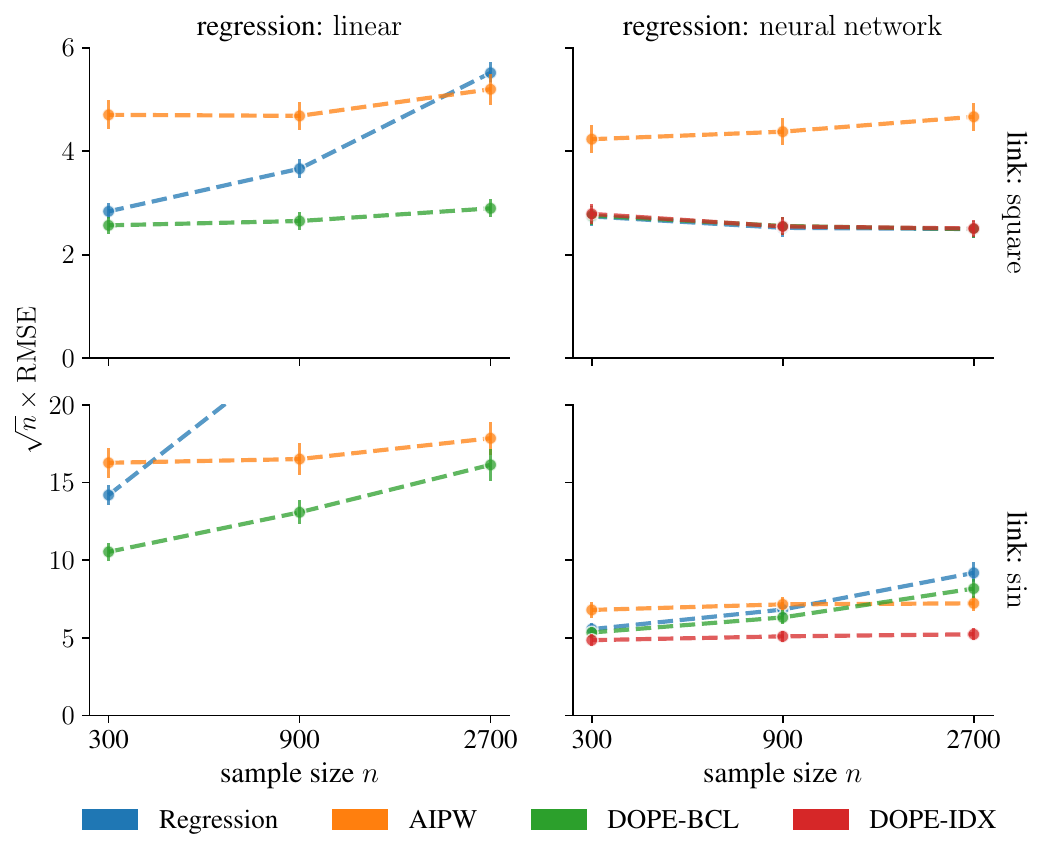}
  \caption{Root mean squared errors (RMSE) for various estimators of $\mu_1$
  plotted against sample size. Each data point is an average over 900 datasets.
  The bars around each point correspond to asymptotic $95\%$ confidence
  intervals based on the CLT. The dashed lines are only included as visual aids
  to make it easier to spot trends across sample sizes. For this plot, the
  outcome regression was fitted separately for each stratum
  $T=0$ and $T=1$.}
  \label{fig:RMSE_stratified_square_sin}
\end{figure}

\begin{figure}
  \centering
  \includegraphics[width=\linewidth]{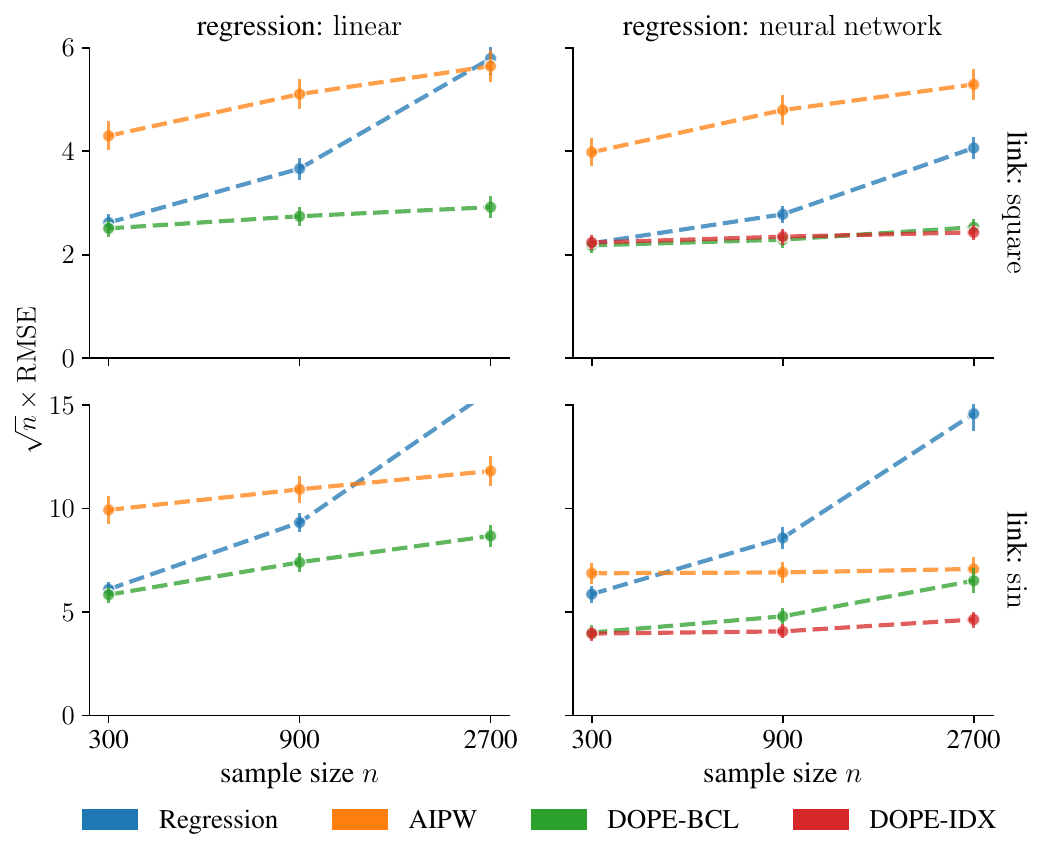}
  \caption{Root mean squared errors (RMSE) for various estimators of $\mu_1$
  plotted against sample size. Each data point is an average over 900 datasets.
  The bars around each point correspond to asymptotic $95\%$ confidence
  intervals based on the CLT. The dashed lines are only included as visual aids
  to make it easier to spot trends across sample sizes. For this plot, the
  outcome regression was fitted jointly onto $(T,\bW)$.}
  \label{fig:RMSE_joint_square_sin}
\end{figure}

Figures~\ref{fig:crossfit_stratified_lin_cbrt} and
\ref{fig:crossfit_joint_lin_cbrt} show the ratio of the RMSE of the
estimators with and without cross-fitting in the settings with linear and cube
root links, for the stratified and joint regressions, respectively.
Figures~\ref{fig:crossfit_stratified_square_sin} and
\ref{fig:crossfit_joint_square_sin} are the corresponding plots for the
square and sin links. We see that generally the cross-fitted estimators perform
worse or no better than the estimators without cross-fitting. 

\begin{figure}
  \centering
  \includegraphics[width=\linewidth]{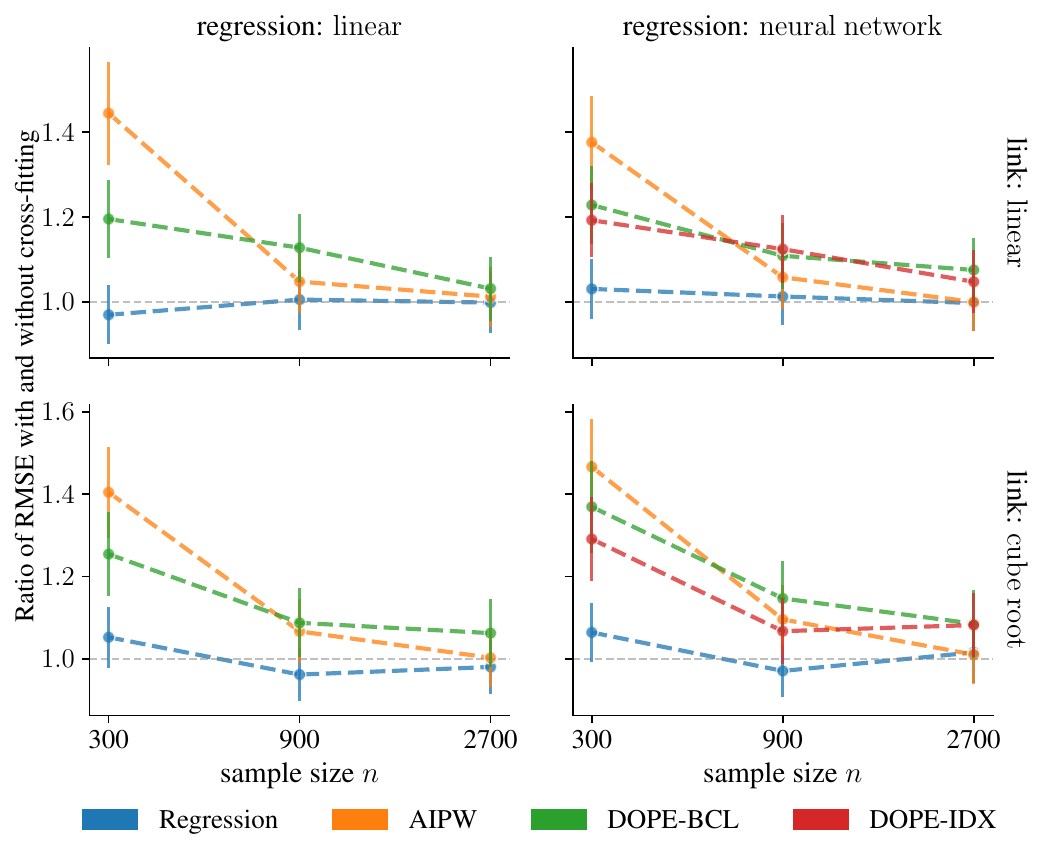}
  \caption{Ratio of root mean squared errors (RMSE) with and without
  cross-fitting for various estimators of $\mu_1$ plotted against sample size.
  Each data point is an average over 900 datasets. The bars around each point
  correspond to asymptotic $95\%$ confidence intervals based on the CLT and the
  Delta method. The dashed lines are only included as visual aids to make it
  easier to spot trends across sample sizes. For this plot, the outcome
  regression was fitted separately for each stratum $T=0$ and $T=1$.}
  \label{fig:crossfit_stratified_lin_cbrt}
\end{figure}

\begin{figure}
  \centering
  \includegraphics[width=\linewidth]{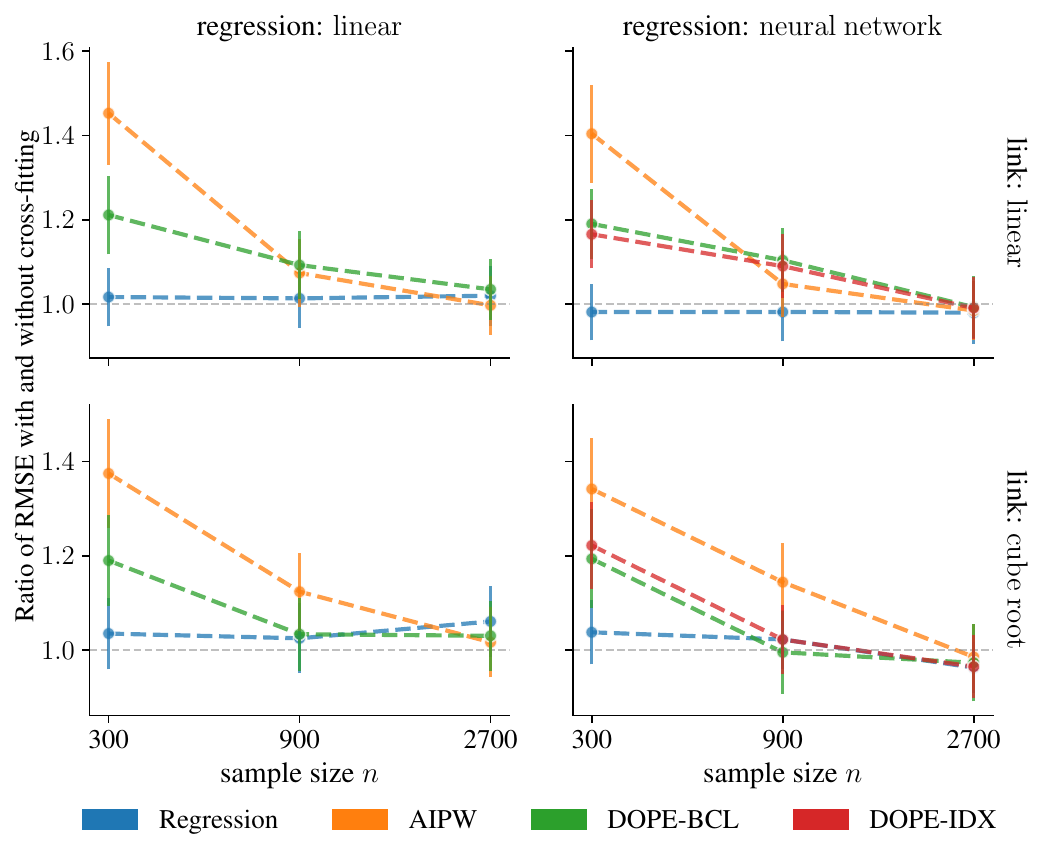}
  \caption{Ratio of root mean squared errors (RMSE) with and without
  cross-fitting for various estimators of $\mu_1$ plotted against sample size.
  Each data point is an average over 900 datasets. The bars around each point
  correspond to asymptotic $95\%$ confidence intervals based on the CLT and the
  Delta method. The dashed lines are only included as visual aids to make it
  easier to spot trends across sample sizes. For this plot, the
  outcome regression was fitted jointly onto $(T,\bW)$.}
  \label{fig:crossfit_joint_lin_cbrt}
\end{figure}

\begin{figure}
  \centering
  \includegraphics[width=\linewidth]{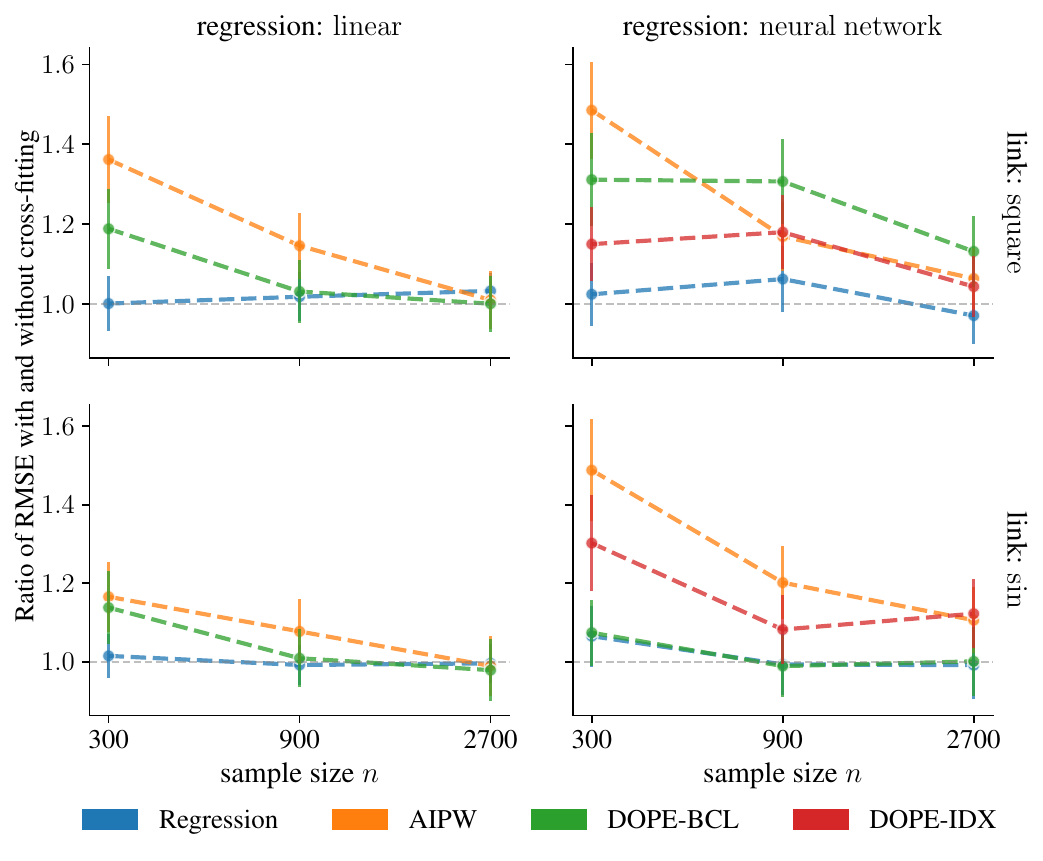}
  \caption{Ratio of root mean squared errors (RMSE) with and without
  cross-fitting for various estimators of $\mu_1$ plotted against sample size.
  Each data point is an average over 900 datasets. The bars around each point
  correspond to asymptotic $95\%$ confidence intervals based on the CLT and the
  Delta method. The dashed lines are only included as visual aids to make it
  easier to spot trends across sample sizes. For this plot, the outcome
  regression was fitted separately for each stratum $T=0$ and $T=1$.}
  \label{fig:crossfit_stratified_square_sin}
\end{figure}

\begin{figure}
  \centering
  \includegraphics[width=\linewidth]{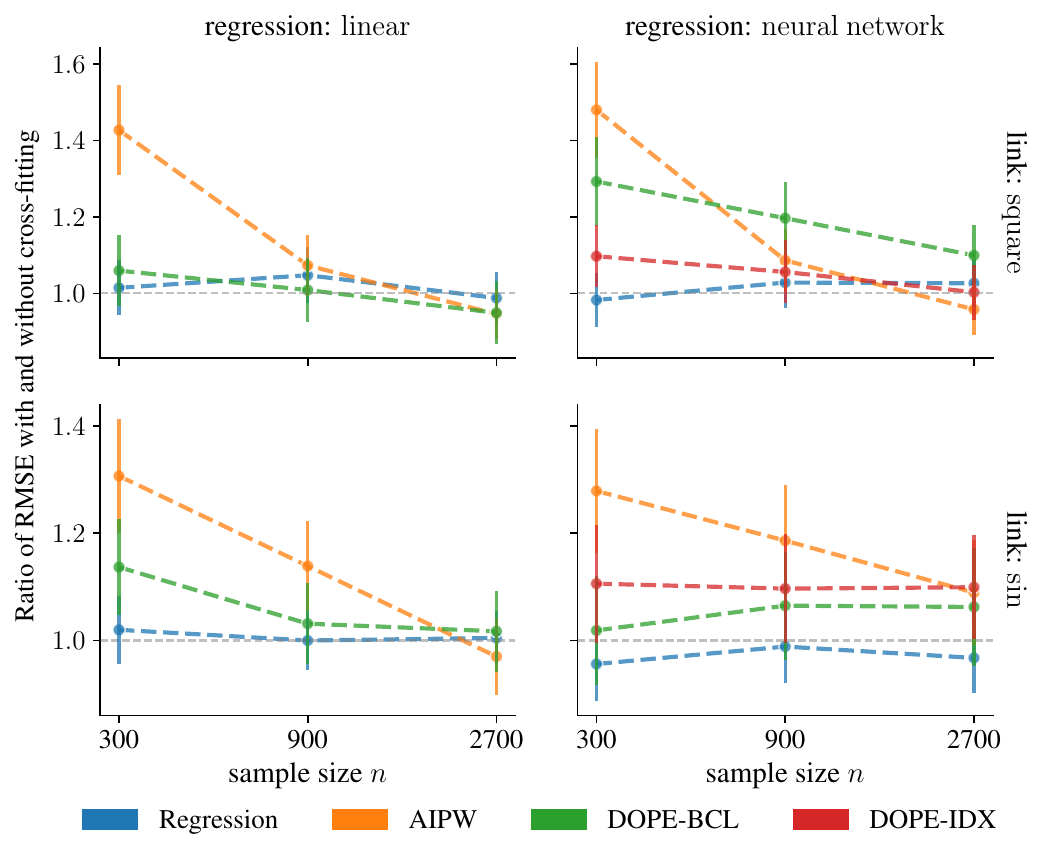}
  \caption{Ratio of root mean squared errors (RMSE) with and without
  cross-fitting for various estimators of $\mu_1$ plotted against sample size.
  Each data point is an average over 900 datasets. The bars around each point
  correspond to asymptotic $95\%$ confidence intervals based on the CLT and the
  Delta method. The dashed lines are only included as visual aids to make it
  easier to spot trends across sample sizes. For this plot, the
  outcome regression was fitted jointly onto $(T,\bW)$.}
  \label{fig:crossfit_joint_square_sin}
\end{figure}

\subsection{Additional results from the NHANES application in
Section~\ref{sec:nhanes}
\label{sup:nhanesdetails}
}

\begin{table}[htb!]
    \centering
    \begin{tabular}{lrrr}
\toprule
Estimator & Estimate & BS se & BS CI \\
\midrule
Regr. (Logistic) & $0.027$ & $0.009$ & $(0.012, 0.048)$ \\
DOPE-BCL (Logistic) & $0.024$ & $0.010$ & $(0.004, 0.040)$ \\
Naive contrast & $0.388$ & $0.010$ & $(0.369, 0.407)$ \\
DOPE-BCL (NN) & $0.010$ & $0.012$ & $(0.000, 0.045)$ \\
Regr. (NN) & $0.022$ & $0.012$ & $(0.003, 0.052)$ \\
DOPE-IDX (NN) & $0.016$ & $0.013$ & $(0.002, 0.050)$ \\
AIPW (Logistic) & $0.022$ & $0.016$ & $(-0.013, 0.051)$ \\
AIPW (NN) & $0.019$ & $0.017$ & $(-0.017, 0.048)$ \\
IPW (Logistic) & $-0.046$ & $0.027$ & $(-0.118, -0.010)$ \\
\bottomrule
\end{tabular}

    \vspace{\baselineskip}
    \caption{Treatment effect estimates for NHANES dataset, where covariates 
    with more than 50\% missing values have been removed. }
    \label{tab:NHANES_pruned_version}
\end{table}

Table~\ref{tab:NHANES_pruned_version} corresponds to
Table~\ref{tab:pulsepressure} in the main manuscript, but where covariates with
more than $50\%$ missing data have been removed rather than imputed. The
results are broadly similar to those in the main manuscript.

\section{Extension to cross-fitting} \label{sup:crossfitting}
\begin{algorithm} \caption{Cross-fitted DOPE} \label{alg:crossfitted}
  \textbf{input}: observations $(T_i,\bW_i,Y_i)_{i\in[n]}$, partition $[n] = J_1\cup \cdots \cup J_K$\;
  \textbf{options}: integer $1\leq m\leq K-2$ and options for Algorithm \ref{alg:generalalg}\;
  \For{$k=1,\ldots,K$}{
    Set $\mathcal{I}_3=J_k$, $\mathcal{I}_1=\bigcup_{l=k+1}^{k+m}J_l$ and 
    $\mathcal{I}_2 = [n]\setminus (\mathcal{I}_1 \cup \mathcal{I}_3)$\;
    
    Compute
    $\widehat{\mu}_{t,k}^{\textsc{DOPE}}$
    as the output of Algorithm~\ref{alg:generalalg}\;
    
    Compute variance estimate $\widehat{\mathcal{V}}_{t,k}$ given as in
    \eqref{eq:varestimator}\; }

  \Return{$\widehat{\mu}_{t}^{\mathsf{x}} \coloneq \frac{1}{K} \sum_{k=1}^K \widehat{\mu}_{t,k}^{\textsc{DOPE}}$ 
  and $\widehat{\mathcal{V}}_{t}^{\mathsf{x}} \coloneq \frac{1}{K} \sum_{k=1}^K\widehat{\mathcal{V}}_{t,k}$}
\end{algorithm}
A cross-fitting procedure for DOPE is described in
Algorithm~\ref{alg:crossfitted}, which computes both a cross-fitted version of
DOPE and its variance estimator. Here the indices of the folds are
understood to cycle modulo $K$ such that $J_{K+1}=J_1$ and so forth. This
version of cross-fitting with three index sets has also been referred to as
`double cross-fitting' by \citet{zivich2021machine}.

We note that the `standard arguments' for establishing convergence of the
cross-fitted estimator cannot be applied directly to our case. This is because,
for each fold $k\in[K]$, the corresponding oracle terms
$U_{\widehat{\phi}_k}^{(n)}$ do not only involve the
data indexed by $\mathcal{I}_3^k$, but also depend on
$\widehat{\phi}_k$ which is estimated from data indexed
by $\mathcal{I}_1^k$. Hence the oracle terms are not independent. However, we
believe that this dependency should be negligible, and perhaps the convergence
can be established under a more refined theoretical analysis.

\end{document}